\documentclass[11pt,leqno]{amsart}
\usepackage{graphicx}
\usepackage{amsmath}
\usepackage{amssymb}
\usepackage{amsthm}
\usepackage{stmaryrd}
\usepackage{mathrsfs}
\usepackage{bm}
\usepackage{color}
\usepackage{enumitem}
\usepackage{verbatim}
\usepackage{mathtools}

\DeclareFontFamily{U}{mathx}{\hyphenchar\font45}
\DeclareFontShape{U}{mathx}{m}{n}{
      <5> <6> <7> <8> <9> <10>
      <10.95> <12> <14.4> <17.28> <20.74> <24.88>
      mathx10
      }{}
\DeclareSymbolFont{mathx}{U}{mathx}{m}{n}
\DeclareFontSubstitution{U}{mathx}{m}{n}
\DeclareMathAccent{\widecheck}{0}{mathx}{"71}

\newcommand{\abs}[1]{\vert#1\vert}

\newcommand{\dtan}{{\mathfrak d}_{\rm tan}}
\newcommand{\dnor}{{\mathfrak d}_{\rm nor}}
\newcommand{\dpar}{{\mathfrak d}_{\parallel}}
\newcommand{\gr}{{\mathtt g}}

\newcommand{\R}{{\mathbb R}}
\newcommand{\N}{{\mathbb N}}
\newcommand{\dt}{\partial_t}

\newcommand{\cI}{{\mathcal I}}
\newcommand{\cE}{{\mathcal E}}

\newcommand{\ul}[1]{\underline{#1}}

\newcommand{\itGamma}{\mathit{\Gamma}}

\newcommand{\cpalpha}{\check{\bm{\partial}}^{\alpha} }
\newcommand{\cmfalpha}{\check{\mathfrak{d}}^{\alpha}}
\newcommand{\cmfpalphaI}{\check{\mathfrak{d}}_{\parallel}^{\alpha_I}}

\makeatletter
\newcommand{\opnorm}{\@ifstar\@opnorms\@opnorm}
\newcommand{\@opnorms}[1]{%
  \left|\mkern-1.5mu\left|\mkern-1.5mu\left|
   #1
  \right|\mkern-1.5mu\right|\mkern-1.5mu\right|
}
\newcommand{\@opnorm}[2][]{%
  \mathopen{#1|\mkern-1.5mu#1|\mkern-1.5mu#1|}
  #2
  \mathclose{#1|\mkern-1.5mu#1|\mkern-1.5mu#1|}
}

\newtheorem{theorem}{Theorem}
\newtheorem{proposition}{Proposition}
\newtheorem{lemma}{Lemma}

\newtheorem{assumption}{Assumption}

\newtheorem{remark}{Remark}

\begin{document}
%----------------------------------------------------------------------------------------------------------------------
%----------------------------------------------------------------------------------------------------------------------
\title[The moving contact line problem for the $2D$ nonlinear shallow water equations]%
{The moving contact line problem for the $2D$ nonlinear shallow water equations with a partially immersed obstacle} 

\author{Tatsuo Iguchi and David Lannes}

\address{Tatsuo Iguchi, Department of Mathematics, Faculty of Science and Technology, Keio University, 
3-14-1 Hiyoshi, Kohoku-ku, Yokohama 223-8522, Japan}

\address{David Lannes, Institut de Math\'ematiques de Bordeaux, Universit\'e de Bordeaux et CNRS UMR 5251, 
351 Cours de la Lib\'eration, 33405 Talence Cedex, France}
\maketitle

\begin{abstract}
We consider the initial value problem for a nonlinear shallow water model in horizontal dimension $d=2$ and 
in the presence of a fixed partially immersed solid body on the water surface. 
We assume that the bottom of the solid body is the graph of a smooth function and part of it is in contact with the water. 
As a result, we have a contact line where the solid body, the water, and the air meet. 
In our setting of the problem, the projection of the contact line on the horizontal plane moves freely due to the motion of the water surface 
even if the solid body is fixed. 
This wave-structure interaction problem reduces to an initial boundary value problem for the nonlinear shallow water equations in an exterior domain 
with a free boundary, which is the projection of the contact line. 
The objective of this paper is to derive a priori energy estimates locally in time for solutions at the quasilinear regularity threshold 
under assumptions that the initial flow is irrotational and subcritical 
and that the initial water surface is transversal to the bottom of the solid body at the contact line. 
The key ingredients of the proof are the weak dissipativity of the system, the introduction of second order Alinhac good unknowns 
associated with a regularizing diffeomorphism, and a new type of hidden boundary regularity for the nonlinear shallow water equations. 
This last point is crucial to control the regularity of the contact line; it is obtained by combining the use of the characteristic fields 
related to the eigenvalues of the boundary matrix together with Rellich type identities. 
\end{abstract}

\noindent
{\bf MSC:} 35L04, 74F10

\noindent
{\bf Keywords:} 
Wave-structure interactions; nonlinear hyperbolic initial boundary value problems; moving contact lines.

%----------------------------------------------------------------------------------------------------------------------
%----------------------------------------------------------------------------------------------------------------------
\section{Introduction}
\subsection{General setting}
Interactions between waves and partially immersed structures are central to study boat navigation and marine renewable energies 
such as floating offshore wind turbines or wave energy convertors. 
Pontoons, offshore platforms, and sea-ice are other examples for which wave-structure interactions are important. 
Engineers  use essentially two kinds of approaches to assess this problem.

The first one consists in describing such interactions with CFD methods. 
Numerical computations in this context turn out to be much more complex than in other fields, such as aeronautics, where the CFD method is standard. 
As explained in \cite{Kim}, this is due to the fact that one has to deal with highly separated flows with high Reynolds number near the hull, 
that there are large scale differences between the hull and mooring systems, 
and that one has to deal with an open ocean environment with non-Gaussian stochastic wave fields. 
For these reasons, the cost of a CFD project has been estimated to be comparable to physical model tests with prototypes. 
CFD is therefore only used for very specific purposes and for projects with a high technology readiness.

The second tool used by engineers to describe wave-structure interactions is much simpler. 
Nonlinearities, viscosity, and vorticity are neglected in the equations, 
while all these effects are taken into account in the CFD approach, which is based on the Navier--Stokes equations. 
The resulting mathematical model was proposed by F. John \cite{John1,John2} and consists in the linear Bernoulli equations cast in the linearized fluid domain; 
we refer to \cite{LannesMing} for a recent mathematical analysis of this problem where the main difficulty comes from the singularities in the fluid domain: 
corners or wedges at the intersection between the surface of the water and the boundary of the object. 
For freely floating objects, Cummins \cite{Cummins} derived a set of six integro-differential equations 
that describe the motion of the object in its six degrees of freedom; 
this equation is used in the commercial softwares for wave structure interactions, such as Wamit. 
As opposed to CFD computations, numerical simulations based on this linear approximation are extremely fast; 
they can for instance be used to model the behavior of an entire wave farm. 
The price to pay is that they miss important physical phenomenons such as the impact of nonlinear effects; 
the predictions based on this approximation are therefore of poor quality in the presence of large amplitude waves, which are common in shallow water.

For these reasons, a third approach was proposed in \cite{Lannes2017}. 
Like John's model, it neglects viscosity and vorticity, but the linear approximation is replaced by an approximation of a different nature. 
It consists in replacing the equations governing the motion of the fluid by a reduced shallow water asymptotic model 
such as the nonlinear shallow water equations or the Boussinesq equations. 
Such approximations no longer depend on the vertical variable and are therefore much easier to compute numerically than the original free surface Euler equations, 
while keeping the nonlinear effets. 
They are commonly used to describe waves in coastal environments where the shallowness assumption under which they are derived is satisfied.

In \cite{Lannes2017}, a method was proposed to extend such shallow water models in the presence of a partially immersed object. 
The idea is to treat the presence of the object as a constraint on the surface elevation of the fluid; 
the pressure exerted by the fluid on the object is then understood as the Lagrange multiplier associated with this constraint; 
see Section \ref{sectpbsetting} below. 
We therefore have to deal with a partially {\it congested} flow; 
such flows also arise in biology \cite{PQV}, collective dynamics \cite{DHN,MRSV}, and other contexts. 
The pressure term is sometimes approximated by a pseudo-compressible relaxation for numerical simulations \cite{GPSW1,GPSW2}; 
this latter approach is often referred to as soft congestion \cite{Perrin}, 
as opposed to hard congestion models where such a relaxation is not performed \cite{ABP,DalibardPerrin1,DalibardPerrin2}.

A general feature for the hard congestion models arising for wave-structure interactions is that 
they can be reformulated as an initial boundary value problem for the wave model under consideration. 
This initial boundary value problem is cast in the so-called exterior region: 
the projection on the horizontal plane of the surface of the water which is contact with the air. 
When the object has vertical sidewalls, this exterior region does not depend on time, 
but it does otherwise because the projection of the contact line varies with time; see Figure \ref{Fig:config}. 
One then has to deal with a free boundary problem.

In horizontal dimension $d=1$ and the case of the nonlinear shallow water equations to describe the waves, 
the problem can be solved when the sidewalls of the object are fixed because the equations are hyperbolic and the boundary conditions are strictly dissipative; 
such a configuration has been considered in \cite{BocchiHeVergara} to model a particular type of wave energy convertor, the oscillating water column, 
and in \cite{SuTucsnak} to simulate a wave generator. 
When the sidewalls are not vertical, 
the problem becomes a free boundary problem and was solved theoretically in \cite{IguchiLannes2021} and numerically in \cite{HaidarMarcheVilar}. 
The difficulty is that the dynamics of the contact points is not kinematic, namely, they do not move at the velocity of the fluid, 
but obeys a more singular evolution equation. 
Still in dimension $d=1$ one can use Boussinesq type models instead of the nonlinear shallow water equations; 
such models are interesting because they do not neglect the dispersive effects. 
For an object with vertical sidewalls, 
the corresponding wave-structure interaction problem has been treated in \cite{BreschLannesMetivier,BeckLannes,BeckLannesWeynans}. 
The difficulty is that contrary to hyperbolic systems, there is no general theory for initial boundary value problems for nonlinear dispersive systems; 
in the case of small dispersion, phenomenons such as dispersive boundary layers must also be taken into account \cite{BreschLannesMetivier}. 
The case of a free boundary, that is, of non-vertical walls is open. 
Let us also mention \cite{MSTT} where a viscous perturbation of the $1d$ nonlinear shallow water equations is considered and \cite{GuoTice,GuoTiceZheng} 
where the dynamics of the contact points in a container with vertical sidewalls is fully solved for the $2D$ Navier--Stokes equations with surface tension.

In horizontal dimension $d=2$, radially symmetric configurations for both the object and the waves have been considered in \cite{Bocchi1,Bocchi2} 
in the case of vertical sidewalls. 
The general non-radial case was then considered for a fixed object with vertical sidewalls in \cite{IguchiLannes2023}. 
Compared with the $1D$ case considered in \cite{IguchiLannes2021}, the difficulty here is that 
the boundary conditions are not strictly dissipative nor can be brought back to that case using Kreiss symmetrizers \cite{Majda,Metivier,BenzoniSerre,Audiard}. 
Also, the boundary is not non-characteristic, 
a situation in which a loss of derivative is likely to occur in the energy estimates \cite{Gues,Secchi,MajdaOsher,CoulombelSecchi}. 
The approach proposed in \cite{IguchiLannes2023} to bypass these obstructions was to introduce a new notion of weak dissipativity 
which ensures the well-posedness of the initial boundary value problem, 
and to show that the wave-structure problem under consideration is weakly dissipative in this sense.

The goal of this paper is to address the free boundary case, that is, 
the case where the sidewalls of the object are not vertical, so that the unknown contact line moves; see Figure \ref{Fig:config}. 
The presence of the free boundary induces several difficulties. 
Firstly, the domain must be fixed by using a diffeomorphism. 
The resulting system is hyperbolic but its coefficients depend on the diffeomorphism, whose regularity is related to the regularity of the free boundary. 
The introduction of the so-called Alinhac good unknowns \cite{Alinhac} yields several cancellations that avoid derivative losses in many free boundary problems 
such as the stability of shocks for compressible gases \cite{Majda,Metivier} or the water waves equations \cite{Lannes2005,Iguchi,AlazardMetivier}. 
In the configuration considered here, this is not enough to avoid derivative losses in the elliptic equation that one has to solve in the interior region; 
we introduce therefore a second order Alinhac good unknown that yields cancellations at subprincipal order. 
In this set of variables the resulting system has the weakly dissipative structure described in \cite{IguchiLannes2023}. 
However, this structure is not sufficient to control the evolution of the free boundary. 
Contrary to related free boundary problems like the shoreline problem \cite{LannesMetivier,Poyferre,MingWang} or 
the vacuum problem for the compressible Euler equations \cite{CoutandShkoller,JangMasmoudi,IfrimTataru}, the dynamic of the contact line is not kinematic. 
This means that it does not move at the normal velocity of the fluid, in particular, it cannot be fixed by working in Lagrangian coordinates. 
The contact line dynamic is also one derivative more singular than the dynamics of shocks for compressible gases, 
inasmuch as it depends on the trace at the boundary of first order derivatives of the solution rather than on the trace of the solution itself. 
It turns out that in order to control the evolution of the free boundary, a control of part of the trace of the solution to the hyperbolic system is needed. 
Such a control would be provided if the boundary conditions where strictly dissipative, but it does not follow from the weakly dissipative structure. 
By a precise analysis of the characteristic fields and using Rellich type identities we however manage to get the desired trace estimate. 
We are then able to obtain a priori estimates at the quasilinear regularity threshold.

%-----------------------------------------------------------
\subsection{Problem setting}\label{sectpbsetting}
We consider waves propagating in shallow water in horizontal dimension $d=2$ and in the presence of a fixed partially immersed solid body on the water surface. 
Denoting by $x=(x_1,x_2)\in\R^2$ and $z\in\R$ the horizontal and vertical coordinates, we assume that the bottom of the water is flat and located at $z=-H_0$, 
that the surface of the water is represented at time $t$ by the graph of a function $Z(t,\cdot)$, and that the bottom of the solid body is represented by the 
graph of a smooth given function $Z_{\rm w}$. 
The horizontal plane $\R^2$ is decomposed at time $t$ as $\R^2=\cI(t)\cup\cE(t)\cup\itGamma(t)$, 
where $\cI(t)$ is the projection of the wetted part of the solid body on the horizontal plane, 
$\cE(t)$ is the projection of the part of the water surface which is in contact with the air, 
and $\itGamma(t)$ is the projection of the contact line, so that $\itGamma(t)=\partial\cI(t)=\partial\cE(t)$. 
We call $\cI(t)$ and $\cE(t)$ the interior and exterior regions, respectively. 
We assume also that $\itGamma(t)$ is a positively oriented Jordan curve. 
The configuration under study is described in Figure \ref{Fig:config}.

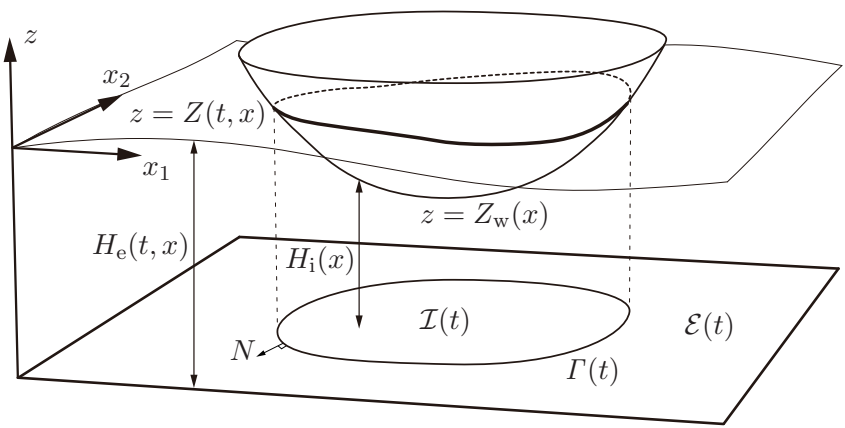
\begin{figure}[ht]
\begin{center}
\includegraphics[width=0.7\linewidth]{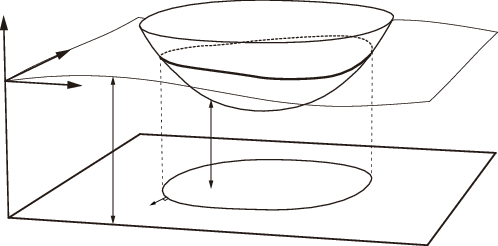}
\end{center}
\setlength{\unitlength}{1pt}
\begin{picture}(0,0)
\put(0,52){$\cI(t)$}
\put(100,50){$\cE(t)$}
\put(55,33){$\itGamma(t)$}
\put(-150,160){$z$}
\put(-105,110){$x_1$}
\put(-120,145){$x_2$}
\put(-124,80){$H_{\rm e}(t,x)$}
\put(-51,75){$H_{\rm i}(x)$}
\put(-110,130){$z=Z(t,x)$}
\put(0,92){$z=Z_{\rm w}(x)$}
\put(-72,40){$N$}
\end{picture}
\vspace{-3ex}
\caption{Waves interacting with a fixed solid body}\label{Fig:config}
\end{figure}

Let $V(t,x)$ be the vertically averaged horizontal velocity and $\ul{P}(t,x)$ denote the pressure on the water surface. 
In this paper, we adopt the shallow water model proposed in \cite{Lannes2017}, so that the motion of the water is governed by the nonlinear shallow water equations 
\begin{equation}\label{NLSWE}
\begin{cases}
 \dt Z+\nabla\cdot(HV) = 0, \\
 \dt V+(V\cdot\nabla)V+\gr\nabla Z = -\frac{1}{\rho}\nabla\ul{P},
\end{cases}
\end{equation}
where $H=H_0+Z$ is the depth of the water, $\gr$ is the acceleration of gravity, and $\rho$ is the constant density of the water per unit volume. 
We assume that the constants $H_0, \gr$, and $\rho$ are positive. 
Throughout the paper, we denote by $(Z_{\rm i},V_{\rm i},\ul{P}_{\rm i})$ the restriction of $(Z,V,\ul{P})$ to the interior region $\cI(t)$, 
while we keep the same notation to denote the restriction of $(Z,V,\ul{P})$ to the exterior region $\cE(t)$. 
This would cause no confusion. 
Similar notation will be used in the following.

In the exterior region, we assume that the pressure is continuous on the water surface so that we impose the constraint $\ul{P}=P_{\rm atm}$ in $\cE(t)$, 
where $P_{\rm atm}$ is the atmospheric pressure assumed to be constant. 
Then, the nonlinear shallow water equations \eqref{NLSWE} reduce to 
\begin{equation}\label{NLSWEe}
\begin{cases}
 \dt Z+\nabla\cdot(HV) = 0 &\mbox{in}\quad \cE(t), \ t>0, \\
 \dt V+(V\cdot\nabla)V+\gr\nabla Z = 0 &\mbox{in}\quad \cE(t), \ t>0.
\end{cases}
\end{equation}
In the interior region, the water is in contact with the solid body so that we must have the constraint $Z_{\rm i}=Z_{\rm w}$ in $\cI(t)$, 
where $Z_{\rm w}$ is independent of time $t$. 
Therefore, the nonlinear shallow water equations \eqref{NLSWE} reduce to 
\begin{equation}\label{NLSWEi}
\begin{cases}
 \nabla\cdot(H_{\rm i}V_{\rm i}) = 0 &\mbox{in}\quad \cI(t), \ t>0, \\
 \dt V_{\rm i}+(V_{\rm i}\cdot\nabla)V_{\rm i}+\gr\nabla Z_{\rm i} = -\frac{1}{\rho}\nabla\ul{P}_{\rm i} &\mbox{in}\quad \cI(t), \ t>0.
\end{cases}
\end{equation}
As matching conditions at the contact line, it is natural to impose the continuity of the unknowns $(Z,V,\ul{P})$ in our setting of the problem so that 
\begin{equation}\label{MC}
Z=Z_{\rm i}, \quad V=V_{\rm i}, \quad \ul{P}_{\rm i}=P_{\rm atm} \quad\mbox{on}\quad \itGamma(t), \ t>0.
\end{equation}
This system of equations \eqref{NLSWEe}--\eqref{MC} is the shallow water model that we are going to investigate in this paper. 
Here, we emphasize that in this problem the projection $\itGamma(t)$ of the contact line is also an unknown quantity so that this is a free boundary problem. 
In this paper we consider the case where the water surface is transversal to the bottom of the solid body. 
This situation can be expressed by 
\begin{equation}\label{TransCond}
|N\cdot\nabla Z-N\cdot\nabla Z_{\rm i}| \geq c_0 \quad\mbox{on}\quad \itGamma(t), \ t>0,
\end{equation}
with a positive constant $c_0$, where $N$ is the unit outward normal vector to $\itGamma(t)$ pointing from $\cI(t)$ to $\cE(t)$. 
Under this transversality condition, the continuity of the water surface at the contact line determines the curve $\itGamma(t)$ from $Z$ and $Z_{\rm i}$. 
Therefore, the matching conditions \eqref{MC} include an equation to determine the unknown curve $\itGamma(t)$.

This wave-structure interaction problem in the case of horizontal dimension $d=1$ has been analyzed in \cite{IguchiLannes2021} and 
the well-posedness of its initial value problem was proved at the quasilinear regularity threshold under assumptions that 
the initial flow is subcritical and that the transversality condition at the contact points is satisfied. 
Our main objective is to extend this result into the case of horizontal dimension $d=2$. 
Toward this goal, in this paper we will derive a priori energy estimates for solutions to \eqref{NLSWEe}--\eqref{MC} under the additional irrotational conditions 
\begin{equation}\label{IrroCond}
\begin{cases}
 \nabla^\perp\cdot V=0 &\mbox{in}\quad \cE(t), \ t>0, \\
 \nabla^\perp\cdot V_{\rm i}=0 &\mbox{in}\quad \cI(t), \ t>0.
\end{cases}
\end{equation}
As we will see in Section \ref{sect:BP}, these irrotational conditions are consistent with the problem \eqref{NLSWEe}--\eqref{MC} in the sense that 
if these conditions are initially satisfied at $t=0$, then regular solutions of \eqref{NLSWEe}--\eqref{MC} satisfy \eqref{IrroCond} for all time $t$ 
as long as they exist.

%-----------------------------------------------------------
\subsection{Coordinate transformation}\label{subsect:CT}
To parameterize the unknown curve $\itGamma(t)$, we will use normal-tangential coordinates related to a reference curve $\ul{\itGamma}$, 
which is a positively oriented Jordan curve of $C^\infty$-class and may be a smooth approximation of the initial curve $\itGamma(0)$. 
We denote by $\ul{\cI}$ the interior domain enclosed by the curve $\ul{\itGamma}$ and put $\ul{\cE}=\R^2\setminus(\ul{\cI}\cup\ul{\itGamma})$, 
which is the exterior domain. 
Suppose that $\ul{\itGamma}$ is parametrized by the arc length $s$ as $x = \ul{x}(s)$ for $0\leq s<L$, where $L$ is the length of the curve $\ul{\itGamma}$. 
As usual, we can regard $\ul{x}(s)$ as a function of $s \in \mathbb{T}_L \simeq \R/(L\mathbb{Z})$. 
The normal-tangential coordinates $(r,s)\in(-r_0,r_0)\times\mathbb{T}_L$ is defined in a tubular neighborhood $U_{\ul{\itGamma}}$ of $\ul{\itGamma}$ 
by the relation $x=\ul{x}(s)+r\ul{n}(s)$, 
where $\ul{n}(s)=-\ul{x}'(s)^\perp$ is the unit outward normal vector to $\ul{\itGamma}$ at the point $\ul{x}(s)$ and pointing from $\ul{\cI}$ to $\ul{\cE}$. 
We denote by $\ul{N}$ the unit outward normal vector to $\ul{\itGamma}$, that is, $\ul{N}(\ul{x}(s))=\ul{n}(s)$. 
We assume that the initial curve $\itGamma(0)$ is a graph in the normal direction, that is, 
it is parameterized in the normal-tangential coordinates as $r=\gamma(0,s)$. 
Then, we see that the unknown curve $\itGamma(t)$ can also be a graph in the normal direction as $r=\gamma(t,s)$ at least locally in time. 
In fact, by the transversality condition \eqref{TransCond} we can apply the implicit function theorem to ensure that the equation 
$Z(t,\ul{x}(s)+r\ul{n}(s)) = Z_{\rm i}(\ul{x}(s)+r\ul{n}(s))$ for $r$ can be solved uniquely in a neighborhood of $(t,r)=(0,\gamma(0,s))$. 
The solution $r$ is nothing but the function $\gamma(t,s)$, which is also an unknown function. 
We will use this parametrization for the unknown curve $\itGamma(t)$ throughout this paper.

As usual in the analysis of free boundary problems, we transform the free boundary problem \eqref{NLSWEe}--\eqref{IrroCond} into a problem 
in time independent domains by using a diffeomorphism $\varphi(t,\cdot):\R^2\to\R^2$, which should have the properties that 
$\varphi(t,\cdot)_{\vert_{\ul{\cE}}} : \ul{\cE}\to \cE(t)$, $\varphi(t,\cdot)_{\vert_{\ul{\cI}}} : \ul{\cI}\to \cI(t)$, 
and $\varphi(t,\cdot)_{\vert_{\ul{\itGamma}}} : \ul{\itGamma}\to \itGamma(t)$ are all diffeomorphisms and that it does not change the orientation. 
Such a diffeomorphism can be constructed from the unknown function $\gamma(t,\cdot)$. 
By using this diffeomorphism, we can transform the free boundary problem \eqref{NLSWEe}--\eqref{IrroCond} into a problem in time independent domains. 
This kind of coordinate transformations in analyses of free boundary problems was first employed in \cite{Hanzawa1981} for the Stefan problem 
and is called the Hanzawa transformation. 
Nowadays, such a transformation is widely used for free boundary problems arising in fluid mechanics.

For a function $F=F(t,x)$ we denote $f=F\circ\varphi$, that is, $f(t,y)=F(t,\varphi(t,y))$. 
We also use the notation 
\begin{equation}\label{DiffOP}
\nabla^\varphi f=(\nabla F)\circ\varphi, \quad \partial_j^\varphi f=(\partial_j F)\circ\varphi,
 \quad \dt^\varphi f=(\dt F)\circ\varphi. 
\end{equation}
We see easily that $\nabla^\varphi$ and $\dt^\varphi$ commute, that is, $\partial_j^\varphi \dt^\varphi = \dt^\varphi \partial_j^\varphi$ for $j=1,2$. 
Then, the free boundary problem \eqref{NLSWEe}--\eqref{IrroCond} is transformed equivalently into 
\begin{equation}\label{nlswe}
\begin{cases}
 \dt^\varphi\zeta+\nabla^\varphi\cdot(hv) = 0 &\mbox{in}\quad (0,T)\times\ul{\cE}, \\
 \dt^\varphi v+(v\cdot\nabla^\varphi)v+\gr\nabla^\varphi \zeta = 0 &\mbox{in}\quad (0,T)\times\ul{\cE}, \\
 (\nabla^\varphi)^\perp\cdot v=0 &\mbox{in}\quad (0,T)\times\ul{\cE}, 
\end{cases}
\end{equation}
\begin{equation}\label{nlswi}
\begin{cases}
 \nabla^\varphi\cdot(h_{\rm i}v_{\rm i}) = 0 &\mbox{in}\quad (0,T)\times\ul{\cI}, \\
 \dt^\varphi v_{\rm i}+(v_{\rm i}\cdot\nabla^\varphi)v_{\rm i}+\gr\nabla^\varphi \zeta_{\rm i}
   = -\frac{1}{\rho}\nabla\ul{p}_{\rm i} &\mbox{in}\quad (0,T)\times\ul{\cI}, \\
 (\nabla^\varphi)^\perp\cdot v_{\rm i}=0 &\mbox{in}\quad (0,T)\times\ul{\cI}, 
\end{cases}
\end{equation}
and 
\begin{equation}\label{mc}
\zeta=\zeta_{\rm i}, \quad v=v_{\rm i}, \quad \ul{p}_{\rm i}=P_{\rm atm} \quad\mbox{on}\quad (0,T)\times\ul{\itGamma},
\end{equation}
where $\zeta=Z\circ\varphi$, $\zeta_{\rm i}=Z_{\rm i}\circ\varphi$, $\zeta_{\rm w}=Z_{\rm w}\circ\varphi$, and so on, 
and the constraint $\zeta_{\rm i}=\zeta_{\rm w}$ is also supposed. 
Although there are infinitely many possible choices for the diffeomorphism $\varphi(t,\cdot)$, we will fix the choice; 
in particular, the diffeomorphism with which we work is regularizing to avoid the loss of half a derivative when taking its trace at the boundary. 
The precise construction of this diffeomorphism from the function $\gamma(t,\cdot)$ will be given in Section \ref{subsect:ChoiceD}. 
For later use, we introduce the $\R^2$-valued functions $w$ and $w_{\rm i}$ defined in $(0,T)\times\ul{\cE}$ and $(0,T)\times\ul{\cI}$, respectively, by 
\begin{equation}\label{defw}
w=v-\dt\varphi, \quad w_{\rm i}=v_{\rm i}-\dt\varphi,
\end{equation}
which represent, roughly speaking, the horizontal velocity of the water related to the motion of the curve $\itGamma(t)$; 
if we work with the diffeomorphism $\varphi$ constructed in Section \ref{subsect:ChoiceD}, 
which coincides with the identity mapping outside a tubular neighborhood $U_{\ul{\itGamma}}$ of $\ul{\itGamma}$, 
then we have also $w=v$ and $w_{\rm i}=v_{\rm i}$ outside $U_{\ul{\itGamma}}$.

%-----------------------------------------------------------
\subsection{Main result}\label{subsect:result}
To state our main result in this paper, we introduce norms of several function spaces. 
We denote by $L^p(\ul{\cE})$, $L^p(\ul{\cI})$, and $L^p(\mathbb{T}_L)$ with $1\leq p\leq \infty$, 
the standard Lebesgue spaces on $\ul{\cE}$, $\ul{\cI}$, and $\mathbb{T}_L$, respectively. 
Double bars are used to denote norms on the two-dimensional domains $\ul{\cE}$ or $\ul{\cI}$ and simple bars on the one-dimensional torus $\mathbb{T}_L$, 
for instance, 
\[
\|u\|_{L^2(\Omega)}=\left(\int_\Omega |u(x)|^2 {\rm d}x \right)^{1/2}, \qquad
|g|_{L^2(\mathbb{T}_L)}=\left(\int_{\mathbb{T}_L} |f(s)|^2 {\rm d}s \right)^{1/2}
\]
with $\Omega=\ul{\cE}$ or $\ul{\cI}$. 
We define $L^p$ Sobolev spaces of order $m\in \N$ as 
\[
W^{m,p}(\Omega)=\{ u\in L^p(\Omega) \,|\, \partial_1^{\alpha_1}\partial_2^{\alpha_2} u \in L^p(\Omega), \alpha_1+\alpha_2\leq m\}
 \qquad (\Omega=\ul{\cE} \mbox{ or } \ul{\cI})
\]
endowed with its canonical norm $\|\cdot\|_{W^{m,p}(\Omega)}$, where $\partial_j=\partial_{x_j}$ $(j=1,2)$. 
We put $H^m(\Omega)=W^{m,2}(\Omega)$. 
We define also the fractional order Sobolev spaces of order $s$ on $\mathbb{T}_L$ endowed with its canonical norm $|\cdot|_{H^s(\mathbb{T}_L)}$, 
which is defined through Fourier series. 
Recalling that $\ul{x}(\cdot):{\mathbb T}_L\to {\mathbb R}^2$ is a smooth parameterization of $\ul{\itGamma}$, 
we classically say that $u\in H^s(\ul{\itGamma})$ if and only if $u\circ\ul{x}\in H^s({\mathbb T}_L)$ with its canonical norm.

For functions $u$ depending also on time $t$, we use the norms 
\[
\|u(t)\|_{m,{\rm e}}=\sum_{j=0}^m \|\dt^j u (t)\|_{H^{m-j}(\ul{\cE})}, \quad
\|u(t)\|_{m,{\rm i}}=\sum_{j=0}^m \|\dt^j u (t)\|_{H^{m-j}(\ul{\cI})},
\]
and 
\[
|u(t)|_m = \sum_{j=0}^m |\dt^j u(t)|_{H^{m-j}(\ul{\itGamma})}.
\]
Using these norms, for a regular solution $u=(\zeta,v^{\rm T})^{\rm T}$, $v_{\rm i}$, $\ul{p}_{\rm i}$, and $\gamma$ for the shallow water model 
\eqref{nlswe}--\eqref{mc}, we define an energy function $E_m(t)$ of order $m\in\mathbb{N}$ by 
\begin{align*}
E_m(t)
&= \sum_{|\alpha|=m} \bigl( \|(\bm{\partial}^\alpha u-(\bm{\partial}^\alpha\varphi\cdot\nabla^\varphi)u)(t)\|_{L^2(\ul{\cE})}
 + \|(\bm{\partial}^\alpha v_{\rm i}-(\bm{\partial}^\alpha\varphi\cdot\nabla^\varphi)v_{\rm i})(t)\|_{L^2(\ul{\cI})} \bigr) \\
&\quad\;
 + \|u(t)\|_{m-1,{\rm e}} + \|v_{\rm i}(t)\|_{m-1,{\rm i}},
\end{align*}
where $\bm{\partial}=(\dt,\partial)=(\dt,\partial_1,\partial_2)$ and $\bm{\partial}^\alpha=\dt^{\alpha_0}\partial_1^{\alpha_1}\partial_2^{\alpha_2}$ 
with a multi-index $\alpha=(\alpha_0,\alpha_1,\alpha_2)$. 
Here, we note that $\bm{\partial}^\alpha u-(\bm{\partial}^\alpha\varphi\cdot\nabla^\varphi)u$ with $|\alpha|=m$ correspond to Alinhac's good unknowns 
for the $m$-th order derivatives of $u$.

As for the given function $Z_{\rm w}$, which represents the bottom of the solid body, we impose the following assumption.

\begin{assumption}\label{ass:BSB}
There exist positive constants $c_0$ and $M_0$, a non-negative integer $m$, and an open set $\ul{\cI}_{0}$ in $\R^2$ such that the followings hold. 
\begin{enumerate}
\item[{\rm (i)}]
$\overline{\cI(0)}\subset\ul{\cI}_{0}$ and $Z_{\rm w} \in C^{m+1}(\overline{\ul{\cI}_0})$. 
\item[{\rm (ii)}]
$H_{\rm w}(x)=H_0+Z_{\rm w}(x) \geq c_0$ for $x\in\ul{\cI}_0$. 
\item[{\rm (iii)}]
$\|Z_{\rm w}\|_{W^{m+1,\infty}(\ul{\cI}_0)} \leq M_0$.
\end{enumerate}
\end{assumption}

The following theorem is our main result in this paper, which gives a priori estimates locally in time for any regular solution to the nonlinear shallow water model 
\eqref{nlswe}--\eqref{mc} at the quasilinear regularity threshold under assumptions that the initial flow is irrotational and subcritical 
and that the initial water surface is transversal to the bottom of the solid body at the contact line. 
We recall that the diffeomorphism $\varphi$ that appears in \eqref{nlswe}--\eqref{mc} is constructed from $\gamma$ as in Section \ref{subsect:ChoiceD} 
and that $r_0$ is the width of the tubular neighborhood $U_{\ul{\itGamma}}$ of $\ul{\itGamma}$ in which normal-tangential coordinates are defined.

\begin{theorem}\label{th:main}
Let $m\geq3$ be an integer and $c_0, M_0, \eta_0, \eta_0^{\rm in}$ positive constants satisfying $\eta_0^{\rm in}<\eta_0<1$. 
Then, there exist a sufficiently small positive time $T$ and a large constant $C$ such that under Assumption \ref{ass:BSB} for any regular solution 
$u=(\zeta,v^{\rm T})^{\rm T},v_{\rm i},\ul{p}_{\rm i}$, and $\gamma$ to \eqref{nlswe}--\eqref{mc},  satisfying initially the conditions 
\begin{equation}\label{EstIni}
 \begin{cases}
  \inf_{x\in\ul{\cE}}(\gr h(0,x)-|w(0,x)|^2) \geq 2c_0, \\
  \inf_{x\in\ul{\itGamma}}| \ul{N}\cdot(\nabla\zeta-\nabla\zeta_{\rm i})(0,x)| \geq 2c_0, \\
  |\gamma(0)|_{L^\infty(\mathbb{T}_L)} \leq \eta_0^{\rm in}r_0, \\
  E_m(0) + |\gamma(0)|_{m-1} \leq M_0,
 \end{cases}
\end{equation}
we have the following a priori bounds 
\begin{equation}\label{APE}
 \begin{cases}
  \inf_{(t,x)\in(0,T)\times\ul{\cE}}(\gr h(t,x)-|w(t,x)|^2) \geq c_0, \\
  \inf_{(t,x)\in(0,T)\times\ul{\itGamma}}| \ul{N}\cdot(\nabla\zeta-\nabla\zeta_{\rm i})(t,x)| \geq c_0, \\
  \sup_{0<t<T}|\gamma(t)|_{L^\infty(\mathbb{T}_L)} \leq \eta_0r_0, \\
  \sup_{0<t<T}(E_m(t) + |\gamma(t)|_{m-1}) + \int_0^T|\gamma(t)|_m^2{\rm d}t \leq C,
 \end{cases}
\end{equation}
where $w$ is defined by \eqref{defw}. 
\end{theorem}

\begin{remark}\label{re:th} \ 
{\bf i.} \ 
The first conditions in \eqref{EstIni} and \eqref{APE} represent that the flow is subcritical relative to the motion of the curve $\itGamma(t)$, 
the second ones represent that the water surface is transversal to the bottom of the solid body at the contact line, 
and the third ones ensure that the curve $\itGamma(t)$ is in the tubular neighborhood $U_{\ul{\itGamma}}$, 
in which the normal-tangential coordinates $(r,s)$ are defined. 

{\bf ii.} \ 
By definition, we have $w(0,x)=v(0,x)-(\dt\varphi)(0,x)$, where $(\dt\varphi)(0,\cdot)$ is determined from $(\dt\gamma)(0,\cdot)$, 
which is the initial velocity of the curve $\itGamma(0)$ and can be written explicitly in terms of the initial data 
$u(0,\cdot)$, $\gamma(0,\cdot)$, and $Z_{\rm w}$. 
For more details, we refer to Remark \ref{re:EEg} in Section \ref{subsect:Eqg}. 

{\bf iii.} \ 
The $m$-th order derivatives of $\gamma$ do not have the continuity in time but have only the square integrability. 
This situation is the same as in the one-dimensional case analyzed in \cite{IguchiLannes2021}. 
\end{remark}

%-----------------------------------------------------------
\subsection{Organization of this article}
In Section \ref{sect:BP}, we present basic properties of the shallow water model \eqref{NLSWEe}--\eqref{MC}. 
We show that the total energy, that is, the sum of the kinetic and potential energies is conserved. 
Then, we analyze the vorticity of the flow and show that if the initial vorticity is identically zero, 
then it continues to zero for all time as long as the regular solution exists. 
We also present two equivalent reformulations of the problem. 
In both reformulations, we eliminate the pressure $\ul{P}_{\rm i}$ from the equations. 
In the first reformulation, we introduce single valued velocity potentials both in the interior and in the exterior regions. 
By using the Dirichlet-to-Neumann map in the interior region, we can reformulate the problem as a free boundary problem for the nonlinear shallow water equations 
in the exterior region with nonlocal boundary conditions. 
These boundary conditions consist of two scalar equations; 
one of them is used to determine the unknown free boundary and the other one is essentially used as the boundary condition to the exterior problem. 
In the second reformulation, contrary to the first one, we use the velocity potential only in the interior region. 
The reformulated problem can be viewed as a transmission problem: 
the nonlinear shallow water equations in the exterior region and the second order elliptic equation in the interior region 
together with two scalar matching conditions: the continuity of the depth of water and of the normal component of the velocity. 
We prove in Proposition \ref{prop:equiv} that the continuity of the tangential component of the velocity across the free boundary is automatically satisfied 
if the velocity is initially continuous on the whole plane. 
In this paper, we follow the spirit of this second reformulation.

In Section \ref{sect:EE}, we derive and study a linearization of the shallow water model \eqref{nlswe}--\eqref{mc}, which is equivalent to 
\eqref{NLSWEe3}--\eqref{BC2} and to \eqref{NLSWEe5}--\eqref{ODE5} as we will see in Section \ref{sect:BP}. 
Both formulations of the problem are cast on time independent domains by using the diffeomorphism $\varphi(t,\cdot)$, 
which includes all the information of the unknown free boundary. 
We first derive the linearized equations. 
Thanks to the introduction of good unknowns, the linearized equations have a nice structure and the equation for the variation of the diffeomorphism is completely 
decoupled from the other equations. 
The linearized problem has a similar structure to the shallow water model in the presence of a fixed partially immersed object with vertical sidewalls 
analyzed in \cite{IguchiLannes2023}. 
In particular, the boundary condition is weakly dissipative so that 
we can derive a basic energy estimate for the linearized problem by modifying slightly the calculations in \cite{IguchiLannes2023}. 
However, in view of the application to the nonlinear problem, we have to pay much more attention to the dependence on the coefficients and the source terms 
than  in \cite{IguchiLannes2023}; 
indeed, these terms involve derivatives of the diffeomorphism $\varphi$ whose regularity is limited and slaved to the regularity of the free boundary.

In Section \ref{sect:AddBR}, we continue to consider the linearized problem derived in Section \ref{sect:EE}. 
Contrary to the case in horizontal dimension $d=1$ analyzed in \cite{IguchiLannes2021}, 
the boundary conditions are not strictly dissipative but only weakly dissipative in the sense \cite{IguchiLannes2023}, 
so that we do not have enough control of the boundary integrals by the general theory. 
Here, we derive an additional boundary regularity estimate for the linearized problem by taking the best advantage of the structure of the boundary conditions. 
More precisely, we evaluate the square integral in space-time on the boundary for the good unknown related to the surface elevation. 
Such an estimate is crucial to control the free boundary and therefore the diffeomorphism $\varphi(t,\cdot)$, 
and to obtain an a priori estimate for the nonlinear problem. 
Since the boundary conditions are weakly dissipative, the boundary integral of some quantity could be controlled, 
but it is not enough to obtain a control of the desired boundary integral. 
To compensate this, we first introduce characteristic fields related to eigenvalues of the boundary matrix, 
derive equations for them, and calculate a corresponding energy function. 
However, the boundary term in the energy estimate does not have good sign as it is. 
We furthermore use Rellich type identities for the solution in the interior domain, which give some relations for the boundary integrals. 
Combining the resulting boundary integrals, we establish the desired additional boundary regularity estimate.

In Section \ref{sect:diffeo}, we first explain the details on the normal-tangential coordinates $(r,s)$ in a tubular neighborhood 
of the reference curve $\ul{\itGamma}$, and then construct the regularizing diffeomorphism $\varphi(t,\cdot)$ from the unknown function $\gamma(t,\cdot)$. 
Here, it is assumed that the unknown curve $\itGamma(t)$, which is the projection of the contact line on the horizontal plane, 
is a graph in the normal direction as $r=\gamma(t,s)$ for $s\in\mathbb{T}_L$. 
We also derive $L^p$-estimates for the derivatives of the diffeomorphism $\varphi$ in terms of norms of $\gamma$.

In Section \ref{sect:GU}, we first introduce good unknowns for the $m$-th order derivatives of the solution. 
We use the normal-tangential coordinates to calculate the derivatives near the boundary $\ul{\itGamma}$, 
while we use the standard coordinates away from the boundary. 
Therefore, we need to introduce good unknowns both near the boundary and away from the boundary by using a partition of unity. 
The good unknowns for the velocities and the surface elevation may be defined in the same way as those of the linearized problem. 
However, if we adopt such a definition to the good unknowns for derivatives of the velocity potential in the interior domain, 
then remainder terms cannot be regarded as lower order terms, so that the a priori estimates for the nonlinear problem would exhibit a derivative loss. 
To bypass  this difficulty, we introduce second order good unknowns that include subprincipal terms to compensate these remaining singular terms. 
We then derive equations for the good unknowns in the exterior and the interior domains near the boundary and away from the boundary, respectively, 
and boundary conditions for them. 
Such equations have essentially the same form as those in the linearized problem considered in Sections \ref{sect:EE} and \ref{sect:AddBR}. 
We also derive equations for the derivatives of $\gamma$, particularly, we obtain an evolution equation for $\gamma$.

In Section \ref{sect:APE}, we prove Theorem \ref{th:main}. 
We first introduce another two energy functions in addition to $E_m(t)$. 
It turns out that these three energy functions are all equivalent. 
We then apply the basic energy estimate and the additional boundary regularity estimate established in Sections \ref{sect:EE} and \ref{sect:AddBR} 
to the equations for the good unknowns derived in Section \ref{sect:GU}. 
Through detailed calculations for the lower order terms and analysis for the transversality condition in the critical case $m=3$, 
we complete the proof of Theorem \ref{th:main}.

\bigskip
\noindent
{\bf Acknowledgement} \\ 
T. I. is partially supported by JSPS KAKENHI Grant Number JP23K22404. 
D. L. is partially supported by the BOURGEONS project, grant ANR-23-CE40-0014-01 of the French National Research Agency (ANR) and 
the project Climath of the PEPR Math-VivEs, ANR-23-EXMA-0003.

%----------------------------------------------------------------------------------------------------------------------
\section{Basic properties of the shallow water model}\label{sect:BP}
We study in this section some properties of the wave-structure interaction model \eqref{nlswe}--\eqref{mc} and 
propose some reformulations that are convenient for the mathematical analysis. 
Conservation of energy is first proved in Section \ref{sectconsNRJ}; 
we then show in Section \ref{sectvort} that the total enstrophy is conserved, 
from which we deduce that the flow remains irrotational if the initial velocity is irrotational. 
Using velocity potentials in the exterior and interior domains, 
a first equivalent formulation of the wave-structure interaction problem is proposed in Section \ref{subsect:EF1}; 
in Section \ref{subsect:EF2} a second reformulation is derived, that does not need the use of a velocity potential in the exterior domain. 
Finally, using the regularizing diffeomorphism, these two formulations are recast in Section \ref{sectfixing} 
as initial boundary value problems on a time independent domain.

%-----------------------------------------------------------
\subsection{Conservation of energy}\label{sectconsNRJ}
As in the case of the shallow water model with a fixed solid body with vertical sidewalls considered in \cite{IguchiLannes2023}, 
the conservation of the total energy is equivalent to the conservation of the mechanical energy, 
which is the sum of the kinetic and the potential energies of the water, denoted by $\mathfrak{E}_\mathrm{fluid}(t)$ at time $t$. 
The mechanical energy $\mathfrak{E}_\mathrm{fluid}(t)$ is the sum of the mechanical energies of the water below the water surface and below the fixed solid body 
\[
\mathfrak{E}_\mathrm{fluid}(t) = \int_{\cE(t)}\mathfrak{e}(t,\cdot) + \int_{\cI(t)}\mathfrak{e}_{\rm i}(t,\cdot),
\]
where the densities of the energies are given by $\mathfrak{e} = \frac12\rho H|V|^2 + \frac12\rho\gr Z^2$ 
and $\mathfrak{e}_{\rm i} = \frac12\rho H_{\rm i}|V_{\rm i}|^2+\frac12\rho\gr Z_{\rm i}^2$. 
Contrary to the case with vertical sidewalls, the interior and exterior regions depend on time $t$. 
Therefore, we use the diffeomorphism $\varphi(t,\cdot)$ introduced in Section \ref{subsect:CT} to calculate the time derivative of the energy. 
Let $Z,V,Z_{\rm i},V_{\rm i},\ul{P}_{\rm i}$, and $\itGamma(t)$ be a regular solution to \eqref{NLSWEe}--\eqref{MC}. 
Then, by the Reynolds transport theorem we have 
\[
\frac{\rm d}{{\rm d}t}\mathfrak{E}_\mathrm{fluid}
 = \int_{\cE(t)}\dt\mathfrak{e} + \int_{\cI(t)}\dt\mathfrak{e}_{\rm i} - \int_{\itGamma(t)}(\mathfrak{e}-\mathfrak{e}_{\rm i})(\dt\varphi\circ\varphi^{-1})\cdot N,
\]
where $(\dt\varphi\circ\varphi^{-1})\cdot N$ represents the deformation speed of the curve $\itGamma(t)$ in the normal direction. 
By the nonlinear shallow water equations \eqref{NLSWEe} and \eqref{NLSWEi}, we have 
\[
\begin{cases}
 \dt\mathfrak{e} + \nabla\cdot\mathfrak{F} = 0 &\mbox{in}\quad \cE(t), \ t>0, \\
 \dt\mathfrak{e}_{\rm i} + \nabla\cdot\mathfrak{F}_{\rm i} = 0 &\mbox{in}\quad \cI(t), \ t>0, 
\end{cases}
\]
where $\mathfrak{F}$ and $\mathfrak{F}_{\rm i}$ are the energy fluxes in the exterior and interior regions, respectively, and given by 
$\mathfrak{F}=\rho(\gr Z+\frac12|V|^2)HV$ and $\mathfrak{F}=(\ul{P}_{\rm i}-P_{\rm atm}+\rho(\gr Z_{\rm i}+\frac12|V_{\rm i}|^2))H_{\rm i}V_{\rm i}$. 
Therefore, we obtain 
\begin{align*}
\frac{\rm d}{{\rm d}t}\mathfrak{E}_\mathrm{fluid}
&= \int_{\itGamma(t)}(\mathfrak{F}-\mathfrak{F}_{\rm i}-(\mathfrak{e}-\mathfrak{e}_{\rm i})(\dt\varphi\circ\varphi^{-1}))\cdot N \\
&= 0,
\end{align*}
where we used the matching conditions \eqref{MC}. 
Therefore, we have the conservation of the total energy for the shallow water model \eqref{NLSWEe}--\eqref{MC}.

%-----------------------------------------------------------
\subsection{Analysis of the vorticity}\label{sectvort}
Let $Z,V,V_{\rm i},\ul{P}_{\rm i}$, and $\itGamma(t)$ be a regular solution to \eqref{NLSWEe}--\eqref{MC} and define the vorticity $\Omega$ 
and $\Omega_\mathrm{i}$ in $\cE(t)$ and $\cI(t)$ by 
\[
\Omega = \nabla^\perp\cdot V, \qquad \Omega_\mathrm{i} = \nabla^\perp\cdot V_\mathrm{i},
\]
respectively. 
We will check that the irrotationality is conserved in the time evolution by \eqref{NLSWEe}--\eqref{MC}, that is, 
if $\Omega(0,\cdot)=0$ in $\cE(0)$ and $\Omega_\mathrm{i}(0,\cdot)=0$ in $\cI(0)$, 
then $\Omega(t,\cdot)=0$ in $\cE(t)$ and $\Omega_\mathrm{i}(t,\cdot)=0$ in $\cI(t)$ for all time $t$, as long as the solution exists. 
We follow the calculation given in \cite[Section 2.4]{IguchiLannes2023} with a slight modification caused by the time evolution of the curve $\itGamma(t)$. 
We use again the diffeomorphism $\varphi(t,\cdot)$ introduced in Section \ref{subsect:CT}. 
It follows from \eqref{NLSWEe} and \eqref{NLSWEi} that 
\[
\begin{cases}
 \dt\bigl(\frac{\Omega^2}{H}\bigr) + \nabla\cdot\bigl(\frac{\Omega^2}{H}V\bigr)=0 &\mbox{in}\quad \cE(t), \ t>0, \\
 \dt\bigl(\frac{\Omega_{\rm i}^2}{H_{\rm i}}\bigr)
  + \nabla\cdot\bigl(\frac{\Omega_{\rm i}^2}{H_{\rm i}}V_{\rm i}\bigr)=0 &\mbox{in}\quad \cI(t), \ t>0.
\end{cases}
\]
Therefore, by the Reynolds transport theorem and the matching conditions \eqref{MC} we have 
\begin{align}\label{vorticity}
\frac{\rm d}{{\rm d}t}\left( \int_{\cE(t)}\frac{\Omega^2}{H} + \int_{\cI(t)}\frac{\Omega_{\rm i}^2}{H_{\rm i}} \right)
&= \int_{\itGamma(t)} \left( \frac{\Omega^2}{H}V - \frac{\Omega_{\rm i}^2}{H_{\rm i}}V_{\rm i}
 - \left( \frac{\Omega^2}{H} - \frac{\Omega_{\rm i}^2}{H_{\rm i}} \right)(\dt\varphi\circ\varphi^{-1}) \right)\cdot N \\
&= \int_{\itGamma(t)} \frac{1}{H}(\Omega+\Omega_{\rm i})(\Omega-\Omega_{\rm i})(V-(\dt\varphi\circ\varphi^{-1}))\cdot N. \nonumber
\end{align}
Differentiating the identity $V(t,\varphi(t,\cdot))=V_{\rm i}(t,\varphi(t,\cdot))$ on $\ul{\itGamma}$ with respect to $t$ and taking the inner product of the 
resulting equations with the tangential vector $N^\perp$, we see that on $\itGamma(t)$ 
\begin{align*}
N^\perp\cdot( \dt V - \dt V_\mathrm{i} )
&= -N^\perp\cdot\bigl( ((\dt\varphi \circ \varphi^{-1})\cdot\nabla)V
    - ((\dt\varphi \circ \varphi^{-1})\cdot\nabla)V_\mathrm{i} \bigr) \\
&= -(\dt\varphi \circ \varphi^{-1})\cdot\bigl(   (N^\perp\cdot\nabla)V - (N^\perp\cdot\nabla)V_\mathrm{i} \bigr)
 - (\Omega - \Omega_\mathrm{i} )(\dt\varphi \circ \varphi^{-1}) \cdot N \\
&= - (\Omega - \Omega_\mathrm{i} )(\dt\varphi \circ \varphi^{-1}) \cdot N,
\end{align*}
where we used the identity 
\begin{equation}\label{F1}
F \cdot (G\cdot\nabla)V = G \cdot (F\cdot\nabla)V + (F\cdot G^\perp)\nabla^\perp \cdot V
\end{equation}
and the matching conditions \eqref{MC}. 
On the other hand, the second equations in \eqref{NLSWEe} and \eqref{NLSWEi} can be written in the form 
\[
\begin{cases}
 \dt V + \Omega V^\perp + \nabla( \frac12|V|^2 + \gr Z ) = 0 & \mbox{in}\quad \cE(t), \ t>0, \\
 \dt V_\mathrm{i} + \Omega_\mathrm{i} V_\mathrm{i}^\perp
  + \nabla( \frac12|V_\mathrm{i}|^2 + \gr Z_\mathrm{i} + \frac{1}{\rho}(\ul{P}_\mathrm{i}-P_\mathrm{atm}) ) = 0 & \mbox{in}\quad \cI(t), \ t>0.
\end{cases}
\]
This together with the matching conditions \eqref{MC} implies $N^\perp\cdot( \dt V - \dt V_\mathrm{i} )=-(\Omega-\Omega_{\rm i})N\cdot V$ on $\itGamma(t)$. 
Therefore, we have 
\[
(\Omega - \Omega_\mathrm{i})(V - \dt\varphi \circ \varphi^{-1})\cdot N = 0 \quad\mbox{on}\quad \itGamma(t), \ t>0.
\]
Plugging this into \eqref{vorticity} we obtain finally that the total enstrophy is conserved, namely, 
\[
\frac{\rm d}{{\rm d}t}\left( \int_{\cE(t)}\frac{\Omega^2}{H} + \int_{\cI(t)}\frac{\Omega_{\rm i}^2}{H_{\rm i}} \right) = 0,
\]
which yields easily the desired conservation of the irrotationality under the positivity of the depth of the water.

Alternatively, we have also 
\[
\frac{\rm d}{{\rm d}t}\left( \int_{\cE(t)}\Omega^2 + \int_{\cI(t)}\Omega_{\rm i}^2\right)
= - \int_{\cE(t)}(\nabla\cdot V)\Omega^2 - \int_{\cI(t)}(\nabla\cdot V_\mathrm{i})\Omega_\mathrm{i}^2,
\]
which together with Gronwall's inequality implies the desired conservation of the irrotationality.

%-----------------------------------------------------------
\subsection{An equivalent formulation I}\label{subsect:EF1}
Let $Z,V,V_{\rm i},\ul{P}_{\rm i}$, and $\itGamma(t)$ be a regular solution to \eqref{NLSWEe}--\eqref{MC} satisfying the irrotational conditions \eqref{IrroCond}. 
Since the interior region $\cI(t)$ is simply connected, the irrotational condition ensures the existence of a single valued potential $\Phi_{\rm i}$ 
of the velocity $V_{\rm i}$, that is, $V_{\rm i}=\nabla\Phi_{\rm i}$. 
Although the exterior domain $\cE(t)$ is not simply connected, by the matching conditions \eqref{MC} we have 
$\int_{\itGamma(t)}N^\perp\cdot V=\int_{\itGamma(t)}N^\perp\cdot\nabla\Phi_{\rm i}=0$, 
so that the irrotational condition ensures also the existence of a single valued potential $\Phi$ of the velocity $V$. 
Then, by adding appropriate functions of time $t$ to $\Phi$ and $\Phi_{\rm i}$, 
the nonlinear shallow water equations \eqref{NLSWEe} and \eqref{NLSWEi} can be transformed equivalently as 
\begin{equation}\label{NLSWEe2}
\begin{cases}
 \dt Z+\nabla\cdot(H\nabla\Phi)=0 &\mbox{in}\quad \cE(t), \ t>0, \\
 \dt\Phi + \frac12|\nabla\Phi|^2 + \gr Z = 0 &\mbox{in}\quad \cE(t), \ t>0, 
\end{cases}
\end{equation}
and 
\begin{equation}\label{NLSWEi2}
\begin{cases}
 \nabla\cdot(H_{\rm i}\nabla\Phi_{\rm i})=0 &\mbox{in}\quad \cI(t), \ t>0, \\
 \dt\Phi_{\rm i} + \frac12|\nabla\Phi_{\rm i}|^2 + \gr Z_{\rm i} = -\frac{1}{\rho}(\ul{P}_{\rm i}-P_{\rm atm}) &\mbox{in}\quad \cI(t), \ t>0.
\end{cases}
\end{equation}
It follows from the second equations in \eqref{NLSWEe2} and \eqref{NLSWEi2} together with the matching conditions \eqref{MC} that 
$\dt\Phi=\dt\Phi_{\rm i}$ on $\itGamma(t)$, so that $\dt(\Phi_{\vert_{\itGamma(t)}})=\dt({ \Phi_{\rm i} }_{\vert_{\itGamma(t)}})$. 
By the matching conditions \eqref{MC}, we have also $N^\perp\cdot\nabla\Phi=N^\perp\cdot\nabla\Phi_{\rm i}$ on $\itGamma(t)$, 
where $N^\perp\cdot\nabla$ is a tangential derivative on the curve $\itGamma(t)$. 
Therefore, by adding an appropriate constant to $\Phi$ or $\Phi_{\rm i}$, we see that $\Phi$ and $\Phi_{\rm i}$ coincide on $\itGamma(t)$, 
and denote by $\Psi_{\rm i}$ their common value; 
\begin{equation}\label{ContPhi}
\Psi_{\rm i}(t,\cdot):= (\Phi)_{\vert_{ \itGamma(t)}}=(\Phi_{\rm i})_{\vert_{ \itGamma(t)}} \quad\mbox{for}\quad  \ t>0.
\end{equation}
In view of this, we introduce the Dirichlet-to-Neumann (DN) map $\Lambda_{\itGamma(t)}$ associated with the elliptic equation 
$\nabla\cdot(H_{\rm w}\nabla\Phi_{\rm i})=0$ in the interior region $\cI(t)$ by 
\begin{equation}\label{defDN}
\Lambda_{\itGamma(t)} \Psi_{\rm i}:= (N\cdot (H_{\rm w}\nabla\Phi_{\rm i}))_{\vert_{\itGamma(t)}},
\end{equation}
where $\Phi_{\rm i}$ is a unique solution to the boundary value problem 
\begin{equation}\label{BVP}
\begin{cases}
 \nabla\cdot(H_{\rm w}\nabla\Phi_{\rm i})=0 &\mbox{in}\quad \cI(t), \\
 \Phi_{\rm i}=\Psi_{\rm i} &\mbox{on}\quad \partial\cI(t).
\end{cases}
\end{equation}
We note that the DN map $\Lambda_{\itGamma(t)}$ is symmetric and non-negative in $L^2(\itGamma(t))$, and depends on the unknown curve $\itGamma(t)$. 
Then, \eqref{NLSWEe}--\eqref{IrroCond} is transformed equivalently to \eqref{NLSWEe2} under the boundary conditions 
\begin{equation}\label{BC}
Z=Z_{\rm w}, \quad N\cdot(H\nabla\Phi)=\Lambda_{\itGamma(t)}\Psi_{\rm i} \quad\mbox{on}\quad \partial\cE(t), \ t>0,
\end{equation}
where $\Psi_{\rm i}$ is the trace of $\Phi$ on $\itGamma(t)$. 
In fact, once we obtain a solution $(\itGamma(t),Z,\Phi)$ of \eqref{NLSWEe2} and \eqref{BC}, then the other quantities can be recovered easily as follows. 
We recover $\Phi_\mathrm{i}$ as a unique solution of the boundary value problem \eqref{BVP} with $\Psi_{\rm i}=\Phi_{\vert_{\itGamma(t)}}$. 
Then, the second boundary condition in \eqref{BC} implies $N\cdot\nabla\Phi=N\cdot\nabla\Phi_{\rm i}$ on $\itGamma(t)$, 
so that $\nabla\Phi=\nabla\Phi_{\rm i}$ on $\itGamma(t)$. 
Using this and the second equation in \eqref{NLSWEe2}, we have $\dt\Phi_{\rm i} + \frac12|\nabla\Phi_{\rm i}|^2 + \gr Z_{\rm i} = 0$ on $\itGamma(t)$. 
In view of the second equation in \eqref{NLSWEi2}, we define 
\begin{equation}\label{defPi}
\ul{P}_{\rm i}=P_{\rm atm}-\rho(\dt\Phi_{\rm i} + \tfrac12|\nabla\Phi_{\rm i}|^2 + \gr Z_{\rm i}), 
\end{equation}
which satisfies $\ul{P}_{\rm i}=P_{\rm atm}$ on $\itGamma(t)$. 
Finally, putting $V=\nabla\Phi$ and $V_{\rm i}=\nabla\Phi_{\rm i}$, 
we see that $Z,V,V_{\rm i},\ul{P}_{\rm i}$ and $\itGamma(t)$ satisfy \eqref{NLSWEe}--\eqref{IrroCond}.

%-----------------------------------------------------------
\subsection{An equivalent formulation II}\label{subsect:EF2}
We can reformulate the problem in a different way without using the velocity potential $\Phi$ in the exterior region similar to the reformulation used in 
\cite[Section 5.1]{IguchiLannes2023} in the case with vertical sidewalls. 
More precisely, we show here that we can reduce the problem to a free boundary problem for the shallow water equations in the exterior domain 
with boundary conditions that are nonlocal in space and in time. 
By the second equation in \eqref{NLSWEi2} and the matching conditions \eqref{MC}, we have $\dt\Phi_{\rm i}+\frac12|V|^2+\gr Z=0$ on $\itGamma(t)$. 
Using a parametrization $x=X(t,s)$ for $s\in {\mathbb T}_L$ of the curve $\itGamma(t)$, 
we can define the functions $\Psi_{{\rm i},t}(t,\cdot)$ and $V_{\itGamma}$ on $\itGamma(t)$ as 
\[
\Psi_{{\rm i},t}(t,X(t,s)):=\partial_t \bigl[ \Psi_{\rm i}(t,X(t,s))\bigr] \quad\mbox{and}\quad V_\itGamma(t,X(t,s))=\dt X(t,s)
\quad\mbox{for}\quad s\in {\mathbb T}_L, \ t>0,
\]
where we recall that $\Psi_{\rm i}$ denotes the common value of $\Phi_{\rm i}$ and $\Phi$ on $\itGamma(t)$. 
In particular, using also the matching conditions \eqref{MC}, we get $\dt \Phi_{\rm i}=\Psi_{{\rm i},t} - V_\itGamma\cdot V$ on $\itGamma(t)$, 
so that the relation $\Psi_{\rm i,t} - V_\itGamma\cdot V+\tfrac12|V|^2+\gr Z=0$ holds on $\itGamma(t)$. 
Therefore, the problem reduces to 
\begin{equation}\label{NLSWEe4}
\begin{cases}
 \dt Z+\nabla\cdot(HV) = 0 &\mbox{in}\quad \cE(t), \ t>0, \\
 \dt V+\nabla(\frac12|V|^2+\gr Z) = 0 &\mbox{in}\quad \cE(t), \ t>0, \\
 \nabla^\perp\cdot V = 0 &\mbox{in}\quad \cE(t), \ t>0,
\end{cases}
\end{equation}
under the boundary conditions 
\begin{equation}\label{BC4}
Z=Z_{\rm i}, \quad (HN\cdot V) = \Lambda_{\itGamma(t)}\Psi_{\rm i} \quad\mbox{on}\quad \itGamma(t), \ t>0,
\end{equation}
with the DN map $\Lambda_{\it\Gamma(t)}$ as defined in \eqref{defDN}--\eqref{BVP}, and coupled with 
\begin{equation}\label{ODE}
\Psi_{{\rm i},t} - V_\itGamma\cdot V+\tfrac12 |V|^2+\gr Z=0 \quad\mbox{on}\quad \itGamma(t), \ t>0.
\end{equation}

\begin{remark}
The quantity $\Psi_{{\rm i},t} - V_\itGamma\cdot V$ is equal to the trace of $\dt\Phi_{\rm i}$ on $\itGamma(t)$ and is therefore intrinsically defined; 
however, both $\Psi_{{\rm i},t} $ and $V_\itGamma$ depend on the chosen parametrization for $\it\Gamma(t)$; 
in particular, any reparametrization of $\it\Gamma(t)$ changes the tangential component of $V_\itGamma$. 
The fact that the tangential velocity of the boundary depends on the chosen parametrization is standard in the analysis of free boundary problems 
where the choice of an appropriate parametrization can be important \cite{AmbroseMasmoudi,CordobaGancedo}.
\end{remark}

The following proposition ensures the equivalence between \eqref{NLSWEe}--\eqref{IrroCond} and \eqref{NLSWEe4}--\eqref{ODE} under appropriate conditions 
on the initial data.

\begin{proposition}\label{prop:equiv}
Let $Z,V,\Psi_{\rm i}$, and $\itGamma(t)$ be a smooth solution of \eqref{NLSWEe4}--\eqref{ODE}, 
and define $V_{\rm i}=\nabla\Phi_{\rm i}$ with $\Phi_{\rm i}$ given by \eqref{BVP}. 
If $V(0,\cdot)=V_{\rm i}(0,\cdot)$ on $\itGamma(0)$, then we have $V(t,\cdot)=V_{\rm i}(t,\cdot)$ on $\itGamma(t)$ for all $t$. 
Particularly, by defining $\ul{P}_{\rm i}$ by \eqref{defPi} we see that $Z,V,V_{\rm i},\ul{P}_{\rm i}$, and $\itGamma(t)$ satisfy \eqref{NLSWEe}--\eqref{IrroCond}. 
\end{proposition}

\begin{proof}
Let $x=X(t,s)$ for $s\in\mathbb{T}_L\simeq\R/(L\mathbb{Z})$ be a parametrization of the curve $\itGamma(t)$. 
It is sufficient to show that $N^\perp\cdot V=N^\perp\cdot V_{\rm i}$ on $\itGamma(t)$, which is equivalent to 
$(\partial_s X)\cdot(V\circ X)=(\partial_s X)\cdot(V_{\rm i}\circ X)$. 
We put 
\[
F(t,s)=(\partial_s X(t,s))\cdot(V-V_{\rm i})(t,X(t,s)).
\]
Then, we see that 
\begin{align*}
\dt F 
&= (\dt\partial_s X)\cdot ((V-V_{\rm i})\circ X) 
 + (\partial_s X)\cdot\{ (\dt V-\dt V_{\rm i})\circ X + (((\dt X)\cdot\nabla)(V-V_{\rm i}))\circ X \}.
\end{align*}
Here, it follows from the second equation in \eqref{NLSWEe4} and the relation $V_{\rm i}=\nabla\Phi_{\rm i}$ that 
\begin{align*}
(\partial_s X)\cdot \{ (\dt V-\dt V_{\rm i})\circ X \}
&= - (\partial_s X)\cdot \{ \nabla(\tfrac12|V|^2+\gr Z+\dt\Phi_{\rm i})\circ X \}  \\
&= - \partial_s \{ (\dt\Phi_{\rm i}+\tfrac12|V|^2+\gr Z)\circ X \}.
\end{align*}
Since $\dt \Phi_{\rm i}=\Psi_{\rm i,t}+V_{\itGamma}\cdot V_{\rm i}$ on $\itGamma(t)$, 
we deduce from \eqref{ODE} that  $(\dt\Phi_{\rm i}+\tfrac12|V|^2+\gr Z)\circ X=\partial_t X\cdot(V-V_{\rm i})\circ X$, so that 
\[
(\partial_s X)\cdot \{ (\dt V-\dt V_{\rm i})\circ X \}=  - \partial_s \{ \partial_t X\cdot(V-V_{\rm i})\circ X\}.
\]
In view of the identity \eqref{F1}, we see that 
\begin{align*}
(\partial_s X)\cdot\{ (((\dt X)\cdot\nabla)(V-V_{\rm i}))\circ X \}
&= (\dt X)\cdot\{ (((\partial_s X)\cdot\nabla)(V-V_{\rm i}))\circ X \} \\
&= (\dt X)\cdot \partial_s( (V-V_{\rm i})\circ X ).
\end{align*}
Therefore, we obtain $\dt F=0$; since $F(0,\cdot)=0$ by assumption, this implies $F(t,\cdot)=0$ for all $t$. 
\end{proof}

%-----------------------------------------------------------
\subsection{Fixing the boundary}\label{sectfixing}
We have reformulated the wave-structure interaction model \eqref{NLSWEe}--\eqref{MC} under the irrotational condition \eqref{IrroCond} into two different ways; 
\eqref{NLSWEe2} and \eqref{BC} in the first formulation, and \eqref{NLSWEe4}--\eqref{ODE} in the second formulation. 
We proceed to transform these free boundary problems into problems cast on a fixed domain. 
In order to do so, we use the diffeomorphism $\varphi(t,\cdot)$ introduced in Section \ref{subsect:CT}. 
We recall the definition \eqref{DiffOP} of the differential operators $\nabla^\varphi$ and $\dt^\varphi$. 
By the definition, we have 
\begin{equation}\label{DiffOP2}
\nabla^\varphi f
 = ((\partial\varphi)^{-1})^\mathrm{T}\nabla f
 = \frac{1}{J}
  \begin{pmatrix}
   \partial_2\varphi_2 & -\partial_1\varphi_2 \\
   -\partial_2\varphi_1 & \partial_1\varphi_1
  \end{pmatrix}
  \nabla f,
\end{equation}
where we used the notation 
\[
\partial\varphi
 = \begin{pmatrix}
    \partial_1\varphi_1 & \partial_2\varphi_1 \\
    \partial_1\varphi_2 & \partial_2\varphi_2 \\
   \end{pmatrix}, \quad 
J = \det(\partial\varphi).
\]
We also have the relation 
\begin{equation}\label{DiffOP3}
\dt^\varphi f = \dt f - (\dt \varphi)\cdot\nabla^\varphi f.
\end{equation}
Let $\ul{N} = (\ul{N}_1,\ul{N}_2)^\mathrm{T}$ be the unit outward normal vector to $\ul{\itGamma}=\partial\ul{\cI}$. 
In view of \eqref{DiffOP2} we introduce a notation 
\begin{equation}\label{Nphi}
N^\varphi = J((\partial\varphi)^{-1})^\mathrm{T}\ul{N} = 
  \begin{pmatrix}
    \partial_2\varphi_2 & -\partial_1\varphi_2 \\
   -\partial_2\varphi_1 &  \partial_1\varphi_1
  \end{pmatrix}
  \ul{N}.
\end{equation}
Then, we have $N^\varphi = -((\ul{N}^{\perp}\cdot\nabla)\varphi)^\perp$, 
which implies that $N^\varphi\circ\varphi^{-1}$ is an outward normal vector to $\itGamma(t)= \partial\cI(t)$. 
Then, denoting $\phi_{\rm i}=\Phi_{\rm i}\circ \varphi$ and $\psi_{\rm i}=\Psi_{\rm i}\circ\varphi$, the boundary value problem \eqref{BVP} is transformed into 
\begin{equation}\label{BVP2}
\begin{cases}
\nabla\cdot(A_{\rm i}\nabla\phi_\mathrm{i}) = 0 & \mbox{in} \quad \ul{\cI}, \\
\phi_\mathrm{i} = \psi_{\rm i} & \mbox{on} \quad \partial \ul{\cI},
\end{cases}
\end{equation}
where the coefficient matrix $A=A(\varphi)$ is defined by 
\[
A_{\rm i} = J h_\mathrm{w} (\partial\varphi)^{-1}((\partial\varphi)^{-1})^\mathrm{T}
  = \frac{h_\mathrm{w}}{J}
  \begin{pmatrix}
    \partial_2\varphi_2 & -\partial_2\varphi_1 \\
   -\partial_1\varphi_2 &  \partial_1\varphi_1
  \end{pmatrix}
  \begin{pmatrix}
    \partial_2\varphi_2 & -\partial_1\varphi_2 \\
   -\partial_2\varphi_1 &  \partial_1\varphi_1
  \end{pmatrix}.
\]
The corresponding DN map is defined by 
\begin{equation}\label{defDNbis}
\Lambda_\varphi \psi_{\rm i} = (\ul{N}\cdot A_{\rm i}\nabla\phi_{\rm i})_{\vert_{\ul{\itGamma}}}
 = N^\varphi \cdot (h_{\rm w} \nabla^\varphi \phi_{\rm i})_{\vert_{\ul{\itGamma}}},
\end{equation}
which is also symmetric and non-negative in $L^2(\ul{\itGamma})$.

%-----------------------------------------------------------
\subsubsection{In the case of the formulation I}
By using the diffeomorphism $\varphi(t,\cdot)$ introduced in Section \ref{subsect:CT}, 
we see that the free boundary problem \eqref{NLSWEe2} and \eqref{BC} in the first reformulation is transformed into 
\begin{equation}\label{NLSWEe3}
\begin{cases}
 \dt^\varphi \zeta + \nabla^\varphi \cdot ( h\nabla^\varphi \phi ) = 0 &\mbox{in}\quad (0,T)\times\ul{\cE}, \\
 \dt^\varphi \phi + \frac12|\nabla^\varphi \phi|^2 + g \zeta = 0 &\mbox{in}\quad (0,T)\times\ul{\cE},
\end{cases}
\end{equation}
under the boundary conditions 
\begin{equation}\label{BC2}
\zeta = \zeta_{\rm i}, \quad  N^\varphi \cdot (h \nabla^\varphi \phi) = \Lambda_\varphi \phi \quad\mbox{on}\quad (0,T)\times\partial \ul{\cE},
\end{equation}
with the DN map $\Lambda_\varphi$ given by \eqref{defDNbis}.

%-----------------------------------------------------------
\subsubsection{In the case of the formulation II}
Similarly, we see that the free boundary problem \eqref{NLSWEe4}--\eqref{ODE} in the second reformulation is transformed into 
\begin{equation}\label{NLSWEe5}
\begin{cases}
 \dt^\varphi\zeta+\nabla^\varphi\cdot(hv) = 0 &\mbox{in}\quad (0,T)\times\ul{\cE}, \\
 \dt^\varphi v+\nabla^\varphi(\frac12|v|^2+\gr\zeta) = 0 &\mbox{in}\quad (0,T)\times\ul{\cE}, \\
 (\nabla^\varphi)^\perp\cdot v = 0 &\mbox{in}\quad (0,T)\times\ul{\cE},
\end{cases}
\end{equation}
under the boundary conditions 
\begin{equation}\label{BC5}
\zeta=\zeta_{\rm i}, \quad N^\varphi\cdot (h v) = \Lambda_\varphi \psi_{\rm i} \quad\mbox{on}\quad (0,T)\times\ul{\itGamma},
\end{equation}
where $\psi_{\rm i}$ is found by solving the following equation 
\begin{equation}\label{ODE5}
\dt \psi_{\rm i} -\dt \varphi\cdot v + \tfrac12|v|^2+\gr\zeta = 0 \quad\mbox{on}\quad  (0,T)\times\ul{\itGamma}.
\end{equation}

%----------------------------------------------------------------------------------------------------------------------
\section{Energy estimate of a linearized problem}\label{sect:EE}
In this section we consider the problem \eqref{NLSWEe5}--\eqref{ODE5}, where the unknowns are $\zeta,v,v_{\rm i}, \phi_{\rm i}$, and the diffeomorphism $\varphi$. 
We recall that $\zeta_{\rm i}$ is given by $\zeta_{i}=\zeta_{\rm w}=Z_{\rm w}\circ\varphi$ with a given function $Z_{\rm w}$. 
We first derive linearized equations. 
To this end, we use an abstract linearization operator denoted with a dot; one can typically think of it as a tangential or time derivative. 
In the derivation, we need to introduce so-called Alinhac's good unknowns to exclude a derivative loss caused by the dependence on the diffeomorphism $\varphi$ 
in the problem \eqref{NLSWEe5}--\eqref{ODE5}. 
Then, we derive an energy estimate for the linearized problem. 
It turns out that the boundary conditions are weakly dissipative in the sense \cite{IguchiLannes2023}.

%-----------------------------------------------------------
\subsection{Derivation of a linearized problem}\label{subsect:DLP}
We consider here $\zeta,v,\psi_{\rm i}$, and $\varphi$ a solution to \eqref{NLSWEe5}--\eqref{ODE5} satisfying $v(0,\cdot)=v_{\rm i}(0,\cdot)$ on $\ul{\itGamma}$, 
where we defined $v_{\rm i}=\nabla \phi_{\rm i}$ with $\phi_{\rm i}$ given by \eqref{BVP2}. 
Using Proposition \ref{prop:equiv}, one can replace the boundary conditions \eqref{BC5} by $\zeta=\zeta_{\rm i}$ and $v=v_{\rm i}$ on $(0,T)\times\ul{\itGamma}$. 
This is what we do in this section for the derivation of the linearized equations.

%-----------------------------------------------------------
\subsubsection{Linearization rules and vectorial identities}
For an unknown function $f$, we denote the variation in the linearization by $\dot{f}$. 
Then, we have 
\begin{equation}\label{linearization 1}
\begin{cases}
(\nabla^\varphi \cdot \mbox{\boldmath$f$})\,\dot{}
= \nabla^\varphi \cdot \dot{\mbox{\boldmath$f$}}
 - ( (\partial^\varphi \dot{\varphi})^\mathrm{T}\nabla^\varphi ) \cdot \mbox{\boldmath$f$}, \\
(\nabla^\varphi f)\,\dot{}
= \nabla^\varphi \dot{f} - (\partial^\varphi \dot{\varphi})^\mathrm{T}\nabla^\varphi f, \\
(\dt^\varphi f)\,\dot{}
= \dt^\varphi\dot{f} - (\dt^\varphi\dot{\varphi}) \cdot \nabla^\varphi f,
\end{cases}
\end{equation}
where for $v = (v_1,v_2)^\mathrm{T}$ we use the notation 
\[
\partial^\varphi v = 
 \begin{pmatrix}
  \partial_1^\varphi v_1 & \partial_2^\varphi v_1 \\
  \partial_1^\varphi v_2 & \partial_2^\varphi v_2
 \end{pmatrix}.
\]
Now, we introduce the good unknown $\check{f}$ associated with a function $f$ by 
\begin{equation}\label{GoodUnknown}
\check{f} = \dot{f} - \dot{\varphi} \cdot \nabla^\varphi f.
\end{equation}
Then, we have the commutation rules $(\nabla^\varphi\cdot  \mbox{\boldmath$f$})\,\check{} =\nabla^\varphi\cdot  \check{\mbox{\boldmath$f$}}$, 
$(\nabla^\varphi f)\,\check{} =\nabla^\varphi \check{f}$, and $(\dt^\varphi f)\,\check{} =\dt^\varphi  \check{f}$, or equivalently, 
\begin{equation}\label{linearization 2}
\begin{cases}
(\nabla^\varphi \cdot \mbox{\boldmath$f$})\,\dot{}
= \nabla^\varphi \cdot \check{\mbox{\boldmath$f$}}
 + (\dot{\varphi} \cdot \nabla^\varphi)(\nabla^\varphi \cdot \mbox{\boldmath$f$}), \\
(\nabla^\varphi f)\,\dot{}
= \nabla^\varphi \check{f} + (\dot{\varphi} \cdot \nabla^\varphi)\nabla^\varphi f, \\
(\dt^\varphi f)\,\dot{}
= \dt^\varphi\check{f} + (\dot{\varphi} \cdot \nabla^\varphi)\dt^\varphi f.
\end{cases}
\end{equation}
We will also use several vectorial identities in the computations of this section. By the identity \eqref{F1}, for any $\R^2$-valued functions $f$ and $g$ we have 
\begin{equation}\label{F2}
f\cdot(g\cdot\nabla^\varphi)v = g\cdot(f\cdot\nabla^\varphi)v+(f\cdot g^\perp)(\nabla^\varphi)^\perp\cdot v.
\end{equation}
We note also the identities 
\begin{equation}\label{F3}
\begin{cases}
 \nabla^\varphi(f\cdot g) = (\partial^\varphi f)^\mathrm{T}g + (\partial^\varphi g)^\mathrm{T}f, \\
 (\partial^\varphi f)^\mathrm{T}g = (g\cdot\nabla^\varphi)f - ((\nabla^\varphi)^\perp\cdot f)g^\perp.
\end{cases}
\end{equation}

%-----------------------------------------------------------
\subsubsection{Linearization of the equation in the exterior domain $\ul{\cE}$}
According to the linearization rules of the previous subsection, the linearized equations of \eqref{NLSWEe5} have the form, in terms of the good unknowns, 
\[\begin{cases}
 \dt^\varphi\check{\zeta}+\nabla^\varphi\cdot(h\check{v}+v\check{\zeta}) = 0 &\mbox{in}\quad (0,T)\times\ul{\cE}, \\
 \dt^\varphi\check{v}+\nabla^\varphi(v\cdot\check{v}+\gr\check{\zeta}) = 0 &\mbox{in}\quad (0,T)\times\ul{\cE}, \\
   (\nabla^\varphi)^\perp\cdot\check{v} = 0 &\mbox{in}\quad (0,T)\times\ul{\cE}.
\end{cases}
\]
Recalling that $w=v-\dt\varphi$ and using \eqref{F3}, these equations are equivalent to 
\begin{equation}\label{LSWEe1}
\begin{cases}
 \dt\check{\zeta}+\nabla^\varphi\cdot(h\check{v}+w\check{\zeta}) = -\check{\zeta}(\nabla^\varphi\cdot\dt\varphi) &\mbox{in}\quad (0,T)\times\ul{\cE}, \\
 \dt\check{v}+\nabla^\varphi(w\cdot\check{v}+\gr\check{\zeta}) = -(\check{v}\cdot\nabla^\varphi)\dt\varphi+((\nabla^\varphi)^\perp\cdot\dt\varphi)\check{v}^\perp
  &\mbox{in}\quad (0,T)\times\ul{\cE}, \\
  (\nabla^\varphi)^\perp\cdot\check{v} = 0 &\mbox{in}\quad (0,T)\times\ul{\cE},
\end{cases}
\end{equation}
where we used the last equation to derive the second equation in \eqref{LSWEe1}. 
We can regard the right-hand sides of these equations as lower order terms.

%-----------------------------------------------------------
\subsubsection{Linearization of the equation in the interior domain $\ul{\cI}$}
We recall that $v_{\rm i}=\nabla^\varphi\phi_{\rm i}$ where $\phi_{\rm i}$ solves \eqref{BVP2}, so that $\nabla^\varphi\cdot(h_{\rm i}v_{\rm i})=0$. 
Using the above linearization rules, this directly implies that 
\begin{equation}\label{LSWEi1}
\begin{cases}
 \nabla^\varphi\cdot(h_{\rm i}\check{v}_{\rm i}+\check{\zeta}_{\rm i} v_{\rm i}) = 0 &\mbox{in}\quad (0,T)\times\ul{\cI}, \\
 \check{v}_{\rm i}=\nabla^\varphi\check{\phi}_{\rm i} &\mbox{in}\quad (0,T)\times\ul{\cI}.
\end{cases}
\end{equation}

%-----------------------------------------------------------
\subsubsection{Linearization of the boundary condition $\zeta=\zeta_{\rm i}$}
Linearizing the boundary condition $\zeta=\zeta_{\rm i}$ on $\ul{\itGamma}$, we have $\dot{\zeta}=\dot{\zeta}_{\rm i}$, or equivalently, 
\[
\check{\zeta}-\check{\zeta}_{\rm i} = \dot{\varphi}\cdot\nabla^\varphi(\zeta_{\rm i}-\zeta) \quad\mbox{on}\quad (0,T)\times\ul{\itGamma}.
\]
Since $(N^\varphi)^\perp\cdot\nabla^\varphi=\ul{N}^\perp\cdot\nabla$ is a tangential derivative on $\ul{\itGamma}$, 
this linearized boundary condition can be written as 
\begin{equation}\label{LBC1}
\check{\zeta}=\check{\zeta}_{\rm i}-\frac{N^\varphi\cdot\nabla^\varphi(\zeta-\zeta_{\rm i})}{|N^\varphi|^2} N^\varphi\cdot\dot{\varphi}
 \quad\mbox{on}\quad (0,T)\times\ul{\itGamma}.
 \end{equation}
In this paper we are considering the case where the transversality condition \eqref{TransCond} is satisfied. 
In terms of the transformed variables, 
this situation is equivalent to $|N^\varphi\cdot\nabla^\varphi(\zeta-\zeta_{\rm i})| \geq c_0>0$ on $(0,T)\times\ul{\itGamma}$. 
Therefore, \eqref{LBC1} yields 
\begin{equation}\label{EqPhi}
N^\varphi\cdot\dot{\varphi} = -\frac{|N^\varphi|^2}{N^\varphi\cdot\nabla^\varphi(\zeta-\zeta_{\rm i})}( \check{\zeta} -\check{\zeta}_{\rm i})
 \quad\mbox{on}\quad (0,T)\times\ul{\itGamma}. 
\end{equation}
This is the equation that determines essentially the variation of the unknown curve $\itGamma(t)$.

%-----------------------------------------------------------
\subsubsection{Linearization of the boundary condition $v=v_{\rm i}$}
Proceeding as for \eqref{LBC1}, the linearization of the boundary condition $v=v_{\rm i}$ yields 
\[
\check{v}=\check{v}_{\rm i}-\frac{(N^\varphi\cdot\nabla^\varphi)(v-v_{\rm i})}{|N^\varphi|^2} N^\varphi\cdot\dot{\varphi}
 \quad\mbox{on}\quad (0,T)\times\ul{\itGamma}.
\]
Plugging \eqref{EqPhi} into this identity to eliminate $\dot{\varphi}$, we obtain 
\begin{equation}\label{LBC2}
\check{v} = \check{v}_{\rm i}
 + \frac{(N^\varphi\cdot\nabla^\varphi)(v-v_{\rm i})}{(N^\varphi\cdot\nabla^\varphi)(\zeta-\zeta_{\rm i})} (\check{\zeta}-\check{\zeta}_{\rm i}).
\end{equation}
On the other hand, it follows from the first equations in \eqref{NLSWEe5}, 
and the fact that $\dt^\varphi\zeta_{\rm i}=0$ and $\nabla^\varphi\cdot(h_{\rm i}v_{\rm i})=0$, that 
\[
\begin{cases}
 \dt\zeta+w\cdot\nabla^\varphi\zeta+h\nabla^\varphi\cdot v=0 &\mbox{in}\quad (0,T)\times\ul{\cE}, \\
 \dt\zeta_{\rm i}+w_{\rm i}\cdot\nabla^\varphi\zeta_{\rm i}+h_{\rm i}\nabla^\varphi\cdot v_{\rm i}=0 &\mbox{in}\quad (0,T)\times\ul{\cI},
\end{cases}
\]
where $w$ and $w_{\rm i}$ are defined by \eqref{defw}. 
Now, taking the trace of these equations on $\ul{\itGamma}$ and using the matching conditions $(\zeta,v)=(\zeta_{\rm i},v_{\rm i})$ on $\ul{\itGamma}$, 
we obtain $w\cdot\nabla^\varphi(\zeta-\zeta_{\rm i})+h\nabla^\varphi\cdot(v-v_{\rm i})=0$ on $\ul{\itGamma}$, which is equivalent to 
$(N^\varphi\cdot w)N^\varphi\cdot\nabla^\varphi(\zeta-\zeta_{\rm i})+hN^\varphi\cdot(N^\varphi\cdot\nabla^\varphi)(v-v_{\rm i})=0$ on $\ul{\itGamma}$. 
We also note that $(N^\varphi)^\perp\cdot(N^\varphi\cdot\nabla^\varphi)(v-v_{\rm i})=N^\varphi\cdot((N^\varphi)^\perp\cdot\nabla^\varphi)(v-v_{\rm i})=0$. 
Therefore, \eqref{LBC2} is equivalent to 
\begin{equation}\label{LBC3}
\begin{cases}
 N^\varphi\cdot(h\check{v}+w\check{\zeta})=N^\varphi\cdot(h_{\rm i}\check{v}_{\rm i}+w_{\rm i}\check{\zeta}_{\rm i}) &\mbox{on}\quad (0,T)\times\ul{\itGamma}, \\
 (N^\varphi)^\perp\cdot\check{v}=(N^\varphi)^\perp\cdot\check{v}_{\rm i} &\mbox{on}\quad (0,T)\times\ul{\itGamma};
\end{cases}
\end{equation}
the term $w_{\rm i}\check{\zeta}_{\rm i}$ in the right-hand side of the first equation can be viewed as a lower order term. 
We note that these conditions are decoupled from the variation of the diffeomorphism $\dot{\varphi}$.

%-----------------------------------------------------------
\subsubsection{Linearization of the evolution equation for $\psi_{\rm i}$}
Finally, we linearize \eqref{ODE5}.
In view of $v=v_{\rm i}$ on $\ul{\itGamma}$, we have $\dt^\varphi\phi_{\rm i}=\dt\phi_{\rm i}-(\dt\varphi)\cdot v$ and 
$\check{\phi}_{\rm i}=\dot{\phi}_{\rm i}-\dot{\varphi}\cdot v$ on $\ul{\itGamma}$. 
We note also that the equation \eqref{ODE5} for $\psi_{\rm i}$ can be written as $\dt^\varphi\phi_{\rm i}+\tfrac12|v|^2+\gr\zeta=0$. 
Therefore, we see that on $\ul{\itGamma}$ 
\begin{align*}
(\dt^\varphi\phi_{\rm i}+\tfrac12|v|^2+\gr\zeta)\,\dot{}
&= \dt\dot{\phi}_{\rm i}-(\dt\dot{\varphi})\cdot v + w\cdot\dot{v}+\gr\dot{\zeta} \\
&= \dt\check{\phi}_{\rm i}+w\cdot\check{v}+\gr\check{\zeta} + \dot{\varphi}\cdot(\dt v+\gr\nabla^\varphi\zeta) + w\cdot(\dot{\varphi}\cdot\nabla^\varphi)v. 
\end{align*}
Using the identity \eqref{F2} we notice that 
\begin{align*}
\dot{\varphi}\cdot(\dt v+\gr\nabla^\varphi\zeta) + w\cdot(\dot{\varphi}\cdot\nabla^\varphi)v
&= \dot{\varphi}\cdot( \dt v+(w\cdot\nabla^\varphi)v+\gr\nabla^\varphi\zeta - ((\nabla^\varphi)^\perp\cdot v)w^\perp ) \\
&= \dot{\varphi}\cdot( \dt^\varphi v+\nabla^\varphi(\tfrac12|v|^2+\gr\zeta) + ((\nabla^\varphi)^\perp\cdot v)(\dt\varphi)^\perp ) \\
&= 0,
\end{align*}
where we used the last two equations in \eqref{NLSWEe5}. 
Denoting by $\check{\psi}_{\rm i}$ the trace of $\check{\phi}_{\rm i}$ on $\ul{\itGamma}$, we obtain 
\begin{equation}\label{LODE1}
\dt\check{\psi}_{\rm i}+w\cdot\check{v}+\gr\check{\zeta}=0 \quad\mbox{on}\quad (0,T)\times\ul{\itGamma}.
\end{equation}

%-----------------------------------------------------------
\subsubsection{The linearized system in general form}
Gathering the above results, we obtain the linearized system associated with \eqref{NLSWEe5}--\eqref{ODE5}. 
We shall actually deal with a slightly more general version that include various perturbations, 
and that will allow us to handle commutator terms in the high-order energy estimates in Section \ref{sect:APE}. 
More precisely, the linearized equations for the unknowns $\check{\zeta},\check{v}$, and $\check{\psi}_{\rm i}$ 
that we are looking for consist of the linearized irrotational shallow water equations 
\begin{equation}\label{LSWEe2}
\begin{cases}
 \dt\check{\zeta}+\nabla^\varphi\cdot(h\check{v}+w\check{\zeta}) = f_1 &\mbox{in}\quad (0,T)\times\ul{\cE}, \\
 \dt\check{v}+\nabla^\varphi(w\cdot\check{v}+\gr\check{\zeta}) = f_2 &\mbox{in}\quad (0,T)\times\ul{\cE}, \\
 (\nabla^\varphi)^\perp\cdot\check{v} = f_3 &\mbox{in}\quad (0,T)\times\ul{\cE},
\end{cases}
\end{equation}
under the boundary conditions 
\begin{equation}\label{LBC4}
\begin{cases}
 N^\varphi\cdot(h\check{v}+w\check{\zeta}) = N^\varphi\cdot(h_{\rm i}\check{v}_{\rm i}+f_{{\rm i},4}) &\mbox{on}\quad (0,T)\times\ul{\itGamma}, \\
 (N^\varphi)^\perp\cdot(h\check{v})=(N^\varphi)^\perp\cdot(h_{\rm i}\check{v}_{\rm i}) &\mbox{on}\quad (0,T)\times\ul{\itGamma},
\end{cases}
\end{equation}
where $\check{v}_{\rm i}$ is given in terms of $\check{\psi}_{\rm i}$ by 
\begin{equation}\label{LSWEi2}
\begin{cases}
 \nabla^\varphi\cdot(h_{\rm i}\check{v}_{\rm i}+f_{{\rm i},2})=f_{{\rm i},1} &\mbox{in}\quad (0,T)\times\ul{\cI}, \\
 \check{v}_{\rm i}=\nabla^\varphi\check{\phi}_{\rm i} + f_{{\rm i},3} &\mbox{in}\quad (0,T)\times\ul{\cI}, \\
 \check{\phi}_{\rm i}=\check{\psi}_{\rm i}  &\mbox{on}\quad (0,T)\times\ul{\itGamma},
\end{cases}
\end{equation}
while $\check{\psi}_{\rm i}$ solves 
\begin{equation}\label{LBC4bis}
\dt\check{\psi}_{\rm i}+w\cdot\check{v}+\gr\check{\zeta}=f_{{\rm i},5} \quad \mbox{on}\quad (0,T)\times\ul{\itGamma},
\end{equation}
and where $f_1,f_2,f_3$ are given functions in $(0,T)\times\ul{\cE}$, whereas $f_{{\rm i},1},\ldots,f_{{\rm i},5}$ are given functions in $(0,T)\times\ul{\cI}$. 
In addition, the evolution of the curve $\itGamma(t)$ is governed by 
\begin{equation}\label{EqPhibis}
N^\varphi\cdot\dot{\varphi} = \frac{|N^\varphi|^2}{N^\varphi\cdot\nabla^\varphi(\zeta-\zeta_{\rm i})}( \check{\zeta} -\check{\zeta}_{\rm i})
 \quad\mbox{on}\quad (0,T)\times\ul{\itGamma}. 
\end{equation}

\begin{remark}
If $f_{{\rm i},1}=0$ and $f_{{\rm i},2}=f_{{\rm i},3}=f_{{\rm i},4}=0$, then the first boundary condition in \eqref{LBC4} can be written as 
\[
N^\varphi\cdot(h\check{v}+w\check{\zeta}) =  \Lambda_\varphi \check{\psi}_{\rm i}\quad \mbox{on}\quad (0,T)\times \ul{\itGamma},
\]
so that the first equation in \eqref{LBC4} can be seen as a perturbation of this boundary condition by lower order terms. 
This perturbation consists in replacing the elliptic problem \eqref{BVP2} for $\check{\phi}_{\rm i}$ by \eqref{LSWEi2}. 
\end{remark}

%-----------------------------------------------------------
\subsection{Energy estimate}
We proceed in this section to derive an energy estimate for solutions $\check{\zeta},\check{v},\check{\psi}_{\rm i}$ to the linearized equations 
\eqref{LSWEe2}--\eqref{LBC4bis}. 
We consider throughout this section that the diffeomorphism $\varphi$ is given, and therefore do not consider the evolution equation \eqref{EqPhibis}. 
We note that the first two equations in \eqref{LSWEe2} for the unknowns $\check{u}=(\check{\zeta},\check{v}^\mathrm{T})^\mathrm{T}$ 
can be written in a matrix form as 
\begin{equation}\label{LSWEe3}
\dt\check{u}+\partial_1^\varphi(G_1\Sigma(u)\check{u})+\partial_2^\varphi(G_2\Sigma(u)\check{u})=f,
\end{equation}
where $G_1,G_2$, and $\Sigma(u)$ with $u=(h,w^\mathrm{T})^\mathrm{T}$ are $3\times3$ symmetric matrices given by 
\[
G_j=
 \begin{pmatrix}
  0 & \mathbf{e}_j^\mathrm{T} \\
  \mathbf{e}_j & 0_{2\times2}
 \end{pmatrix}, \quad
\Sigma(u)=
 \begin{pmatrix}
  \gr & w^\mathrm{T} \\
  w & h\mathrm{Id}_{2\times2}
 \end{pmatrix}
\]
for $j=1,2$ with the standard basis $\mathbf{e}_j$ in $\R^2$ and $f=(f_1,f_2^\mathrm{T})^\mathrm{T}$. 
We note that the matrix $\Sigma(u)$ is positive definite if and only if $\gr h-|w|^2>0$, that is, 
the flow is subcritical related to the motion of the curve $\itGamma(t)$. 
Throughout this paper, we assume this condition. 
Under the subcritical condition of the flow, \eqref{LSWEe3} forms a symmetrizable hyperbolic system, with symmetrizer $\Sigma(u)$; 
the density of the corresponding energy function is therefore given by $\check{u}\cdot\Sigma(u)\check{u}$, 
and one readily checks that the energy conservation law 
\begin{equation}\label{localCE}
 \dt(J(\tfrac12\check{u}\cdot\Sigma(u)\check{u}))
  + J\nabla^\varphi\cdot((\gr\check{\zeta}+w\cdot\check{v})(h\check{v}+w\check{\zeta})) = F_{\rm e,1}
    \quad\mbox{in}\quad (0,T)\times\ul{\cE}.
\end{equation}
holds with 
\[
F_{\rm e,1}=\tfrac{1}{2}\check{u}\cdot\dt (J\Sigma(u))\check{u} + f\cdot \Sigma(u)\check{u}.
\]
In view of this, we assume the following conditions on the coefficients $h,w,h_{\rm i}$ and the diffeomorphism $\varphi$. 
\begin{assumption}\label{ass:LP}
There exit positive constants $c_0, M_0,M_1$, and $T$ such that $h,w,h_{\rm i}$, and $\varphi$ satisfy the following properties: 
\begin{enumerate}
\item[{\rm (i)}]
$\gr h(t,x)-|w(t,x)|^2\geq c_0$ and $h(t,x)\leq M_0$ for $(t,x)\in[0,T]\times\ul{\cE}$. 
\item[{\rm (ii)}]
$c_0 \leq h_{\rm i}(t,x) \leq M_0$ for $(t,x)\in[0,T]\times\ul{\cI}$. 
\item[{\rm (iii)}]
$J(t,x)=\det(\partial\varphi(t,x))\geq c_0$ and $|\partial\varphi(t,x)|\leq M_0$ for $(t,x)\in[0,T]\times\R^2$. 
\item[{\rm (iv)}]
$\|\bm{\partial}(h,w)\|_{L^1(0,T;L^\infty(\ul{\cE}))} + \|\bm{\partial} h_{\rm i}\|_{L^1(0,T;L^\infty(\ul{\cI}))}
 + \|\bm{\partial}\partial\varphi\|_{L^1(0,T;L^\infty(\R^2))} \leq M_1$.
\end{enumerate}
\end{assumption}

\begin{remark}
If we work in the space $H^m$ with $m\geq4$ for the nonlinear problem \eqref{NLSWEe5}--\eqref{ODE5}, we may replace the above condition {\rm (iv)} 
with a standard and stronger assumption $\|\bm{\partial}(h,w)\|_{L^\infty((0,T)\times\ul{\cE})} + \|\bm{\partial} h_{\rm i}\|_{L^\infty((0,T)\times\ul{\cI})}
 + \|\bm{\partial}\partial\varphi\|_{L^\infty((0,T)\times\R^2)} \leq M_1$. 
However, in the critical case $m=3$ at the quasilinear regularity threshold, the assumption $\|\bm{\partial}\partial\varphi\|_{L^\infty((0,T)\times\R^2)} \leq M_1$ 
cannot be acceptable for the nonlinear problem due to the lack of enough regularity of the contact line. 
\end{remark}

Noting that by the relation $J\nabla^\varphi\cdot f=\nabla\cdot(J(\partial\varphi)^{-1}f)$, we have 
\begin{equation}\label{DivTh}
\int_{\ul{\cE}}J\nabla^\varphi\cdot f = -\int_{\ul{\itGamma}}N^\varphi\cdot f, \quad
\int_{\ul{\cI}}J\nabla^\varphi\cdot f_{\rm i} = \int_{\ul{\itGamma}}N^\varphi\cdot f_{\rm i},
\end{equation}
so that one can integrate \eqref{localCE} in space and time to obtain the energy identity 
\begin{align}\label{weakdissip0}
\tfrac{1}{2} \big(\check{u},\Sigma(u)\check{u}\big)_{L^2(\ul{\cE})}(t)
&= \tfrac{1}{2} \big(\check{u},\Sigma(u)\check{u}\big)_{L^2(\ul{\cE})}(0)
 + \int_0^t (F_{{\rm e},1},\Sigma(u)\check{u})_{L^2(\ul{\cE})} \\
&\quad
 + \int_0^t \int_{\ul\itGamma} (\gr\check{\zeta}+w\cdot\check{v}) N^\varphi\cdot (h\check{v}+w\check{\zeta}), \nonumber
\end{align}
where $(\cdot,\cdot)_{L^2(\ul{\cE})}$ stands for the standard $L^2(\ul{\cE})$ scalar product. 
Now, under Assumption \ref{ass:LP}, $(\check{u},\Sigma(u)\check{u})_{L^2(\ul{\cE})}^{1/2}$ is equivalent to $\Vert \cdot \Vert_{L^2(\ul{\cE})}$, 
so that one could deduce from the above identity a control of $\big(\check{u},\Sigma(u)\check{u}\big)_{L^2(\ul{\cE})}$ by a Gronwall type argument 
if the last term were not present. 
The control of this boundary integral is the central point in the analysis of hyperbolic initial boundary value problems. 
The general idea is that an energy estimate is possible if this boundary integral is non-positive, up to terms that depend only on the data of the problem. 
When the boundary conditions have such a property, they are called {\it dissipative}. 
There actually exist various notions of dissipativity that focus on the sign properties of the integrand of the boundary integral, that is, 
in the present situation, of $(\gr\check{\zeta}+w\cdot\check{v}) N^\varphi\cdot (h\check{v}+w\check{\zeta})$. 
As discussed in \cite{IguchiLannes2023}, none of them applies here. 
For this reason, the notion of {\it weak dissipativity} was proposed in \cite{IguchiLannes2023}; as opposed to the other notions of dissipativity, 
it deals with the positivity of the double integral in space and time of the boundary term, rather than on the integrand. 
The following proposition gives a basic energy estimate for solutions to the linearized problem \eqref{LSWEe2}--\eqref{LBC4bis}; 
the key point of the proof is to show that the boundary conditions \eqref{LBC4}--\eqref{LBC4bis} are weakly dissipative in the sense discussed above, 
and under the subcriticality of the flow.

\begin{proposition}\label{prop:BEE}
Under Assumption \ref{ass:LP}, there exists a positive constant $C_0=C(c_0,M_0)$ such that any regular solution 
$\check{u}=(\check{\zeta},\check{v}), \check{\psi}_{\rm i}$ to \eqref{LSWEe2}--\eqref{LBC4bis} satisfies, 
with $\check{\phi}_{\rm i}$ and $\check{v}_{\rm i}$ given by \eqref{LSWEi2}, 
\begin{align*}
\|\check{u}(t)\|_{L^2(\ul{\cE})}^2 + \|\check{v}_{\rm i}(t)\|_{L^2(\ul{\cI})}^2
&\leq C_0 e^{C_0M_1}\bigl\{ \|\check{u}(0)\|_{L^2(\ul{\cE})}^2 + \|(\check{v}_{\rm i},f_{{\rm i},2},f_{{\rm i},3},f_{{\rm i},4})(0)\|_{L^2(\ul{\cI})}^2  \\
&\quad\;
 + \left( \int_0^t \bigl( \|(f_1,f_2)(t')\|_{L^2(\ul{\cE})}
 + \|\bm{\partial}(f_{{\rm i},2},f_{{\rm i},3},f_{{\rm i},4},f_{{\rm i},5})(t')\|_{L^2(\ul{\cI})} \bigr){\rm d}t' \right)^2\\
&\quad\;
  + \int_0^t\|(\dt\check{\phi}_{\rm i},f_{i,5})(t')\|_{L^2(\ul{\cI})}\|f_{{\rm i},1}(t')\|_{L^2(\ul{\cI})}{\rm d}t' \bigr\}
\end{align*}
for any $t\in[0,T]$. 
\end{proposition}

\begin{remark} \ 
{\bf i.} \ 
In the proof of this proposition, we do not use the last equation in \eqref{LSWEe2} nor the second boundary condition in \eqref{LBC4}. 
These will be used in the next section to derive an additional boundary regularity. 

{\bf ii.} \ 
The norm for $(f_1,f_2)$ and $\bm{\partial}(f_{{\rm i},2},f_{{\rm i},3},f_{{\rm i},4},f_{{\rm i},5})$ in the energy estimate is $L^1$ with respect to time. 
This is important for application to the nonlinear problem. 
In fact, if we replace the $L^1$-norm with $L^2$-norm, then a difficulty would arise. 

{\bf iii.} \ 
In the case $f_{{\rm i},1}\ne0$, there is a term $\|\dt\check{\phi}_{\rm i}\|_{L^2(\ul{\cI})}$ in the right-hand side of the above energy estimate. 
However, in the applications to the nonlinear problem, this term does not cause any difficulties, 
because the order of derivatives of $\dt\check{\phi}_{\rm i}$ is the same as that of $\check{v}_{\rm i}$. 
See also Remark \ref{re:GU2} in Section \ref{sectnearbdry}. 
\end{remark}

\begin{proof}[Proof of Proposition \ref{prop:BEE}]
Let $\check{u}=(\check{\zeta},\check{v}), \check{\psi}_{\rm i}$ be a regular solution to \eqref{LSWEe2}--\eqref{LBC4bis}, 
and let $\check{\phi}_{\rm i}$ and $\check{v}_{\rm i}$ be given by \eqref{LSWEi2}. 
For the sake of clarity, let us first prove the proposition in the case where $f_{{\rm i},j}=0$ for $1\leq j\leq 5$ 
in order to make the mechanism for weak dissipativity more apparent. 
Under such an assumption, one can use the boundary condition \eqref{LBC4} to write the boundary integral in \eqref{weakdissip0} as 
\begin{equation}\label{weakdissip1}
\int_{\ul{\itGamma}} (\gr \check{\zeta}+w\cdot\check{v})N^\varphi\cdot (h \check{v}+w\check{\zeta})
 = -\int_{\ul{\itGamma}} \dt \check{\psi}_{\rm i} N^\varphi\cdot (h_{\rm i}\check{v}_{\rm i}).
\end{equation}
In order to control the right-hand side, let us remark that 
\begin{align*}
\dt (\tfrac{1}{2}Jh_{\rm i}\abs{\check{v}_{\rm i}}^2)
&= \tfrac{1}{2}\dt( J h_{\rm i}) \abs{\check{v}_{\rm i}}^2+J h_{\rm i}\check{v}_{\rm i}\cdot \dt \check{v}_{\rm i} \\
&= \tfrac{1}{2}\dt( J h_{\rm i}) \abs{\check{v}_{\rm i}}^2+J h_{\rm i}\check{v}_{\rm i}\cdot \nabla^\varphi \dt \check{\phi}_{\rm i}
 + J h_{\rm i}\check{v}_{\rm i}\cdot [\dt,\nabla^\varphi] \check{\phi}_{\rm i},
\end{align*}
where we used the second equation in \eqref{LSWEi2} to derive the second equality. 
Using now the first equation in \eqref{LSWEi2} as well as the observation that $[\dt,\nabla^\varphi]=-(\partial^\varphi\dt\varphi)^\mathrm{T}\nabla^\varphi$, 
we deduce that 
\begin{equation}\label{weakdissip2}
\dt(J(\tfrac12h_{\rm i}|\check{v}_{\rm i}|^2 )
 + J\nabla^\varphi\cdot((-\dt\check{\phi}_{\rm i})h_{\rm i}\check{v}_{\rm i}) = F_{{\rm i},1} \quad\mbox{in}\quad (0,T)\times\ul{\cI}
\end{equation}
with 
\[
F_{{\rm i},1} = \tfrac12(\dt(Jh_{\rm i}))|\check{v}_{\rm i}|^2 - J(h_{\rm i}\check{v}_{\rm i})\cdot  (\partial^\varphi\dt\varphi)^\mathrm{T}\check{v}_{\rm i}.
\]
Integrating this conservation law over $\ul{\cI}$, and using the fact that $\check{\phi}_{\rm i}=\check{\psi}_{\rm i}$ on $\ul{\itGamma}$, yields 
\[
\frac{{\rm d}}{{\rm d}t} \int_{\ul{\cI}} \tfrac12 J h_{\rm i}\abs{\check{v}_{\rm i}}^2
 - \int_{\ul{\itGamma}} \dt \check{\psi}_{\rm i} N^\varphi\cdot (h_{\rm i}\check{v}_{\rm i}) = \int_{\ul\cI}F_{\rm i,1}.
\]
Plugging this into \eqref{weakdissip1} and integrating it in time, we see that the boundary term in \eqref{weakdissip0} satisfies 
\[
\int_0^t \int_{\ul\itGamma} (\gr\check{\zeta}+w\cdot\check{v}) N^\varphi\cdot (h\check{v}+w\check{\zeta})
 = - \bigl(\int_{\ul{\cI}} h_{\rm i}\abs{\check{v}_{\rm i}}^2\bigr)(t)
  + \bigl(\int_{\ul{\cI}} h_{\rm i}\abs{\check{v}_{\rm i}}^2\bigr)(0)+\int_0^t \int_{\ul{\cI}}F_{\rm i,1};
\]
it is therefore non-positive, up to terms that depend only on the data of the problem and lower order terms. 
This is the weak dissipative property that allows us to derive the energy estimate of the proposition.

We now derive the energy estimate of the proposition with full details and without assuming that $f_{{\rm i},j}=0$ for $1\leq j\leq 5$. 
In the above argument, that is, without these terms, a key point was that the boundary integral obtained when integrating \eqref{localCE} over $\ul{\cE}$ 
corresponds exactly to the boundary integral obtained when integrating \eqref{weakdissip2} over $\ul{\cI}$, 
thanks to the boundary conditions \eqref{LBC4}--\eqref{LBC4bis}. 
Since now $f_{{\rm i},4}$ may differ from $f_{{\rm i},2}$ on $\ul{\itGamma}$, this is no longer the case. 
However, by using a standard smooth extension operator which maps a function defined in the interior domain $\ul{\cI}$ into that in the whole space $\R^2$, 
we can construct a function $f_4$ defined in $(0,T)\times\ul{\cE}$ satisfying 
\begin{equation}\label{Ext1}
\begin{cases}
 f_4 = f_{{\rm i},2}-f_{{\rm i},4} &\mbox{on}\quad (0,T)\times\ul{\itGamma}, \\
 \|\dt^jf_4(t)\|_{H^k(\ul{\cE})} \lesssim \|\dt^j(f_{{\rm i},2}-f_{{\rm i},4})(t)\|_{H^k(\ul{\cI})} &\mbox{for}\quad t\in[0,T], \ j,k=0,1.
\end{cases}
\end{equation}
Then, the first boundary condition in \eqref{LBC4} can be written as 
\begin{equation}\label{LBC5}
N^\varphi\cdot(h\check{v}+w\check{\zeta}+f_4) = N^\varphi\cdot(h_{\rm i}\check{v}_{\rm i}+f_{{\rm i},2}) \quad\mbox{on}\quad (0,T)\times\ul{\itGamma}
\end{equation}
and we can also modify the conservation laws \eqref{localCE} and \eqref{weakdissip2} by adding lower order terms, 
so that \eqref{LBC5} can be used to match the boundary contributions. 
More precisely, by \eqref{LSWEe2} and \eqref{LSWEi2}, we have 
\[
\begin{cases}
 \dt(J(\frac12\check{u}\cdot\Sigma(u)\check{u}+\check{v}\cdot f_4))
  + J\nabla^\varphi\cdot((\gr\check{\zeta}+w\cdot\check{v})(h\check{v}+w\check{\zeta}+f_4)) = F_{{\rm e},1} &\mbox{in}\quad (0,T)\times\ul{\cE}, \\
 \dt(J(\frac12h_{\rm i}|\check{v}_{\rm i}|^2 + \check{v}_{\rm i}\cdot f_{{\rm i},2}))
  + J\nabla^\varphi\cdot((f_{{\rm i},5}-\dt\check{\phi}_{\rm i})(h_{\rm i}\check{v}_{\rm i}+f_{{\rm i},2})) = F_{{\rm i},1} &\mbox{in}\quad (0,T)\times\ul{\cI},
\end{cases}
\]
where $F_{{\rm e},1}$ and $F_{{\rm i},1}$ are collections of lower order terms and given by 
\begin{align*}
F_{{\rm e},1}
&= \tfrac12\check{u}\cdot(\dt(J\Sigma(u)))\check{u} + \check{v}\cdot\dt(Jf_4)
 + J\Sigma(u)\check{u}\cdot\bigl( \begin{smallmatrix} f_1+\nabla^\varphi\cdot f_4 \\ f_2 \end{smallmatrix} \bigr) + Jf_2\cdot f_4, \\
F_{{\rm i},1}
&= \tfrac12(\dt(Jh_{\rm i}))|\check{v}_{\rm i}|^2 + \check{v}_{\rm i}\cdot\dt(Jf_{{\rm i},2})
 + J(f_{{\rm i},5}-\dt\check{\phi}_{\rm i})f_{{\rm i},1} \\
&\quad\;
 + J(h_{\rm i}\check{v}_{\rm i}+f_{{\rm i},2}))\cdot(\nabla^\varphi f_{{\rm i},5}+\dt f_{{\rm i},3}
 + (\partial^\varphi\dt\varphi)^\mathrm{T}(f_{{\rm i},3}-\check{v}_{\rm i}) ), 
\end{align*}
and we compensated the densities of the energy by adding the terms $J\check{v}\cdot f_4$ and $J\check{v}_{\rm i}\cdot f_{{\rm i},2}$ to control boundary terms. 
Here, in the calculation of $F_{{\rm i},1}$ we used the identity $[\dt,\nabla^\varphi]=-(\partial^\varphi\dt\varphi)^\mathrm{T}\nabla^\varphi$. 
In view of these equations, we define an energy function $\mathscr{E}(t)=\mathscr{E}_{\rm e}(t)+\mathscr{E}_{\rm i}(t)$ by 
\begin{align*}
\mathscr{E}_{\rm e}(t) 
&= \int_{\ul{\cE}}\{ J(\tfrac12\check{u}\cdot\Sigma(u)\check{u}+\check{v}\cdot f_4) + \lambda_0|f_4|^2 \}, \\
\mathscr{E}_{\rm i}(t)
&= \int_{\ul{\cI}}\{ J(\tfrac12h_{\rm i}|\check{v}_{\rm i}|^2 + \check{v}_{\rm i}\cdot f_{{\rm i},2}) + \lambda_0( |f_{{\rm i},2}|^2+|f_{{\rm i},3}|^2) \},
\end{align*}
where $\lambda_0$ is a large positive constant depending on $c_0$ and $M_0$ such that the energy function $\mathscr{E}(t)$ is equivalent to 
\[
E(t)=\|(\check{u},f_4)(t)\|_{L^2(\ul{\cE})}^2+\|(\check{v}_{\rm i},f_{{\rm i},2},f_{{\rm i},3})(t)\|_{L^2(\ul{\cE})}^2.
\]
Such a choice of $\lambda_0$ is possible thanks to Assumption \ref{ass:LP} (i)--(iii). 
By the relation \eqref{DivTh} we obtain 
\begin{equation}\label{BEI}
\begin{cases}
 \frac{\rm d}{{\rm d}t}\mathscr{E}_{\rm e}(t) = \int_{\ul{\itGamma}}(\gr\check{\zeta}+w\cdot\check{v})N^\varphi\cdot(h\check{v}+w\check{\zeta}+f_4)
  + \int_{\ul{\cE}}F_{{\rm e},2}, \\
 \frac{\rm d}{{\rm d}t}\mathscr{E}_{\rm i}(t) = -\int_{\ul{\itGamma}}(f_{{\rm i},5}-\dt\check{\phi}_{\rm i})N^\varphi\cdot(h_{\rm i}\check{v}_{\rm i}+f_{{\rm i},2})
  + \int_{\ul{\cI}}F_{{\rm i},2},
\end{cases}
\end{equation}
where $F_{{\rm e},2}=F_{{\rm e},1}+2\lambda_0f_4\cdot\dt f_4$ and 
$F_{{\rm i},2}=F_{{\rm i},1}+2\lambda_0( f_{{\rm i},2}\cdot\dt f_{{\rm i},2} + f_{{\rm i},3}\cdot\dt f_{{\rm i},3} )$. 
Adding these two identities we obtain $\frac{\rm d}{{\rm d}t}\mathscr{E}(t)=\int_{\ul{\cE}}F_{{\rm e},2}+\int_{\ul{\cI}}F_{{\rm i},2}$, 
because the boundary terms are cancelled due to the boundary conditions \eqref{LBC4bis} and \eqref{LBC5}. 
Let $a(t)$ be a non-negative function, which will be determined later, and put $A(t)=\int_0^t a(t')\mathrm{d}t'$. 
Then, we have 
\[
\frac{\rm d}{{\rm d}t}\{ \mathscr{E}(t)e^{-A(t)} \} + a(t)\mathscr{E}(t)e^{-A(t)}
= \left( \int_{\ul{\cE}}F_{{\rm e},2}+\int_{\ul{\cI}}F_{{\rm i},2} \right) e^{-A(t)},
\]
so that 
\[
\sup_{0\leq t'\leq t}E(t')e^{-A(t')} + \int_0^t a(t')E(t')e^{-A(t')}{\rm d}t'
\lesssim E(0) + \int_0^t \left( \int_{\ul{\cE}}|F_{{\rm e},2}|+\int_{\ul{\cI}}|F_{{\rm i},2}| \right) e^{-A(t')} {\rm d}t'.
\]
We can easily evaluate the term $\int_{\ul{\cE}}|F_{{\rm e},2}|$ and $\int_{\ul{\cI}}|F_{{\rm i},2}|$ and obtain 
\[
\int_{\ul{\cE}}|F_{{\rm e},2}|+\int_{\ul{\cI}}|F_{{\rm i},2}| 
\lesssim \mathscr{N}(t)E(t)+\sqrt{E(t)}\mathscr{F}_1(t)+\mathscr{F}_2(t),
\]
where 
\begin{align*}
\mathscr{N}(t)
&= \|\dt u(t)\|_{L^\infty(\ul{\cE})}+\|\dt h_{\rm i}(t)\|_{L^\infty(\ul{\cI})}+\|\dt\partial\varphi(t)\|_{L^\infty(\R^2)}, \\
\mathscr{F}_1(t)
&= \|(f_1,f_2)(t)\|_{L^2(\ul{\cE})} + \|\bm{\partial}(f_{{\rm i},2},f_{{\rm i},3},f_{{\rm i},4},f_{{\rm i},5})(t)\|_{L^2(\ul{\cI})}, \\
\mathscr{F}_2(t)
&= \|(\dt\check{\phi}_{\rm i},f_{i,5})\|_{L^2(\cI)}\|f_{{\rm i},1}\|_{L^2(\cI)}.
\end{align*}
Now, we choose the function $a(t)$ as $a(t)=C_0\mathscr{N}(t)$ with a sufficiently large positive constant $C_0$ depending on $c_0$ and $M_0$. 
Then, the term related to $\mathscr{N}(t)E(t)$ in the right-hand side can be absorbed in the left-hand side. 
As for the term related to $\sqrt{E(t)}\mathscr{F}_1(t)$, we see that 
\begin{align*}
\int_0^t \sqrt{E(t')}\mathscr{F}_1(t')e^{-A(t')} {\rm d}t'
&\leq \left( \sup_{0\leq t'\leq t} E(t')e^{-A(t')} \right)^{1/2}\int_0^t \mathscr{F}_1(t')e^{-\frac12A(t')}{\rm d}t' \\
&\leq \epsilon \sup_{0\leq t'\leq t} E(t')e^{-A(t')} + \frac{1}{4\epsilon}\left( \int_0^t \mathscr{F}_1(t')e^{-\frac12A(t')}{\rm d}t' \right)^2
\end{align*}
for any $\epsilon>0$.
By choosing $\epsilon$ sufficiently small, then the first term in the right-hand side of the above equation can also be absorbed in the left-hand side. 
Therefore, we obtain 
\begin{align*}
E(t) &\lesssim e^{A(t)} \bigl\{ E(0) + \left( \int_0^t \mathscr{F}_1(t')e^{-\frac12A(t')}{\rm d}t' \right)^2
 + \int_0^t \mathscr{F}_2(t')e^{-A(t')} {\rm d}t' \bigr\}.
\end{align*}
Moreover, thanks to Assumption \ref{ass:LP} (iv), we have $0\leq A(t)\leq C_0M_1$. 
These estimates together with \eqref{Ext1} to evaluate $f_4$ give the desired energy estimate. 
\end{proof}

%----------------------------------------------------------------------------------------------------------------------
\section{Additional boundary regularity for the linearized problem}\label{sect:AddBR}
In this section, we continue to consider the linearized problem \eqref{LSWEe2}--\eqref{LBC4bis}. 
In view of \eqref{EqPhibis}, which is essentially the equation for the variation of the unknown curve $\itGamma(t)$, 
we see easily the importance of obtaining a boundary regularity for $\check{\zeta}$ on $\ul{\itGamma}$. 
Contrary to the case in horizontal dimension $d=1$ analyzed in \cite{IguchiLannes2021}, 
in the case $d=2$ the boundary conditions \eqref{LBC4}--\eqref{LBC4bis} are not strictly dissipative but only weakly dissipative 
in the sense of \cite{IguchiLannes2023}. 
As a result, one cannot control the boundary integral $\int_0^T|\check{u}(t)|_{L^2(\ul{\itGamma})}^2 {\rm d}t$ in terms of the natural energy function $E(t)$ 
used in the previous section by the general theory such as the trace theorem or more sharper estimate in \cite{Tataru1998}. 
Nevertheless, one can actually control the component $\int_0^T|\check{\zeta}(t)|_{L^2(\ul{\itGamma})}^2 {\rm d}t$ by taking a better advantage of the structure 
of the boundary conditions \eqref{LBC4} and obtain the following proposition.

\begin{proposition}\label{prop:ABR}
Under Assumption \ref{ass:LP}, there exists a positive constant $C_0=C(c_0,M_0)$ such that any regular solution 
$\check{u}=(\check{\zeta},\check{v}), \check{v}_{\rm i}, \check{\phi}_{\rm i}$ to \eqref{LSWEe2}--\eqref{LSWEi2} satisfies 
\begin{align*}
\int_0^t|\check{\zeta}(t')|_{L^2(\ul{\itGamma})}^2\mathrm{d}t'
&\leq C_0 \bigl\{ (1+t+M_1)\sup_{0\leq t'\leq t}\bigl( \|\check{u}(t')\|_{L^2(\ul{\cE})}^2
  + \|(\check{v}_{\rm i},f_{{\rm i},2},f_{{\rm i},4})(t')\|_{L^2(\ul{\cI})}^2 \bigr) \\
&\quad\;
 + \left( \int_0^t \bigl( \|(f_1,f_2,f_3)(t')\|_{L^2(\ul{\cE})}
  + \|(f_{{\rm i},1},\bm{\partial}(f_{{\rm i},2},f_{{\rm i},3},f_{{\rm i},4}))(t')\|_{L^2(\ul{\cI})}  \bigr){\rm d}t' \right)^2 \bigr\}
\end{align*}
for any $t\in[0,T]$. 
\end{proposition}

\begin{remark}
In the proof of this proposition, we do not use the evolution equation \eqref{LBC4bis} for $\check{\psi}_{\rm i}$, 
which was used to derive the basic energy estimate in Proposition \ref{prop:BEE}. 
\end{remark}

By this proposition together with the energy estimate in Proposition \ref{prop:BEE}, we can control the boundary integral 
$|\check{\zeta}|_{L^2((0,T)\times\ul{\itGamma})}$. 
Such an additional boundary regularity estimate is crucial to control the unknown curve $\itGamma(t)$ in the nonlinear problem. 
In the rest of this section, 
we will prove this Proposition \ref{prop:ABR} by using characteristic fields related to eigenvalues of the boundary matrix for the exterior problem \eqref{LSWEe2} 
together with Rellich type identities for solutions to the interior problem \eqref{LSWEi2}. 
For the sake of clarity, the main ideas of the proof are first explained on a simple toy model.

%-----------------------------------------------------------
\subsection{The trace estimate for a toy model}
We prove here the new trace estimate of Proposition \ref{prop:ABR} on a much simpler model, in order to better expose the main ideas. 
The toy model consists in taking $w=0$, $h=h_{\rm i}=h_0$, $f_j=0$ ($1\leq j\leq 3$), $f_{\rm i,2}=f_{\rm i,4}=0$, 
and $\varphi={\rm Id}$ in \eqref{LSWEe2}--\eqref{LBC4}; 
moreover, we consider a simpler geometry in which $\ul{\mathcal E}=\R^2_+=\R_+\times \R$ and $\ul{\mathcal I}=\R^2_-=\R_-\times \R$, 
so that $\ul{\itGamma}=\{x_1=0\}$. 
The toy model we consider in this section is therefore 
\begin{equation}\label{toy1}
\begin{cases}
 \dt \check{\zeta}+h_0\nabla\cdot \check{v}=0 &\mbox{in}\quad (0,T)\times \R^2_+, \\
 \dt \check{v}+\gr \nabla \check{\zeta}=0 &\mbox{in}\quad (0,T)\times \R^2_+, \\
 \nabla^\perp\cdot\check{v}=0 &\mbox{in}\quad (0,T)\times \R^2_+, 
\end{cases}
\end{equation}
under the boundary condition 
\begin{equation}\label{toy2}
\check{v}=\check{v_{\rm i}} \quad\mbox{ on }\quad (0,T)\times \{x_1=0\},
\end{equation}
where $\check{v}_{\rm i}$ solves 
\begin{equation}\label{toy3}
\begin{cases}
 \nabla\cdot \check{v}_{\rm i}=g_{\rm i,1} &\mbox{in}\quad (0,T)\times \R^2_-, \\
 \nabla^\perp\cdot \check{v}_{\rm i}=g_{\rm i,2}  &\mbox{in}\quad (0,T)\times \R^2_-.
\end{cases}
\end{equation}
Here, $g_{\rm i,1} $ and $g_{\rm i,2} $ are given functions. 
The equivalent of Proposition \ref{prop:ABR} for this toy model is stated in the proposition below.

\begin{proposition}\label{prop:ABRtoy}
Any regular solution $(\check{\zeta},\check{v}, \check{v}_{\rm i})$ to \eqref{toy1}--\eqref{toy3} satisfies 
\begin{align*}
\int_0^t|\check{\zeta}(t')|_{L^2(\{x_1=0\})}^2\mathrm{d}t'
&\leq C \bigl\{  (1+t)\sup_{0\leq t'\leq t}\bigl( \|(\check{\zeta},\check{v})(t')\|_{L^2({\R^2_+})}^2
  + \|\check{v}_{\rm i}(t')\|_{L^2(\R^2_-)}^2 \bigr) \\
&\quad\;
 + \left( \int_0^t
   \|(g_{{\rm i},1},g_{{\rm i},2})(t')\|_{L^2(\R^2_-)} {\rm d}t' \right)^2 \bigr\}
\end{align*}
for any $t\in[0,T]$. 
\end{proposition}

\begin{proof}
{\bf Step 1.} Characteristic fields. 
The boundary matrix for the hyperbolic system in \eqref{toy1} is given by 
$\begin{pmatrix} 0 & h_0\bf{e}_1^{\rm T} \\ \gr\bf{e}_1 & O\end{pmatrix}$, whose eigenvalues are $0, \pm\sqrt{\gr h_0}$. 
Denoting $\check{v}=(\check{v}_1,\check{v}_2)^{\rm T}$, we put $\check{\alpha}=\sqrt{\gr h_0}\check{\zeta}+h_0\check{v}_1$ and $\check{\beta} = h_0\check{v}_2$, 
which are the characteristic fields of the boundary matrix related to the eigenvalues $\sqrt{\gr h_0}$ and $0$, respectively. 
Then, we see readily that 
\[
\begin{cases}
 \dt \check{\alpha}+\sqrt{\gr h_0}\partial_1 \check{\alpha}+\sqrt{\gr h_0}\partial_2 \check{v}_2=0, \\
 \dt \check{\beta}+\gr h_0\partial_2 \zeta=0,
\end{cases}
\]
so that 
$\tfrac12\dt(\check{\alpha}^2+\check{\beta}^2)+\tfrac12\sqrt{\gr h_0}\partial_1(\check{\alpha}^2-\check{\beta}^2)+\gr h_0^2\partial_2(\check{\zeta}\check{v}_2)=0$. 
Therefore, introducing $\widetilde{\mathscr E}_{\rm e}(t)=\frac{1}{2}\int_{\R^2_+} (\check{\alpha}^2+\check{\beta}^2)$, we deduce that 
\[
\frac{\rm d}{{\rm d}t}\widetilde{\mathscr E}_{\rm e}(t)-\int_{\{x_1=0\}}\tfrac12\sqrt{\gr h_0}(\check{\alpha}^2-\check{\beta}^2)=0.
\]

\noindent
{\bf Step 2.} Boundary integrals. 
Denoting ${\mathscr E}_{\rm e}(t)=\frac{1}{2}\int_{\R_+^2} (\gr \check{\zeta}^2 +h_0 |\check{v}|^2)$, 
the standard energy estimate for the hyperbolic system in \eqref{toy1} yields 
\[
\frac{\rm d}{{\rm d}t} {\mathscr E}_{\rm e}(t) -  \int_{\{x_1=0\}}{\gr h_0}\check{\zeta}\check{v}_1 = 0.
\]
Remarking further that 
\[
\tfrac{1}{2}\sqrt{\gr h_0}(\check{\alpha}^2-\check{\beta}^2)
 = \tfrac{1}{2}(\gr h_0)^{3/2} \check{\zeta}^2 + \tfrac{1}{2}\sqrt{\gr h_0}\big( (h_0 \check{v}_1)^2-(h_0 \check{v}_2)^2\big)+\gr h_0 \check{\zeta}\check{v}_1,
\]
we have 
\[
\tfrac{1}{2}(\gr h_0)^{3/2} \int_{\{x_1=0\}} \check{\zeta}^2=\frac{\rm d}{{\rm d}t}\bigl( \widetilde{\mathscr E}_{\rm e}(t) - {\mathscr E}_{\rm e}(t) \bigr)
+\tfrac{1}{2}\sqrt{\gr h_0}h_0^2  \int_{\{x_1=0\}} \bigl( \check{v}_2^2-\check{v}_1^2 \bigr).
\]

\noindent
{\bf Step 3.} Rellich type identity. 
Using the boundary condition \eqref{toy2} and Green's identity, we can rewrite the last boundary integral as 
\begin{align*}
\int_{\{x_1=0\}}\big(  \check{v}_2^2- \check{v}_1^2\big)
&= \int_{\{x_1=0\}}\big(  \check{v}_{\rm i,2}^2- \check{v}_{\rm i,1}^2\big)\\
&= \int_{\R^2_-}\nabla\cdot \big[\big(  \check{v}_{\rm i,2}^2- \check{v}_{\rm i,1}^2\big) {\bf e}_1- 2 \check{v}_{\rm i,1}\check{v}_{\rm i,2}{\bf e}_2\big].
\end{align*}
Here, one see easily that 
$\frac{1}{2}\nabla\cdot \big[\big( \check{v}_{\rm i,2}^2-\check{v}_{\rm i,1}^2\big) {\bf e}_1- 2 \check{v}_{\rm i,1}\check{v}_{\rm i,2}{\bf e}_2\big]
=-(\nabla\cdot \check{v}_{\rm i})\check{v}_{\rm i,1}+(\nabla^\perp \cdot \check{v}_{\rm i})\check{v}_{\rm i,2}$. 
Therefore, by \eqref{toy3} we get 
\[
\tfrac{1}{2}\sqrt{\gr h_0}h_0^2  \int_{\{x_1=0\}} \big( \check{v}_2^2-\check{v}_1^2\big)=\int_{\R^2_-}G_{\rm i}
\]
with $G_{\rm i}=\sqrt{\gr h_0}h_0^2  (-g_{\rm i,1}\check{v}_{\rm i,1}+g_{\rm i,2}\check{v}_{\rm i,2}\big)$.

\noindent
{\bf Step 4} Conclusion. 
Integrating the identity obtained in Step 2 and using the Rellich type identity obtained in Step 3, we get 
\[
\tfrac{1}{2}(\gr h_0)^{3/2}\int_0^t \vert \check{\zeta}(t')\vert^2_{L^2(\{x_1=0\})}{\rm d}t'
\leq \widetilde{\mathscr E}_{\rm e}(t) + {\mathscr E}_{\rm e}(0)+\int_0^t \Vert G_{\rm i}(t')\Vert_{L^1(\R^2_-)}{\rm d}t',
\]
from which the estimate of the proposition follows directly. 
\end{proof}

%-----------------------------------------------------------
\subsection{Characteristic fields}
This section corresponds to Step 1 of the proof of Proposition \ref{prop:ABRtoy} for the original system \eqref{LSWEe2}. 
The difference is that one has to handle the non trivial background $(h,v)$ and the lower order terms, to track the dependance on the diffeomorphism $\varphi$, 
and to take into account the influence of the curvature of the boundary $\ul{\itGamma}$. 
By \eqref{LSWEe2}, we have 
\begin{equation}\label{LSWEe5}
\begin{cases}
 \dt\check{\zeta}+w\cdot\nabla^\varphi\check{\zeta}+h\nabla^\varphi\cdot\check{v} = f_1-\nabla^\varphi h\cdot\check{v}-(\nabla^\varphi\cdot w)\check{\zeta}
  &\mbox{in}\quad (0,T)\times\ul{\cE}, \\
 \dt\check{v}+(w\cdot\nabla^\varphi)\check{v}+\gr\nabla^\varphi\check{\zeta} = f_2 + f_3w^\perp-(\partial^\varphi w)^\mathrm{T}\check{v}
  &\mbox{in}\quad (0,T)\times\ul{\cE}. \\
\end{cases}
\end{equation}
Here, we used the last equation in \eqref{LSWEe2} to derive the second equation. 
The corresponding boundary matrix is given by $J^{-1}|N^\varphi| M$, where 
\[
M = 
\begin{pmatrix}
 n^\varphi\cdot w & h(n^\varphi)^\mathrm{T} \\
 \gr n^\varphi & (n^\varphi\cdot w)\mathrm{Id}_{2\times2}
\end{pmatrix}
\]
with $n^\varphi = \frac{N^\varphi}{|N^\varphi|}$. 
The eigenvalues of this matrix $M$ are $\pm\lambda_{\pm}=n^\varphi\cdot w \pm \sqrt{\gr h}$ and $\lambda_0=n^\varphi\cdot w$, and related characteristic fields are 
$\sqrt{\gr h}\check{\zeta}\pm h(n^\varphi\cdot\check{v})$ and $h(n^\varphi)^\perp\cdot\check{v}$, respectively. 
We note that under Assumption \ref{ass:LP} (i) it holds that $\lambda_+\geq\lambda_-\gtrsim1$. 
It is natural to expect that these characteristic fields would be useful, so that we introduce the following quantities: 
we first extend the unit outward normal vector $\ul{N}$ to a vector field in $C_0^\infty(\R^2)$ and take a cut-off function $\chi_{\rm b}\in C_0^\infty(\R^2)$ 
such that $\chi_{\rm b}=1$ near the boundary $\ul{\itGamma}$ and $|\ul{N}|\geq\frac12$ on $\operatorname{supp}\chi_{\rm b}$. 
Using this extended vector field $\ul{N}$, 
we can extend naturally $N^\varphi=J((\partial\varphi)^{-1})^\mathrm{T}\ul{N}$ to a vector field in the whole space $\R^2$; see \eqref{Nphi}. 
We note that under Assumption \ref{ass:LP} (iii) it holds that $|N^\varphi| \gtrsim 1$ on $\operatorname{supp}\chi_{\rm b}$ so that 
$\chi_{\rm b}n^\varphi$ is defined as a vector field in the whole space $\R^2$. 
As in the previous section, we use again the function $f_4$ defined in $(0,T)\times\ul{\cE}$ satisfying \eqref{Ext1}. 
Then, we introduce new quantities $\check{\alpha}$ and $\check{\beta}$ by 
\begin{equation}\label{defCF}
\begin{cases}
 \check{\alpha}=\chi_{\rm b}( \sqrt{\gr h}\check{\zeta} + n^\varphi\cdot(h\check{v}+f_4) ), \\
 \check{\beta}=\chi_{\rm b}(n^\varphi)^\perp\cdot(h\check{v}),
\end{cases}
\end{equation}
where we compensated the characteristic field $\check{\alpha}$ by adding the term $n^\varphi\cdot f_4$ in order to make the following calculations simpler.

\begin{remark}
In the above definition of $\check{\alpha}$, we used a characteristic field related to the eigenvalue $\lambda_+$. 
Instead of this, we may use a characteristic field related to the eigenvalue $\lambda_-$. 
Even if we use both of them, one cannot obtain any further estimates. 
\end{remark}

It follows from \eqref{LSWEe5} that 
\[
\begin{cases}
 (\dt + w\cdot\nabla^\varphi)\check{\alpha} + \nabla^\varphi\cdot(
  \chi_\mathrm{b}(\sqrt{\gr h}h\check{v} + \gr h \check{\zeta}n^\varphi) ) = g_1, \\
 (\dt+w\cdot\nabla^\varphi)\check{\beta}
  + \nabla^\varphi\cdot( \chi_\mathrm{b}\gr h\check{\zeta}(n^\varphi)^\perp ) = g_2,
\end{cases}
\]
where 
\begin{align*}
g_1
&= ( (\dt+w\cdot\nabla^\varphi)(\chi_{\rm b}\sqrt{\gr h}) )\check{\zeta} + ( (\dt+w\cdot\nabla^\varphi)(\chi_{\rm b}h n^\varphi) )\cdot\check{v} \\
&\quad\;
 + (\dt+w\cdot\nabla^\varphi)(\chi_{\rm b}n^\varphi\cdot f_4)
 + \nabla^\varphi(\chi_\mathrm{b}\sqrt{\gr h}h)\cdot\check{v} + ( \nabla^\varphi\cdot(\chi_{\rm b}\gr h n^\varphi) )\check{\zeta} \\
&\quad\;
 + \chi_\mathrm{b}\{ \sqrt{\gr h}( f_1-\nabla^\varphi h\cdot\check{v}-(\nabla^\varphi\cdot w)\check{\zeta} )
  + h n^\varphi\cdot( f_2 + f_3w^\perp-(\partial^\varphi w)^\mathrm{T}\check{v} ) \}, \\
g_2
&= ( (\dt+w\cdot\nabla^\varphi)(\chi_{\rm b}h(n^\varphi)^\perp) )\cdot\check{v} + ( (\nabla^\varphi\cdot(\chi_\mathrm{b}\gr h(n^\varphi)^\perp) )\check{\zeta} \\
&\quad\;
 + \chi_\mathrm{b}h(n^\varphi)^\perp\cdot( f_2 + f_3w^\perp-(\partial^\varphi w)^\mathrm{T}\check{v} ).
\end{align*}
These equations imply that 
\begin{equation}\label{EqCF1}
\begin{cases}
\frac12\dt(\check{\alpha}^2) + \nabla^\varphi\cdot( \frac12\check{\alpha}^2w )
 + \check{\alpha}\nabla^\varphi\cdot( \chi_\mathrm{b}(\sqrt{\gr h}h\check{v} + \gr h\check{\zeta}n^\varphi) )
 = \check{\alpha}g_1 + \frac12(\nabla^\varphi\cdot w)\check{\alpha}^2, \\
\frac12\dt(\check{\beta}^2) + \nabla^\varphi\cdot( \frac12\check{\beta}^2w )
 + \check{\beta}\nabla^\varphi\cdot( \chi_\mathrm{b}\gr h\check{\zeta}(n^\varphi)^\perp )
= \check{\beta}g_2 + \frac12(\nabla^\varphi\cdot w)\check{\beta}^2.
\end{cases}
\end{equation}
Here, we see that 
\begin{align*}
& \check{\alpha}\nabla^\varphi\cdot( \chi_\mathrm{b}(\sqrt{\gr h}h\check{v} + \gr h\check{\zeta}n^\varphi) )
 + \check{\beta}\nabla^\varphi\cdot( \chi_\mathrm{b}\gr h(n^\varphi)^\perp\check{\zeta} ) \\
&= \check{\alpha}\nabla^\varphi\cdot( \sqrt{\gr h}( (\check{\alpha}-\chi_{\rm b}n^\varphi\cdot f_4)n^\varphi + \check{\beta}(n^\varphi)^\perp ) ) 
 + \check{\beta}\nabla^\varphi\cdot( \sqrt{\gr h}( \check{\alpha}-\chi_\mathrm{b}n^\varphi\cdot(h\check{v}+f_4) ) (n^\varphi)^\perp ) \\
&= \nabla^\varphi\cdot( \sqrt{\gr h}( \tfrac12\check{\alpha}^2n^\varphi + \check{\alpha}\check{\beta}(n^\varphi)^\perp ) )
 - \check{\beta}\chi_\mathrm{b}\sqrt{\gr h}h n^\varphi\cdot( (n^\varphi)^\perp\cdot\nabla^\varphi )\check{v} + G_1,
\end{align*}
where 
\begin{align*}
G_1 
&= \tfrac12( \nabla^\varphi\cdot(\sqrt{\gr h}n^\varphi) )\check{\alpha}^2 + ( \nabla^\varphi\cdot(\sqrt{\gr h}(n^\varphi)^\perp) )\check{\alpha}\check{\beta}
 - ( \nabla^\varphi\cdot(\chi_{\rm b}\sqrt{\gr h}h n^\varphi\otimes(n^\varphi)^\perp) )\check{\beta}\check{v} \\
&\quad\;
 - \check{\alpha}\nabla^\varphi\cdot( \sqrt{\gr h}\chi_{\rm b}(n^\varphi\cdot f_4)n^\varphi )
 - \check{\beta}\nabla^\varphi\cdot( \sqrt{\gr h}\chi_{\rm b}(n^\varphi\cdot f_4)(n^\varphi)^\perp ). 
\end{align*}
Moreover, by the identity \eqref{F2} together with the last equation in \eqref{LSWEe2} we have 
\[
- \check{\beta}\chi_\mathrm{b}\sqrt{\gr h}h n^\varphi\cdot( (n^\varphi)^\perp\cdot\nabla^\varphi )\check{v}
= -\nabla^\varphi(\tfrac12\sqrt{\gr h}\check{\beta}^2n^\varphi) + G_2,
\]
where 
\[
G_2 = \tfrac12( \nabla^\varphi\cdot(\sqrt{\gr h} n^\varphi) )\check{\beta}^2
 + \check{\beta}\{ \chi_{\rm b}\sqrt{\gr h}hf_3 + \sqrt{\gr h}( (n^\varphi\cdot\nabla^\varphi)( \chi_\mathrm{b}h(n^\varphi)^\perp) )\cdot\check{v} \}.
\]
Therefore, we obtain 
\begin{align}\label{EqCF2}
& \check{\alpha}\nabla^\varphi\cdot( \chi_\mathrm{b}(\sqrt{\gr h}h\check{v} + \gr h\check{\zeta}n^\varphi) )
 + \check{\beta}\nabla^\varphi\cdot( \chi_\mathrm{b}\gr h(n^\varphi)^\perp\check{\zeta} ) \\
&= \nabla^\varphi\cdot( \sqrt{\gr h}( \tfrac12(\check{\alpha}^2-\check{\beta}^2)n^\varphi + \check{\alpha}\check{\beta}(n^\varphi)^\perp ) ) + G_1 + G_2.
 \nonumber
\end{align}
Furthermore, we have 
\begin{equation}\label{EqCF3}
\tfrac12w(\check{\alpha}^2+\check{\beta}^2) + \sqrt{\gr h}( \tfrac12(\check{\alpha}^2-\check{\beta}^2)n^\varphi + \check{\alpha}\check{\beta}(n^\varphi)^\perp )
= \tfrac12(\lambda_+\check{\alpha}^2-\lambda_-\check{\beta}^2)n^\varphi + \check{\mu}(n^\varphi)^\perp,
\end{equation}
where 
\[
\begin{cases}
\lambda_{\pm} = \sqrt{\gr h}\pm n^\varphi\cdot w, \\
\check{\mu} = \tfrac12((n^\varphi)^\perp\cdot w)(\check{\alpha}^2+\check{\beta}^2) + \sqrt{\gr h}\check{\alpha}\check{\beta}.
\end{cases}
\]
Adding two equations in \eqref{EqCF1} and using \eqref{EqCF2} and \eqref{EqCF3}, we obtain 
\begin{equation}\label{EqCF4}
\dt(\tfrac12aJ(\check{\alpha}^2+\check{\beta}^2))
 + J\nabla^\varphi\cdot(a(\tfrac12(\lambda_+\check{\alpha}^2-\lambda_-\check{\beta}^2)n^\varphi + \check{\mu}(n^\varphi)^\perp) ) = G_3,
\end{equation}
where $a$ is an arbitrary function which will be determined later by \eqref{defa} and 
\begin{align*}
G_3
&= aJ( \check{\alpha}g_1+\check{\beta}g_2-(G_1+G_2) ) + \tfrac12( \dt(aJ)+aJ(\nabla^\varphi\cdot w) )( \check{\alpha}^2+\check{\beta}^2 ) \\
&\quad\;
 + J\nabla^\varphi a \cdot ( \tfrac12(\lambda_+\check{\alpha}^2-\lambda_-\check{\beta}^2)n^\varphi + \check{\mu}(n^\varphi)^\perp ).
\end{align*}
In view of this, we define another energy function $\widetilde{\mathscr{E}}_{\rm e}(t)$ by 
\begin{equation}\label{defEF2}
\widetilde{\mathscr{E}}_{\rm e}(t) = \frac12\int_{\ul{\cE}}aJ(\check{\alpha}^2+\check{\beta}^2),
\end{equation}
which can be controlled by the energy function $E(t)$ defined in the previous section. 
Then, by \eqref{EqCF4} we have 
\begin{equation}\label{EqCF5}
\frac{\rm d}{{\rm d}t}\widetilde{\mathscr{E}}_{\rm e}(t) - \frac12\int_{\ul{\itGamma}} a|N^\varphi|( \lambda_+\check{\alpha}^2-\lambda_-\check{\beta}^2 )
= \int_{\ul{\cE}}G_3.
\end{equation}

%-----------------------------------------------------------
\subsection{Boundary integrals}
This section corresponds to Step 2 of the proof of Proposition \ref{prop:ABRtoy}. 
Integrating the identity \eqref{EqCF5} with respect to time $t$, we see that the boundary integral 
$\int_0^t \bigl( \int_{\ul{\itGamma}} a|N^\varphi|( \lambda_+\check{\alpha}^2-\lambda_-\check{\beta}^2 ) \bigr){\rm d}t'$ can be controlled by $E(t)$. 
On the other hand, by the first identity in \eqref{BEI} the boundary integral 
$\int_0^t\bigl( \int_{\ul{\itGamma}}(\gr\check{\zeta}+w\cdot\check{v})N^\varphi\cdot(h\check{v}+w\check{\zeta}+f_4) \bigr){\rm d}t'$ 
has also been controlled by $E(t)$. 
Now, we proceed to relate these two boundary integrals to another one, which is equivalent to 
$\int_0^t \bigl( \int_{\ul{\itGamma}}\check{\zeta}^2 \bigr){\rm d}t'$.

By the definition \eqref{defCF} of the characteristic field $\check{\alpha}$ and $\check{\beta}$, on the boundary $\ul{\itGamma}$ we have 
$\check{\alpha}=\lambda_-\check{\zeta}+n^\varphi\cdot(h\check{v}+w\check{\zeta}+f_4)$, so that 
\begin{align*}
\tfrac12a|N^\varphi|( \lambda_+\check{\alpha}^2-\lambda_-\check{\beta}^2 )
&= a|N^\varphi|\{ \tfrac12\lambda_+\lambda_-^2\check{\zeta}^2 + \lambda_+\lambda_-\check{\zeta}n^\varphi\cdot(h\check{v}+w\check{\zeta}+f_4) \\
&\qquad
 + \tfrac12\lambda_+(n^\varphi\cdot(h\check{v}+w\check{\zeta}+f_4))^2 - \tfrac12\lambda_-((n^\varphi)^\perp\cdot(h\check{v}))^2 \}.
\end{align*}
Similarly, on $\ul{\itGamma}$ we also have 
\begin{align*}
& (\gr\check{\zeta}+w\cdot\check{v})N^\varphi\cdot(h\check{v}+w\check{\zeta}+f_4) \\
&= \tfrac{|N^\varphi|}{h}\{ \lambda_+\lambda_-\check{\zeta}n^\varphi\cdot(h\check{v}+w\check{\zeta}+f_4)
 + \tfrac12(\lambda_+-\lambda_-)(n^\varphi\cdot(h\check{v}+w\check{\zeta}+f_4))^2 \\
&\qquad
 + \bigl( ((n^\varphi)^\perp\cdot w)((n^\varphi)^\perp\cdot(h\check{v}) )-(n^\varphi\cdot w)(n^\varphi\cdot f_4) \bigr)
  n^\varphi\cdot(h\check{v}+w\check{\zeta}+f_4) \}.
\end{align*}
In view of these identities, we define the function $a$ used in the definition \eqref{defEF2} of the energy function $\widetilde{\mathscr{E}}_{\rm e}(t)$ by 
\begin{equation}\label{defa}
a=\frac{1}{h}
\end{equation}
in order to equalize the coefficients of the term $\check{\zeta}n^\varphi\cdot(h\check{v}+w\check{\zeta}+f_4)$ in the right-hand sides of the above two identities. 
In other words, we can cancel out this term from the two boundary integrals. 
In the following, we fix the function $a$ by \eqref{defa}. 
Then, by the above two identities together with the boundary conditions \eqref{LBC4} and \eqref{LBC5} we have 
\begin{align*}
& \tfrac12a|N^\varphi|( \lambda_+\check{\alpha}^2-\lambda_-\check{\beta}^2 )
 - (\gr\check{\zeta}+w\cdot\check{v})N^\varphi\cdot(h\check{v}+w\check{\zeta}+f_4) \\
&= \tfrac{|N^\varphi|}{h}\bigl\{ \tfrac12\lambda_+\lambda_-^2\check{\zeta}^2
 + \tfrac12\lambda_-\bigl( (n^\varphi\cdot(h_{\rm i}\check{v}_{\rm i}+f_{{\rm i},2}))^2 - ((n^\varphi)^\perp\cdot(h_{\rm i}\check{v}_{\rm i}))^2 \bigr) \\
&\qquad
 - \bigl( ((n^\varphi)^\perp\cdot w)((n^\varphi)^\perp\cdot(h_{\rm i}\check{v}_{\rm i}))
  - (n^\varphi\cdot w)(n^\varphi\cdot f_4) \bigr) ( n^\varphi\cdot(h_{\rm i}\check{v}_{\rm i}+f_{{\rm i},2}) ) \bigr\} \\
&= \tfrac{|N^\varphi|}{h}\bigl\{ \tfrac12\lambda_+\lambda_-^2\check{\zeta}^2
 + \tfrac12\lambda_-\bigl( (n^\varphi\cdot(h_{\rm i}\check{v}_{\rm i}+f_{{\rm i},2}))^2
  - ((n^\varphi)^\perp\cdot(h_{\rm i}\check{v}_{\rm i}+f_{{\rm i},2}))^2 \bigr) \\
&\qquad
 -  ((n^\varphi)^\perp\cdot w)( (n^\varphi)^\perp\cdot(h_{\rm i}\check{v}_{\rm i}+f_{{\rm i},2}) )( n^\varphi\cdot(h_{\rm i}\check{v}_{\rm i}+f_{{\rm i},2}) ) \\
&\qquad
 + \lambda_- ( (n^\varphi)^\perp\cdot f_{{\rm i},2} )( (n^\varphi)^\perp\cdot(h_{\rm i}\check{v}_{\rm i}+f_{{\rm i},2}) )
 - \tfrac12\lambda_- ( (n^\varphi)^\perp\cdot f_{{\rm i},2} )^2 \\
&\qquad
 + \bigl( w\cdot f_{{\rm i},2} - (n^\varphi\cdot w)(n^\varphi\cdot f_{{\rm i},4}) \bigr) ( n^\varphi\cdot(h_{\rm i}\check{v}_{\rm i}+f_{{\rm i},2}) ) \bigr\}
\end{align*}
on $\ul{\itGamma}$, where we used the first condition in \eqref{Ext1}. 
Therefore, we obtain 
\begin{equation}\label{EqCF6}
\int_{\ul{\itGamma}}\frac{\lambda_+\lambda_-^2|N^\varphi|}{2h}\check{\zeta}^2
= \frac{\rm d}{{\rm d}t}( \widetilde{\mathscr{E}}_{\rm e}(t)-\mathscr{E}_{\rm e}(t) ) + \int_{\ul{\cE}}(F_{{\rm e},2}-G_3) + I_1(t) + I_2(t) + I_3(t),
\end{equation}
where 
\begin{align*}
I_1(t)
&= \int_{\ul{\itGamma}}\frac{\lambda_-|N^\varphi|}{2h}
 \bigl( ((n^\varphi)^\perp\cdot(h_{\rm i}\check{v}_{\rm i}+f_{{\rm i},2}))^2 - (n^\varphi\cdot(h_{\rm i}\check{v}_{\rm i}+f_{{\rm i},2}))^2 \bigr), \\
I_2(t)
&= \int_{\ul{\itGamma}}\frac{|N^\varphi|}{h}
 ((n^\varphi)^\perp\cdot w)( (n^\varphi)^\perp\cdot(h_{\rm i}\check{v}_{\rm i}+f_{{\rm i},2}) )( n^\varphi\cdot(h_{\rm i}\check{v}_{\rm i}+f_{{\rm i},2}) ), \\
I_3(t)
&= \textcolor{blue}{-}\int_{\ul{\itGamma}}\frac{|N^\varphi|}{h}\bigl\{
 \lambda_- ( (n^\varphi)^\perp\cdot f_{{\rm i},2} )( (n^\varphi)^\perp\cdot(h_{\rm i}\check{v}_{\rm i}+f_{{\rm i},2}) )
  - \tfrac12\lambda_- ( (n^\varphi)^\perp\cdot f_{{\rm i},2} )^2 \\
&\qquad
 + \bigl( w\cdot f_{{\rm i},2} - (n^\varphi\cdot w)(n^\varphi\cdot f_{{\rm i},4}) \bigr) ( n^\varphi\cdot(h_{\rm i}\check{v}_{\rm i}+f_{{\rm i},2}) ) \bigr\}.
\end{align*}

%-----------------------------------------------------------
\subsection{Rellich type identities}
This section corresponds to Step 3  of the proof of Proposition  \ref{prop:ABRtoy}. 
We proceed to evaluate the above boundary integrals $I_1(t),I_2(t)$, and $I_3(t)$. 
In view of the second equation in \eqref{LSWEi2}, $n^\varphi\cdot \check{v}_{\rm i}$ and $(n^\varphi)^\perp\cdot \check{v}_{\rm i}$ are, roughly speaking, 
normal and tangential derivatives of $\check{\phi}_{\rm i}$, which satisfies a second order elliptic equation in the interior domain $\ul{\cI}$. 
Rellich type identities give some relations of boundary integrals for such derivatives. 
We first prepare the following identity. 
The important point here is that the derivative on $q$ in the right-hand side only appear through a dependence on $\nabla^\varphi\cdot q$ 
and $(\nabla^\varphi)^\perp\cdot q$.

\begin{lemma}\label{lem:Rellich1}
For any $\R^2$-valued functions $f$ and $q$, we have 
\begin{align*}
& \nabla^\varphi\cdot\bigl\{ \bigl( (f\cdot q)^2-(f^\perp\cdot q)^2 \bigr)f + 2(f\cdot q)(f^\perp\cdot q)f^\perp \bigr\} \\
&= (\nabla^\varphi\cdot f)\bigl( (f\cdot q)^2-(f^\perp\cdot q)^2 \bigr) + 2(\nabla^\varphi\cdot f^\perp)(f\cdot q)(f^\perp\cdot q) \\
&\quad\;
 + 2|f|^2\bigl\{ \bigl( (q\cdot\nabla^\varphi)f + (q^\perp\cdot\nabla^\varphi)f^\perp \bigr)\cdot q
  + (f\cdot q)\nabla^\varphi\cdot q - (f^\perp\cdot q)(\nabla^\varphi)^\perp\cdot q \bigr\}.
\end{align*}
\end{lemma}

\begin{proof}
By a direct calculation, we have 
\begin{align*}
& \nabla^\varphi\cdot\bigl\{ \bigl( (f\cdot q)^2-(f^\perp\cdot q)^2 \bigr)f + 2(f\cdot q)(f^\perp\cdot q)f^\perp \bigr\} \\
&= (\nabla^\varphi\cdot f)\bigl( (f\cdot q)^2-(f^\perp\cdot q)^2 \bigr) + 2(\nabla^\varphi\cdot f^\perp)(f\cdot q)(f^\perp\cdot q) \\
&\quad\;
 + 2\bigl\{ (f\cdot q)\bigl( (f\cdot\nabla^\varphi)f + (f^\perp\cdot\nabla^\varphi)f^\perp \bigr)
  + (f^\perp\cdot q)\bigl( (f^\perp\cdot\nabla^\varphi)f - (f\cdot\nabla^\varphi)f^\perp \bigr) \bigr\}\cdot q \\
&\quad\;
 + 2(f\cdot q)\bigl( f\cdot(f\cdot\nabla^\varphi)q + f^\perp\cdot(f^\perp\cdot\nabla^\varphi)q \bigr)
 + 2(f^\perp\cdot q)\bigl( f\cdot(f^\perp\cdot\nabla^\varphi)q - f^\perp\cdot(f\cdot\nabla^\varphi)q \bigr). 
\end{align*}
Here, as for the second line in the right-hand side we see that 
\begin{align*}
& (f\cdot q)\bigl( (f\cdot\nabla^\varphi)f + (f^\perp\cdot\nabla^\varphi)f^\perp \bigr)
  + (f^\perp\cdot q)\bigl( (f^\perp\cdot\nabla^\varphi)f - (f\cdot\nabla^\varphi)f^\perp \bigr) \\
&= \bigl( (f\cdot q)(f\cdot\nabla^\varphi) + (f^\perp\cdot q)(f^\perp\cdot\nabla^\varphi) \bigr)f
 + \bigl( (f^\perp\cdot q^\perp)(f^\perp\cdot\nabla^\varphi) + (f\cdot q^\perp)(f\cdot\nabla^\varphi) \bigr)f^\perp \\
&= |f|^2 \bigl( (q\cdot\nabla^\varphi)f + (q^\perp\cdot\nabla^\varphi)f^\perp \bigr). 
\end{align*}
As for the third line, we use the identity \eqref{F2} with $g=f^\perp$ and that obtained by replacing $\nabla^\varphi$ with $(\nabla^\varphi)^\perp$ to obtain 
\[
\begin{cases}
 f\cdot(f\cdot\nabla^\varphi)q + f^\perp\cdot(f^\perp\cdot\nabla^\varphi)q = |f|^2\nabla^\varphi\cdot q, \\
 f\cdot(f^\perp\cdot\nabla^\varphi)q - f^\perp\cdot(f\cdot\nabla^\varphi)q = -|f|^2(\nabla^\varphi)^\perp\cdot q.
\end{cases}
\]
Therefore, we obtain the desired identity. 
\end{proof}

Now, we transform the boundary integrals $I_j(t)$ into integrals over the interior domain $\ul{\cI}$, which can be controlled by the energy function $E(t)$ 
by using the identity in Lemma \ref{lem:Rellich1}. 
To this end, we need to extend the functions $h$ and $w$ defined in the exterior domain $\ul{\cE}$ smoothly to those in the whole space $\R^2$. 
By using a standard smooth extension operator, we can construct functions $h_{{\rm i},*}$ and $w_{\rm i}$ defined in $(0,T)\times\ul{\cI}$ satisfying 
\begin{equation}\label{Ext2}
\begin{cases}
 h_{{\rm i},*}=h, \quad w_{\rm i}=w &\mbox{on}\quad (0,T)\times\ul{\itGamma}, \\
 \inf_{x\in\ul{\cI}}h_{{\rm i},*}(t,x) \gtrsim \inf_{x\in\ul{\cE}}h(t,x) &\mbox{for}\quad t\in[0,T], \\
 \|\dt^jh_{{\rm i},*}(t)\|_{W^{k,\infty}(\ul{\cI})} \lesssim \|\dt^jh(t)\|_{W^{k,\infty}(\ul{\cE})} &\mbox{for}\quad t\in[0,T], \ j,k=0,1, \\
 \|\dt^jw_{\rm i}(t)\|_{W^{k,\infty}(\ul{\cI})} \lesssim \|\dt^jw(t)\|_{W^{k,\infty}(\ul{\cE})} &\mbox{for}\quad t\in[0,T], \ j,k=0,1.
\end{cases}
\end{equation}
Accordingly, we can extend $\lambda_{\pm}$ smoothly by 
\[
\lambda_{{\rm i},\pm} = \chi_{\rm b}( \sqrt{\gr h_{{\rm i},*}}\pm n^\varphi\cdot w_{\rm i} ).
\]
Then, we have the following lemma.

\begin{lemma}\label{lem:Rellich2}
It holds that $I_j(t)=\int_{\ul{\cI}}JG_{{\rm i},j}$ for $j=1,2,3$, where 
\begin{align*}
G_{{\rm i},1}
&= \Bigl( \nabla^\varphi\cdot\Bigl( \chi_{\rm b}\tfrac{\lambda_{{\rm i},-}}{2h_{{\rm i},*}}n^\varphi \Bigr) \Bigr)
 ( ((n^\varphi)^\perp\cdot\check{q}_{\rm i})^2 - (n^\varphi\cdot\check{q}_{\rm i})^2 ) \\
&\quad\;
 - \Bigl( \nabla^\varphi\cdot\Bigl( \chi_{\rm b}\tfrac{\lambda_{{\rm i},-}}{h_{{\rm i},*}}(n^\varphi)^\perp \Bigr) \Bigr)
  ((n^\varphi)^\perp\cdot\check{q}_{\rm i})(n^\varphi\cdot\check{q}_{\rm i}) \\
&\quad\;
 - \chi_{\rm b}\tfrac{\lambda_{{\rm i},-}}{h_{{\rm i},*}}\bigl\{
  \bigl( (\check{q}_{\rm i}\cdot\nabla^\varphi)n^\varphi + (\check{q}_{\rm i}^\perp\cdot\nabla^\varphi)(n^\varphi)^\perp \bigr)\cdot\check{q}_{\rm i} 
 + (n^\varphi\cdot\check{q}_{\rm i})f_{{\rm i},1} \\
&\makebox[3em]{}
 - ((n^\varphi)^\perp\cdot\check{q}_{\rm i})( (\nabla^\varphi h_{\rm i})^\perp\cdot\check{v}_{\rm i} + h_{\rm i}(\nabla^\varphi)^\perp\cdot f_{{\rm i},3}
  + (\nabla^\varphi)^\perp\cdot f_{{\rm i},2} ) \bigr\}, \\
G_{{\rm i},2}
&= \Bigl( \nabla^\varphi\cdot\Bigl( \chi_{\rm b}\tfrac{(n^\varphi)^\perp\cdot w_{\rm i}}{h_{{\rm i},*}}n^\varphi \Bigr) \Bigr)
  ((n^\varphi)^\perp\cdot\check{q}_{\rm i})(n^\varphi\cdot\check{q}_{\rm i}) \\
&\quad\;
 + \Bigl( \nabla^\varphi\cdot\Bigl( \chi_{\rm b}\tfrac{(n^\varphi)^\perp\cdot w_{\rm i}}{2h_{{\rm i},*}}(n^\varphi)^\perp \Bigr) \Bigr)
  ( ((n^\varphi)^\perp\cdot\check{q}_{\rm i})^2 - (n^\varphi\cdot\check{q}_{\rm i})^2 ) \\
&\quad\;
 + \chi_{\rm b}\tfrac{(n^\varphi)^\perp\cdot w_{\rm i}}{h_{{\rm i},*}}\bigl\{
  \bigl( (\check{q}_{\rm i}\cdot\nabla^\varphi)(n^\varphi)^\perp - (\check{q}_{\rm i}^\perp\cdot\nabla^\varphi)(n^\varphi) \bigr)\cdot\check{q}_{\rm i} 
 + (n^\varphi\cdot\check{q}_{\rm i})f_{{\rm i},1} \\
&\makebox[3em]{}
 + ((n^\varphi)^\perp\cdot\check{q}_{\rm i})( (\nabla^\varphi h_{\rm i})^\perp\cdot\check{v}_{\rm i} + h_{\rm i}(\nabla^\varphi)^\perp\cdot f_{{\rm i},3}
  + (\nabla^\varphi)^\perp\cdot f_{{\rm i},2} ) \bigr\}, \\
G_{{\rm i},3}
&= \Bigl( \nabla^\varphi\Bigl( \chi_{\rm b}\tfrac{\lambda_{{\rm i},-}}{h_{{\rm i},*}}(n^\varphi)^\perp\cdot f_{{\rm i},2} \Bigr) \Bigr)^\perp\cdot
  (\check{q}_{\rm i}-\tfrac12f_{{\rm i},2}) \\
&\quad\;
 + \Bigl( \nabla^\varphi\Bigl( \chi_{\rm b}\tfrac{1}{h_{{\rm i},*}}
  \bigl( w_{\rm i}\cdot f_{{\rm i},2}-(n^\varphi\cdot w_{\rm i})(n^\varphi\cdot f_{{\rm i},4}) \bigr) \Bigr) \Bigr)\cdot\check{q}_{\rm i} \\
&\quad\;
 + \chi_{\rm b}\tfrac{1}{h_{{\rm i},*}}\bigl\{
  \lambda_{{\rm i},-}((n^\varphi)^\perp\cdot f_{{\rm i},2})\bigl( (\nabla^\varphi h_{\rm i})^\perp\cdot\check{v}_{\rm i}
   + h_{\rm i}(\nabla^\varphi)^\perp\cdot f_{{\rm i},3} + \tfrac12(\nabla^\varphi)^\perp\cdot f_{{\rm i},2} \bigr) \\
&\makebox[3em]{}
 + \bigl( w_{\rm i}\cdot f_{{\rm i},2}-(n^\varphi\cdot w_{\rm i})(n^\varphi\cdot f_{{\rm i},4}) \bigr)f_{{\rm i},1} \bigr\}
\end{align*}
with $\check{q}_{\rm i}=h_{\rm i}\check{v}_{\rm i}+f_{{\rm i},2}$. 
\end{lemma}

\begin{proof}
In view of \eqref{DivTh}, for any function $f_{\rm i}$ and $g_{\rm i}$ defined in $\ul{\cI}$ we have 
\begin{align*}
\int_{\ul{\itGamma}}|N^\varphi|f_{\rm i}
&= \int_{\ul{\itGamma}}N^\varphi\cdot( f_{\rm i}n^\varphi + g_{\rm i}(n^\varphi)^\perp ) \\
&= \int_{\ul{\cI}}J\nabla^\varphi\cdot\bigl(\chi_{\rm b}( f_{\rm i}n^\varphi + g_{\rm i}(n^\varphi)^\perp ) \bigr).
\end{align*}
Therefore, we can choose $G_{{\rm i},1}$ as 
\[
G_{{\rm i},1} = \nabla^\varphi\cdot\Bigl( \chi_{\rm b}\tfrac{\lambda_{{\rm i},-}}{2h_{{\rm i},*}}\bigl\{
 \bigl( ((n^\varphi)^\perp\cdot\check{q}_{\rm i})^2 - (n^\varphi\cdot\check{q}_{\rm i})^2 \bigr)n^\varphi
 - 2(n^\varphi\cdot\check{q}_{\rm i})((n^\varphi)^\perp\cdot\check{q}_{\rm i}) (n^\varphi)^\perp \bigr\} \Bigr).
\]
Here, we apply Lemma \ref{lem:Rellich1} with $f=n^\varphi$ and $q=\check{q}_{\rm i}$. 
Moreover, by \eqref{LSWEi2} we have 
\begin{equation}\label{DivRot}
\begin{cases}
 \nabla^\varphi\cdot\check{q}_{\rm i}=f_{{\rm i},1}, \\
 (\nabla^\varphi)^\perp\cdot\check{q}_{\rm i}
  =(\nabla^\varphi h_{\rm i})^\perp\cdot\check{v}_{\rm i} + h_{\rm i}(\nabla^\varphi)^\perp\cdot f_{{\rm i},3} + (\nabla^\varphi)^\perp\cdot f_{{\rm i},2}.
\end{cases}
\end{equation}
Hence, we obtain the expression of $G_{{\rm i},1}$ in the lemma.

Similarly, we can choose $G_{{\rm i},2}$ as 
\[
G_{{\rm i},2 } = \nabla^\varphi\cdot\Bigl( \chi_{\rm b}\tfrac{(n^\varphi)^\perp\cdot w_{\rm i}}{2h_{{\rm i},*}}\bigl\{
 2((n^\varphi)^\perp\cdot\check{q}_{\rm i})(n^\varphi\cdot\check{q}_{\rm i}) n^\varphi 
 + \bigl( ((n^\varphi)^\perp\cdot\check{q}_{\rm i})^2 - (n^\varphi\cdot\check{q}_{\rm i})^2 \bigr) (n^\varphi)^\perp \bigr\} \Bigr),
\]
which together with Lemma \ref{lem:Rellich1} with $f=(n^\varphi)^\perp$ and $q=\check{q}_{\rm i}$ 
and \eqref{DivRot} yields the expression of $G_{{\rm i},2}$ in the lemma.

Finally, it is straightforward to calculate $G_{{\rm i},3}$. 
In fact, we rewrite $I_3(t)$ as 
\begin{align*}
I_3(t) 
&= \int_{\ul{\itGamma}}\Bigl\{
 (N^\varphi)^\perp\cdot \Bigl( \tfrac{\lambda_{{\rm i},-}}{h_{{\rm i},*}} ((n^\varphi)^\perp\cdot f_{{\rm i},2})(\check{q}_{\rm i}-\tfrac12 f_{{\rm i},2}) \Bigr) \\
&\makebox[3em]{}
 + N^\varphi\cdot \Bigl( \tfrac{1}{h_{{\rm i},*}}\bigl( w_{\rm i}\cdot f_{{\rm i},2}-(n^\varphi\cdot w_{\rm i})(n^\varphi\cdot f_{{\rm i},4})
 \bigr) \check{q}_{\rm i} \Bigr) \Bigr\}.
\end{align*}
Therefore, we can choose $G_{{\rm i},3}$ as 
\begin{align*}
G_{{\rm i},3}
&= (\nabla^\varphi)^\perp\cdot \Bigl( \tfrac{\lambda_{{\rm i},-}}{h_{{\rm i},*}} ((n^\varphi)^\perp\cdot f_{{\rm i},2})
  (\check{q}_{\rm i}-\tfrac12 f_{{\rm i},2}) \Bigr) + \nabla^\varphi\cdot \Bigl( \tfrac{1}{h_{{\rm i},*}}
 \bigl( w_{\rm i}\cdot f_{{\rm i},2}-(n^\varphi\cdot w_{\rm i})(n^\varphi\cdot f_{{\rm i},4}) \bigr)\check{q}_{\rm i} \Bigr),
\end{align*}
which together with \eqref{DivRot} yields the expression of $G_{{\rm i},3}$ in the lemma. 
\end{proof}

%-----------------------------------------------------------
\subsection{Completion of the additional regularity estimate}
It follows from \eqref{EqCF6} together with Lemma \ref{lem:Rellich2} that 
\begin{align*}
\int_0^t|\check{\zeta}(t')|_{L^2(\ul{\itGamma})}^2\mathrm{d}t'
&\lesssim \widetilde{\mathscr{E}}_{\rm e}(t) + \mathscr{E}_{\rm e}(0) 
 + \int_0^t \bigl( \|(F_{\rm e},G_3)(t')\|_{L^1(\ul{\cE})} + \|(G_{{\rm i},1},G_{{\rm i},2},G_{{\rm i},3})(t')\|_{L^1(\ul{\cI})} \bigr){\rm d}t'.
\end{align*}
Here, under Assumption \ref{ass:LP} we see that 
\[
\begin{cases}
 \widetilde{\mathscr{E}}_{\rm e}(t) + \mathscr{E}_{\rm e}(t) \lesssim \|(\check{u},f_4)(t)\|_{L^2(\ul{\cE})}^2, \\
 |(F_{\rm e},G_3)| \lesssim (1+|\bm{\partial}(h,w,\partial\varphi)|)|(\check{u},f_4)|^2 + |(\check{u},f_4)||(f_1,f_2,f_3,\bm{\partial}f_4)|, \\
 \sum_{j=1}^3|G_{{\rm i},j}| \lesssim (1+|\partial(h_{\rm i},h_{{\rm i},*},w_{\rm i},\partial\varphi)|)|(\check{v}_{\rm i},f_{{\rm i},2},f_{{\rm i},4})|^2 \\
 \makebox[6em]{}
  + |(\check{v}_{\rm i},f_{{\rm i},2},f_{{\rm i},4})||(f_{{\rm i},1},\partial(f_{{\rm i},2},f_{{\rm i},3},f_{{\rm i},4}))|.
\end{cases}
\]
These estimates together with those in \eqref{Ext1} and \eqref{Ext2} imply the desired one. 
\hfill$\Box$

%----------------------------------------------------------------------------------------------------------------------
\section{Construction of a regularizing diffeomorphism}\label{sect:diffeo}
In this section, we construct the diffeomorphism $\varphi(t,\cdot)$ from the unknown function $\gamma(t,\cdot)$ as was mentioned in Section \ref{subsect:CT}. 
In order to close the estimates for the solutions to the nonlinear problem, it is important to use a regularizing diffeomorphism, whose regularity index is 
$\frac12$ larger than that of $\gamma$. 
Let $\ul{\itGamma}$ be a positively oriented Jordan curve of $C^\infty$-class and suppose that $\ul{\itGamma}$ is parameterized by the arc length $s$ as 
$x = \ul{x}(s) = (\ul{x}_1(s),\ul{x}_2(s))^\mathrm{T}$ for $s \in \mathbb{T}_L$. 
Then, it holds that $|\ul{x}'(s)| \equiv 1$ and that $\ul{n}(s)=-\ul{x}'(s)^\perp$, where $\ul{n}$ is the unit outward normal vector to $\ul{\itGamma}$ 
pointing from $\ul{\cI}$ to $\ul{\cE}$. 
We also have $\ul{x}''(s)=\kappa(s)\ul{x}'(s)^\perp$, where $\kappa(s)$ is the scalar curvature of the curve $\ul{\itGamma}$ at $\ul{x}(s)$.

%-----------------------------------------------------------
\subsection{Normal-tangential coordinate system}\label{subsect:NTCS}
For $r_0>0$, we define a map $\theta: (-r_0,r_0)\times\mathbb{T}_L\ni(r,s) \mapsto \theta(r,s)\in\R^2$ by 
\begin{equation}\label{diffeo-theta}
\theta(r,s) = \ul{x}(s) + r\ul{n}(s) = 
\begin{pmatrix}
 \ul{x}_1(s) + r\ul{x}_2'(s) \\
 \ul{x}_2(s) - r\ul{x}_1'(s)
\end{pmatrix},
\end{equation}
and a tubular neighborhood $U_{\ul{\itGamma}}$ of $\ul{\itGamma}$ by 
\[
U_{\ul{\itGamma}} = \{ x=\theta(r,s)\in\R^2 \,|\, (r,s)\in (-r_0,r_0)\times\mathbb{T}_L \}.
\]
Then, the Jacobian determinant of the map $\theta(r,s)$ is given by $\det\bigl( \frac{\partial\theta(r,s)}{\partial(r,s)} \bigr) = 1+r\kappa(s)$. 
Therefore, if we take $r_0>0$ so small that $r_0 |\kappa|_{L^\infty(\mathbb{T}_L)}<1$, 
then the map $\theta: (-r_0,r_0)\times\mathbb{T}_L \to U_{\ul{\itGamma}}$ is a $C^\infty$-diffeomorphism. 
In the following, we fix such a constant $r_0>0$ so that $U_{\ul{\itGamma}}$ is also fixed. 
Each point of the neighborhood $U_{\ul{\itGamma}}$ can therefore be uniquely determined by its normal-tangential coordinates $(r, s)$. 
Associated with these normal-tangential coordinate system, 
we can define normal and tangential derivatives of functions $f$ defined in the tubular neighborhood $U_{\ul{\itGamma}}$ by 
\begin{equation}\label{NorTanD}
(\dnor f)\circ\theta = \partial_r(f\circ\theta), \quad
(\dtan f)\circ\theta = \partial_s(f\circ\theta). 
\end{equation}
Then, we have 
\begin{equation}\label{NorTanD2}
\dnor = \ul{N}\cdot\nabla,\quad
\dtan = \ul{T}\cdot\nabla,
\end{equation}
where $\ul{T}\circ\theta = (1+r\kappa(s))\ul{x}'(s)$ and $\ul{N}\circ\theta = (-\ul{x}'(s))^\perp$. 
We note that $\dnor$ and $\dtan$ commute with each other. 
Conversely, we have a decomposition 
\begin{equation}\label{diff deco}
\nabla = \ul{N}\dnor + \frac{1}{|\ul{T}|^2}{\ul{T}}\dtan.
\end{equation}

Now, we impose the following assumptions on the unknown curve $\itGamma(t)$.

\begin{assumption}\label{ass:CL}
In the normal-tangential coordinate system $(r,s)$, 
the curve $\itGamma(t)$ is parameterized as a graph in the form $\itGamma(t): r=\gamma(t,s)$ for $s\in\mathbb{T}_L$. 
Moreover, there exist positive constants $\delta_0,M_0$, and $T$ and real number $m_0>\frac12$ such that $\gamma$ satisfies the following properties: 
\begin{enumerate}
\item[{\rm (i)}]
$|\gamma(t,s)| \leq \frac{r_0}{(1+\delta_0)^2}$ for $(t,s)\in[0,T]\times\mathbb{T}_L$.

\item[{\rm (ii)}]
$|\partial_s\gamma(t,\cdot)|_{H^{m_0}(\mathbb{T}_L)} \leq 2 M_0$ for $t\in[0,T]$.
\end{enumerate}
\end{assumption}

%-----------------------------------------------------------
\subsection{Choice of a regularizing diffeomorphism}\label{subsect:ChoiceD}
Under this Assumption \ref{ass:CL}, we are going to construct the diffeomorphism $\varphi(t,\cdot)$. 
We first extend the function $\gamma(t,s)$ on $[0,T]\times\mathbb{T}_L$ to a function $\gamma^{\rm ext}(t,r,s)$ on $[0,T]\times\R\times\mathbb{T}_L$, 
which should satisfies $\gamma^{\rm ext}|_{r=0}=\gamma$. 
We choose such an extension by $\gamma^{\rm ext}(t,r,s)=(\chi(\varepsilon r\langle D\rangle)\gamma(t,\cdot))(s)$ with a small parameter $\varepsilon>0$ 
which will be defined below, where $\chi \in C_0^\infty(\R)$ is a cut-off function satisfying $\mbox{supp}\,\chi \subset (-1,1)$, 
$0\leq \chi(r)\leq 1$, $\chi(r)=1$ for $|r|\ll1$, and $|\chi'(r)| \leq 1+\delta_0$, and $\langle D\rangle=(1-\partial_s^2)^{1/2}$. 
We define a scalar function $R(t,r,s)$ by 
\[
R(t,r,s) = r + \gamma^{\rm ext}(t,r,s)\chi\bigl(\tfrac{r}{r_0}\bigr).
\]
By Assumption \ref{ass:CL} we see that 
\begin{align*}
\partial_r R(t,r,s)
&= 1 + \frac{\gamma(t,s)}{r_0}\chi'\bigl(\tfrac{r}{r_0}\bigr) + \frac{\gamma^{\rm ext}(t,r,s)-\gamma(t,s)}{r_0}\chi'\bigl(\tfrac{r}{r_0}\bigr)
 + (\partial_r\gamma^{\rm ext})(t,r,s)\chi\bigl(\tfrac{r}{r_0}\bigr) \\
&\geq \frac{\delta_0}{1+\delta_0} - \varepsilon C(m_0)(1+\delta_0)M_0,
\end{align*}
where $C(m_0)$ is a positive constant depending only on $m_0$. 
Therefore, if we take $\varepsilon>0$ sufficiently small, then we have $\partial_r R(t,r,s) \geq \frac{\delta_0}{2(1+\delta_0)}$, 
so that $R(t,\cdot,s): \R\to\R$ is a diffeomorphism for each $(t,s)$. 
In the following, we fix the parameter $\varepsilon>0$ in this way. 
We note also that $R(t,\cdot,s): (-r_0,r_0) \to (-r_0,r_0)$ is a diffeomorphism. 
In order to construct a regularizing diffeomorphism, we use this diffeomorphism $R$ to modify the map $\theta$ by 
$\tilde{\theta}(t,\cdot): (-r_0,r_0)\times\mathbb{T}_L \to U_{\ul{\itGamma}}$ with
\begin{equation}\label{theta-modify}
\tilde{\theta}(t,r,s) = \ul{x}(s) + R(t,r,s)\ul{n}(s).
\end{equation}
Then, we see that 
\begin{align*}
\det\biggl( \frac{\partial\tilde{\theta}(t,r,s)}{\partial(r,s)} \biggr)
&= (\partial_rR(t,r,s))(1+R(t,r,s)\kappa(s)) \\
&\geq \frac{\delta_0}{2(1+\delta_0)}(1-r_0|\kappa|_{L^\infty(\mathbb{T}_L)})>0,
\end{align*}
so that $\tilde{\theta}(t,\cdot): (-r_0,r_0)\times\mathbb{T}_L \to U_{\ul{\itGamma}}$ is a diffeomorphism, 
and we also have $\tilde{\theta}(t,r,s)=\theta(r,s)$ for $|r|\simeq r_0$. 
Now, we define the diffeomorphism $\varphi(t,\cdot) : \R^2 \to \R^2$ by 
\begin{equation}\label{const phi}
\varphi(t,y) = 
\begin{cases}
 \tilde{\theta}(t,\theta^{-1}(y)) &\mbox{for}\quad y \in U_{\ul{\itGamma}}, \\
 y & \mbox{for}\quad y \notin U_{\ul{\itGamma}}.
\end{cases}
\end{equation}
We can easily check that this diffeomorphism $\varphi(t,\cdot)$ satisfies the desired properties, that is, 
$\varphi(t,\cdot)_{\vert_{\ul{\cE}}} : \ul{\cE}\to \cE(t)$, $\varphi(t,\cdot)_{\vert_{\ul{\cI}}} : \ul{\cI}\to \cI(t)$, and 
$\varphi(t,\cdot)_{\vert_{\ul{\itGamma}}} : \ul{\itGamma}\to \itGamma(t)$ are all diffeomorphisms and that it does not change the orientation. 
Throughout this paper, we use this diffeomorphism.

This diffeomorphism is decomposed as $\varphi(t,y)=y+\widetilde{\varphi}(t,y)$ with 
\[
\widetilde{\varphi}(t,\theta(r,s))=\gamma^{\rm ext}(t,r,s)\chi\bigl(\tfrac{r}{r_0}\bigr)\ul{n}(s).
\]
We proceed to evaluate this perturbation term $\widetilde{\varphi}$ in terms of $\gamma$. 
Here, we recall the norm $|\gamma(t,\cdot)|_m$ for a non-negative integer $m$, which was introduced in Section \ref{subsect:result}. 
The fact that the diffeomorphism $\varphi(t,\cdot)$ is regularizing allows us to control $m$-th order derivatives of $\varphi(t,\cdot)$ 
not only in $L^2(\R^2)$ but also in $L^4(\R^2)$ by $|\gamma(t,\cdot)|_m$.

\begin{lemma}\label{lem:EstDM}
For any multi-index $\alpha=(\alpha_0,\alpha_1,\alpha_2)$ satisfying $|\alpha|=m$ and any $p\in[4,\infty]$, we have 
\begin{enumerate}
\item[{\rm (i)}]
$\|\bm{\partial}^\alpha\widetilde{\varphi}(t,\cdot)\|_{L^2\cap L^4(\R^2)} + |\bm{\partial}^\alpha\widetilde{\varphi}(t,\cdot)|_{L^2(\ul{\itGamma})}
 \leq C|\gamma(t,\cdot)|_m;$
\item[{\rm (ii)}]
$\|\bm{\partial}^\alpha\widetilde{\varphi}(t,\cdot)\|_{L^p(\R^2)}\leq C|\gamma(t,\cdot)|_m^{1/2+2/p}|\gamma(t,\cdot)|_{m+1}^{1/2-2/p};$
\item[{\rm (iii)}]
$\|\bm{\partial}^\alpha\widetilde{\varphi}(t,\cdot)\|_{L^2(\R^2)}\leq C|\gamma(t,\cdot)|_{m-1}^{1/2}|\gamma(t,\cdot)|_m^{1/2}$ if $\alpha_1+\alpha_2\geq1;$
\end{enumerate}
where $\bm{\partial}^\alpha=\dt^{\alpha_0}\partial_1^{\alpha_1}\partial_2^{\alpha_2}$ and $C>0$ is a constant depending on $r_0$ and $\varepsilon$. 
\end{lemma}

\begin{proof}
It is well-known that the extension $\gamma^{\rm ext}(t,\cdot)$ of $\gamma(t,\cdot)$ gains $\frac12$-regularity in the sense that 
it maps $L^2(\mathbb{T}_L)$ into $H^{1/2}(\R\times\mathbb{T}_L)$ continuously. 
This fact together with the embedding $H^{1/2}(\R\times\mathbb{T}_L) \hookrightarrow L^4(\R\times\mathbb{T}_L)$ implies 
$\|\dt^j\gamma^{\rm ext}(t,\cdot)\|_{W^{k,4}(\R^2)} \lesssim |\dt^j\gamma(t,\cdot)|_{H^k(\mathbb{T}_L)}$. 
Using this, we obtain the estimate for $\|\bm{\partial}^\alpha\widetilde{\varphi}(t,\cdot)\|_{L^2\cap L^4(\R^2)}$ in (i). 
On the other hand, evaluation for $|\bm{\partial}^\alpha\widetilde{\varphi}(t,\cdot)|_{L^2(\ul{\itGamma})}$ in (i) is straightforward. 
By the Gagliardo--Nirenberg interpolation inequality $\|f\|_{L^p(\R^2)}\lesssim \|f\|_{L^4(\R^2)}^{1/2+2/p}\|\nabla f\|_{L^4(\R^2)}^{1/2-2/p}$, 
the estimate in (ii) follows from (i). 
As for (iii), we note that $\alpha$ can be decomposed as $\alpha=\alpha'+\alpha''$ with $\alpha'=(0,1,0)$ or $(0,0,1)$. 
Therefore, by the Gagliardo--Nirenberg interpolation inequality $\|\nabla f\|_{L^2(\R^2)} \lesssim \|\nabla f\|_{L^4(\R^2)}^{1/2} \|f\|_{L^4(\R^2)}^{1/2}$, 
we see that 
\begin{align*}
\|\bm{\partial}^\alpha\widetilde{\varphi}(t)\|_{L^2(\R^2)}
&\lesssim \|\nabla\bm{\partial}^{\alpha''}\widetilde{\varphi}(t)\|_{L^4(\R^2)}^{1/2} \|\bm{\partial}^{\alpha''}\widetilde{\varphi}(t)\|_{L^4(\R^2)}^{1/2} \\
&\lesssim |\gamma(t)|_m^{1/2} |\gamma(t)|_{m-1}^{1/2},
\end{align*}
where we used (i). 
This shows (iii). 
\end{proof}

%----------------------------------------------------------------------------------------------------------------------
\section{Good unknowns and their equations}\label{sect:GU}
In this section, we consider the nonlinear shallow water model \eqref{NLSWEe5}--\eqref{ODE5} for the unknowns $\zeta,v$ and $\psi_{\rm i}$, 
and the diffeomorphism $\varphi$, which was constructed from the unknown curve $\itGamma(t)$ by \eqref{const phi}. 
We first introduce in Section \ref{subsect:defGU} the good unknowns associated with the highest order derivatives of order $m$ of 
$\zeta,v,$ and $\psi_{\rm i}$, as well as for $\phi_{\rm i}$ and $v_{\rm i}=\nabla^\varphi \phi_{\rm i}$, with $\phi_{\rm i}$ given by \eqref{BVP2}; 
we insist on the fact that the good unknowns associated with $\phi_{\rm i}$ are of second order, in the sense that they contain subprincipal terms. 
In Section \ref{subsect:EGUe} we derive the equations satisfied by the good unknowns in the exterior region $\ul{\cE}$ near the boundary, 
and therefore in normal-tangential coordinates. 
When only time and tangential derivatives are involved, the procedure is similar to the linearization performed in Section \ref{subsect:DLP}, 
but one has to keep track of additional source terms due to the fact that high order derivatives are not exactly linearization operators. 
When normal derivatives are involved, despite the fact that the problem is partially characteristic, 
we manage to express them in terms of quantities that contain only time and tangential derivatives. 
In Section \ref{sectnearbdry}, the equations for the good unknowns are derived near the boundary, but this time in the interior region $\ul{\cI}$. 
Since the equations satisfied by $\phi_{\rm i}$ is of second order, it is crucial to use here the second order Alinhac good unknowns. 
In Section \ref{secttangdiff} we proceed to apply time and tangential derivatives to the boundary conditions on $\ul{\itGamma}$ 
and to rewrite the resulting equations in terms of the good unknowns. 
Using the Cartesian coordinates, equations for the good unknowns far from the boundary are derived in Section \ref{subsect:EqGUb}. 
Finally, we derive in Section \ref{subsect:Eqg} equations for the highest order derivatives of the parametrization $\gamma$ of the free boundary $\itGamma(t)$; 
see Assumption \ref{ass:CL}.

%-----------------------------------------------------------
\subsection{Definition of good unknowns and differentiation rules}\label{subsect:defGU}
In Section \ref{subsect:DLP}, 
we provided derivation rules that allowed us to write the linearization of the equations \eqref{NLSWEe5}--\eqref{ODE5} in a convenient way. 
We showed that for an unknown function $f$, 
we could associate its variation in the linearization, denoted by $\dot{f}$, but also a good unknown $\check{f}=\dot{f}-\dot{\varphi}\cdot\nabla^\varphi f$. 
The key observation was \eqref{linearization 2}, that expressed the fact that we have the commutation rules 
$(\nabla^\varphi\cdot  \mbox{\boldmath$f$})\,\check{} = \nabla^\varphi\cdot \check{\mbox{\boldmath$f$}}$, $(\nabla^\varphi f)\,\check{}
 = \nabla^\varphi \check{f}$, and $(\dt^\varphi f)\,\check{} = \dt^\varphi  \check{f}$. 
In this section, instead of linearizing the equations, 
we want to derive the system deduced from \eqref{NLSWEe5}--\eqref{ODE5} after multiple differentiation in space and time. 
We first introduce in Section \ref{sectdefGU} the associated good unknowns and then derive some useful differentiation rules in Section \ref{sectdiffrules}.

%-----------------------------------------------------------
\subsubsection{Definition of the good unknowns}\label{sectdefGU}
Since it is natural to use the normal and the tangential derivatives $\dnor$ and $\dtan$ defined by \eqref{NorTanD} 
instead of the Cartesian derivatives $\partial_1$ and $\partial_2$ near the boundary $\ul{\itGamma}$, we are led to introduce the following notations. 
For a multi-index $\alpha=(\alpha_0,\alpha_1,\alpha_2)=(\alpha_I,\alpha_2)$, we write 
\[
\bm{\partial}^\alpha=\dt^{\alpha_0}\partial_1^{\alpha_1}\partial_2^{\alpha_2}, \qquad 
\mathfrak{d}^\alpha=\dt^{\alpha_0}\dtan^{\alpha_1}\dnor^{\alpha_2},
\quad\mbox{  and } \quad
\dpar^{\alpha_I}=\dt^{\alpha_0}\dtan^{\alpha_1}.
\]
Note that if $\alpha_1=\alpha_2=0$, then $\bm{\partial}^\alpha=\mathfrak{d}^\alpha=\dpar^{\alpha_I}=\dt^{\alpha_0}$. 
Inspired by the considerations of  Section \ref{subsect:DLP}, we also define
\[
\cpalpha f=\bm{\partial}^\alpha f -\bm{\partial}^\alpha\varphi\cdot \nabla^\varphi f, \qquad
\cmfalpha f=\mathfrak{d}^\alpha f -\mathfrak{d}^\alpha\varphi\cdot \nabla^\varphi f, 
\quad\mbox{ and }\quad
\cmfpalphaI f=\dpar^{\alpha_I} f -\dpar^{\alpha_I}\varphi\cdot \nabla^\varphi f.
\]
Let $u=(\zeta,v^\mathrm{T})^\mathrm{T},\psi_{\rm i}$, and $\varphi$ be a solution of \eqref{NLSWEe5}--\eqref{ODE5}, 
and let $v_{\rm i}=\nabla^\varphi\phi_{\rm i}$, with $\phi_{\rm i}$ solving the elliptic system \eqref{BVP2}. 
We now introduce the good unknowns associated with these quantities.

If $\alpha=(j,0,0)$, which means that we only differentiate in time, then we define 
\[
\check{u}^{(\alpha)} = \cpalpha u
\quad\mbox{ and }\quad
\check{v}_{\rm i}^{(\alpha)} = \cpalpha v_{\rm i}.
\]
As for the good unknown for $\dt^j\phi_{\rm i}$, 
we need to compensate a lower order correction to the standard definition of good unknowns due to technical reasons; see Remark \ref{re:GU2} below. 
Here, for $\alpha=(j,0,0)$ we adopt the following definition 
\begin{align*}
\check{\phi}_{\rm i}^{(\alpha)}
&= \dt^j \phi_{\rm i}-(\dt^j\varphi)\cdot\nabla^\varphi\phi_{\rm i}-j(\dt^{j-1}\varphi)\cdot\dt\nabla^\varphi\phi_{\rm i} \\
&= \dt^j \phi_{\rm i}-(\dt^j\varphi)\cdot v_{\rm i}-j(\dt^{j-1}\varphi)\cdot\dt v_{\rm i}.
\end{align*}

If $\alpha=(\alpha_0,\alpha_1,\alpha_2)$, with $\alpha_1+\alpha_2\geq 1$, that is, if the derivatives include spatial ones, 
then we need to separate the good unknowns into those supported near the boundary and far from boundary. 
Let $\chi_{\rm b}\in C_0^\infty(\R^2)$ be a cut-off function such that $\chi_{\rm b}=1$ near the boundary $\ul{\itGamma}$ and that its support is in the 
tubular neighborhood $U_{\ul{\itGamma}}$ of the reference curve $\ul{\itGamma}$ defined in Section \ref{subsect:NTCS}. 
We also introduce other cut-off functions $\chi_{\rm e},\chi_{\rm i}\in C^\infty(\R^2)$ such that their supports are in $\ul{\cE}$ and $\ul{\cI}$, respectively, 
and that $\chi_{\rm b}+\chi_{\rm e}+\chi_{\rm i}\equiv1$ holds. 
\begin{itemize}
\item
Near the boundary $\ul{\itGamma}$, the good unknowns are then defined by 
\[
\check{u}^{(\alpha)} = \chi_{\rm b} \cmfalpha u
\quad\mbox{ and }\quad
\check{v}_{\rm i}^{(\alpha)} = \chi_{\rm b} \cmfalpha v_{\rm i};
\]
as for $\phi_{\rm i}$, as in the case of time derivatives, we need to compensate lower order corrections and define the good unknown by 
\begin{align*}
\check{\phi}_{\rm i}^{(\alpha)} 
&= \chi_{\rm b} \{ \mathfrak{d}^{\alpha}\phi_{\rm i} - \sum_{\beta\leq\alpha, |\beta|\geq|\alpha|-1} \binom{\alpha}{\beta} 
  (\mathfrak{d}^{\beta}\varphi)\cdot \mathfrak{d}^{\alpha-\beta}v_{\rm i} \}.
\end{align*}
Moreover, if the derivatives do not include normal ones, then we write $\check{u}^{(\alpha_I)}=\check{u}^{(\alpha_I,0)}$, 
$\check{v}_{\rm i}^{(\alpha_I)}=\check{v}_{\rm i}^{(\alpha_I,0)}$, 
and $\check{\phi}_{\rm i}^{(\alpha_I)}=\check{\phi}_{\rm i}^{(\alpha_I,0)}$ for $\alpha_I=(\alpha_0,\alpha_1)$. 
We also denote by $\check\psi_{\rm i}^{(\alpha_{I})}$ the trace of  $\check{\phi}_{\rm i}^{(\alpha_I)}$ in $\ul{\itGamma}$, so that
\[
\check\psi_{\rm i}^{(\alpha_{I})}=\mathfrak{d}_\parallel^{\alpha_{I}}\psi_{\rm i} - \sum_{\beta_I\leq\alpha_I, |\beta_I|\geq|\alpha_I|-1} \binom{\alpha_I}{\beta_I} 
  (\mathfrak{d}_\parallel^{\beta}\varphi)\cdot \mathfrak{d}_\parallel^{\alpha-\beta}v_{\rm i}. 
\]
\item
Far from the boundary $\ul{\itGamma}$, the good unknowns are defined by using the standard Cartesian derivatives as 
\[
\check{u}_{\rm r}^{(\alpha)} = \chi_{\rm e} \cpalpha u
\quad\mbox{ and }\quad
\check{v}_{\rm i,r}^{(\alpha)} = \chi_{\rm i}\cpalpha v_{\rm i},
\]
and 
\[
\check{\phi}_{\rm i,r}^{(\alpha)} 
= \chi_{\rm i} \{ \bm{\partial}^\alpha\phi_{\rm i} - \sum_{\beta\leq\alpha, |\beta|\geq|\alpha|-1} \binom{\alpha}{\beta} 
  (\bm{\partial}^{\beta}\varphi)\cdot\bm{\partial}^{\alpha-\beta}v_{\rm i} \}.
\]
\end{itemize}

%-----------------------------------------------------------
\subsubsection{Differentiation rules} \label{sectdiffrules}
The difference here with Section \ref{subsect:DLP} is that $\cpalpha$ or $\cmfalpha$ are linearization operators only up to lower order terms. 
Because of this, the commutation relations \eqref{linearization 2} only hold up to lower order terms. More precisely, we have
\begin{equation}\label{comrules2}
\begin{cases}
\mathfrak{d}^\alpha(\nabla^\varphi f)
 = \nabla^\varphi \cmfalpha f  + ((\mathfrak{d}^\alpha\varphi)\cdot\nabla^\varphi)\nabla^\varphi f +{\mathcal C}^1(\mathfrak{d}^\alpha,\partial \varphi) f, \\
\mathfrak{d}^\alpha(\nabla^\varphi\cdot v)
 = \nabla^\varphi\cdot \cmfalpha v 
  + ((\mathfrak{d}^\alpha\varphi)\cdot\nabla^\varphi)(\nabla^\varphi\cdot v) +{\mathcal C}^1(\mathfrak{d}^\alpha,\partial \varphi)\cdot v, \\
\mathfrak{d}^\alpha(\dt^\varphi f)
 = \dt^\varphi \cmfalpha f+ ((\mathfrak{d}^\alpha\varphi)\cdot\nabla^\varphi)\dt^\varphi f +
 {\mathcal C}^2(\mathfrak{d}^\alpha,\dt \varphi) f
 - (\dt\varphi)\cdot  {\mathcal C}^1(\mathfrak{d}^\alpha,\partial \varphi)f,
\end{cases}
\end{equation}
where the differential operators ${\mathcal C}^1(\mathfrak{d}^\alpha,\partial \varphi)$ and ${\mathcal C}^2(\mathfrak{d}^\alpha,\dt \varphi)$ are defined as 
\begin{equation}\label{defcomC}
\begin{cases}
{\mathcal C}^1(\mathfrak{d}^\alpha,\partial \varphi) f
 = ( [\mathfrak{d}^\alpha,\nabla^\varphi]+(\partial^\varphi\mathfrak{d}^\alpha\varphi)^\mathrm{T}\nabla^\varphi)f, \\
{\mathcal C}^2(\mathfrak{d}^\alpha,\dt \varphi) f
 = - \textstyle\sum_{j=1,2}[\mathfrak{d}^\alpha; \dt\varphi_j,\partial_j^\varphi f].
\end{cases}
\end{equation}

\begin{remark}\label{remC1C2}
If $|\alpha|=m\geq1$, then ${\mathcal C}^1(\mathfrak{d}^\alpha,\partial \varphi) f$ 
and ${\mathcal C}^2(\mathfrak{d}^\alpha,\dt \varphi) f$ include derivatives of $f$ and $\varphi$ up to at most order $m$; 
the fact that the derivatives of $\varphi$ of order $m+1$ are cancelled in ${\mathcal C}^1(\mathfrak{d}^\alpha,\partial \varphi) f$ 
follows from the second equation in \eqref{linearization 1}. 
\end{remark}

%-----------------------------------------------------------
\subsection{Equations for the good unknowns in $\ul{\cE}$ near the boundary}\label{subsect:EGUe}
In this subsection, we derive evolution equations for the good unknowns $\check{u}^{(\alpha)}$ for $\alpha=(\alpha_0,\alpha_1,\alpha_2)$, 
and we also show how these equations can be used to provide a control of normal derivatives of the good unknown in terms of tangential ones.

%-----------------------------------------------------------
\subsubsection{Evolution equations for the good unknowns}
Using the identities \eqref{comrules2} and applying $\chi_{\rm b}\mathfrak{d}^{\alpha}$, 
or simply $\mathfrak{d}^{\alpha}$ if $\alpha_1=\alpha_2=0$, to \eqref{NLSWEe5}, we get 
\begin{equation}\label{QLSWEe}
\begin{cases}
 \dt\check{\zeta}^{(\alpha)}+\nabla^\varphi\cdot(h\check{v}^{(\alpha)}+w\check{\zeta}^{(\alpha)}) = f_1^{(\alpha)} &\mbox{in}\quad (0,T)\times\ul{\cE}, \\
 \dt\check{v}^{(\alpha)}+\nabla^\varphi(w\cdot\check{v}^{(\alpha)}+\gr\check{\zeta}^{(\alpha)}) = f_2^{(\alpha)} &\mbox{in}\quad (0,T)\times\ul{\cE}, \\
 (\nabla^\varphi)^\perp\cdot\check{v}^{(\alpha)} = f_3^{(\alpha)} &\mbox{in}\quad (0,T)\times\ul{\cE},
\end{cases}
\end{equation}
where in the case $\alpha_1+\alpha_2\geq1$ 
\begin{align*}
f_1^{(\alpha)}
&= - \chi_{\rm b}\{ {\mathcal C}^2(\mathfrak{d}^\alpha,\dt \varphi)\zeta
  - (\dt\varphi)\cdot {\mathcal C}^1(\mathfrak{d}^\alpha,\partial\varphi)\zeta 
+  {\mathcal C}^1(\mathfrak{d}^\alpha,\partial\varphi)\cdot(hv)
 + \nabla^\varphi\cdot( [\mathfrak{d}^\alpha;h,v] ) \} \\
&\quad\;
+ (\dt^\varphi\chi_{\rm b}) \cmfalpha \zeta
 + \nabla^\varphi\chi_{\rm b}\cdot\bigl( (\cmfalpha \zeta )v + h (\cmfalpha v) \bigr)
 - (\nabla^\varphi\cdot\dt\varphi)\check{\zeta}^{(\alpha)}, \\
f_2^{(\alpha)}
&= -\chi_{\rm b}\{ {\mathcal C}^2(\mathfrak{d}^\alpha,\dt \varphi) v
  - (\dt\varphi)\cdot {\mathcal C}^1(\mathfrak{d}^\alpha,\partial\varphi)v 
+  {\mathcal C}^1(\mathfrak{d}^\alpha,\partial\varphi)  (\tfrac12|v|^2+\gr\zeta)
 + \nabla^\varphi\cdot( \tfrac12[\mathfrak{d}^\alpha;v,\cdot v] ) \} \\
&\quad\;
 + (\dt^\varphi\chi_{\rm b}) \cmfalpha v  
 + (\nabla^\varphi\chi_{\rm b})\bigl( v\cdot\cmfalpha v + \gr\cmfalpha\zeta \bigr) \\
&\quad
 - (\check{v}^{(\alpha)}\cdot\nabla^\varphi)\dt\varphi + ((\nabla^\varphi)^\perp\cdot\dt\varphi)(\check{v}^{(\alpha)})^\perp
 + f_3^{(\alpha)}(\dt\varphi)^\perp, \\
f_3^{(\alpha)}
&= -\chi_{\rm b} {\mathcal C}^1(\mathfrak{d}^\alpha,\partial\varphi)^\perp \cdot v
 + (\nabla^\varphi\chi_{\rm b})^\perp\cdot   \cmfalpha v.
\end{align*}
Here, we used the identities in \eqref{F3} to calculate $f_2^{(\alpha)}$. 
In the case $\alpha_1=\alpha_2=0$, it is sufficient to put $\chi_{\rm b}=1$ in the above equations.

\begin{remark}\label{re:GU1}
If $|\alpha|=m\geq1$, then the right-hand sides $f_1^{(\alpha)}$, $f_2^{(\alpha)}$, and $f_3^{(\alpha)}$ include derivatives of $u$ and $\varphi$ 
up to at most order $m$. 
The fact that the derivatives of $\varphi$ of order $m+1$ are cancelled is a consequence of the introduction of the good unknowns $\check{u}^{(\alpha)}$: 
if we had worked with standard derivatives $\chi_{\rm b}\mathfrak{d}^\alpha u$, derivatives of order $m+1$ in $\varphi$ would have remained. 
\end{remark}

%-----------------------------------------------------------
\subsubsection{Expressing normal derivatives in terms of tangential ones}
The system \eqref{QLSWEe} has the same structure as the linearized system \eqref{LSWEe2}; 
we shall prove below that the good unknowns $\check{u}^{(\alpha_I)}$ for $\alpha_I=(\alpha_0,\alpha_1)$, 
which corresponds to the tangential derivatives in space-time, satisfy boundary condition of the type \eqref{LBC4}--\eqref{LBC4bis}. 
We will therefore be able to control them using the energy estimate given in Proposition \ref{prop:BEE}. 
In order to evaluate $\check{u}^{(\alpha)}$ when $\alpha_2>0$, that is, in the presence of a normal derivative, 
we use \eqref{QLSWEe} to convert the normal derivative into tangential derivatives in space-time. 
Here, we derive this conversion formula. 
In view of \eqref{diff deco}, we have 
\begin{equation}\label{diff deco2}
J\nabla^\varphi=N^\varphi\dnor+T^\varphi\dtan
\end{equation}
with $N^\varphi=J((\partial\varphi)^{-1})^\mathrm{T}\ul{N}$ and $T^\varphi=\frac{J}{|\ul{T}|^2}((\partial\varphi)^{-1})^\mathrm{T}\ul{T}$. 
Plugging this into \eqref{QLSWEe} and adding the last equation in \eqref{QLSWEe} multiplied by an $\R^2$-valued function $\bm{a}$ to the second equation, 
we obtain 
\begin{equation}\label{QLSWEe2}
J\dt\check{u}^{(\alpha)}+A(u,T^\varphi,\bm{a})\dtan\check{u}^{(\alpha)}+A(u,N^\varphi,\bm{a})\dnor\check{u}^{(\alpha)}=f_4^{(\alpha)},
\end{equation}
where 
\begin{align*}
& A(u,N,\bm{a}) = 
 \begin{pmatrix}
  N\cdot w & hN^\mathrm{T} \\
  \gr N & N\otimes w+\bm{a}\otimes N^\perp
 \end{pmatrix}, \\
& f_4^{(\alpha)} = J
 \begin{pmatrix}
  f_1^{(\alpha)} - ( \nabla^\varphi h\cdot\check{v}^{(\alpha)} + (\nabla^\varphi\cdot w)\check{\zeta}^{(\alpha)} ) \\
  f_2^{(\alpha)} - (\partial^\varphi w)^\mathrm{T}\check{v}^{(\alpha)} + f_3^{(\alpha)}\bm{a}
 \end{pmatrix}.
\end{align*}
Here, we have 
\[
\det A(u,N^\varphi,\bm{a}) = (\bm{a}^\perp\cdot N^\varphi)( \gr h|N^\varphi|^2-(w\cdot N^\varphi)^2 ).
\]
Therefore, by choosing $\bm{a}=-(N^\varphi)^\perp$, we see that $\det A(u,N^\varphi,\bm{a})$ is strictly positive 
in the support of the cut-off function $\chi_{\rm b}$ under the subcriticality assumption of the flow. 
In the following, we fix this choice of the function $\bm{a}$.

\begin{remark}
The original equations \eqref{QLSWEe} are partially characteristic, 
and it is not possible in general to use them to convert normal derivatives in terms of tangential ones. 
The above procedure shows that the last equation of \eqref{QLSWEe}, which corresponds to the irrotationality assumption, 
can be used to transform the problem into an equivalent non-characteristic problem. 
\end{remark}

We proceed to rewrite the derivatives of $\check{u}^{(\alpha)}$ in \eqref{QLSWEe2} in terms of good unknowns without using the derivatives. 
Let $\alpha=(j,k,l)$. 
Then, by the definition of the good unknowns we see that in the case $k+l\geq1$ 
\begin{equation}\label{DGU1}
\begin{cases}
 \dt\check{u}^{(\alpha)} = \check{u}^{(j+1,k,l)} - \chi_{\rm b}((\mathfrak{d}^\alpha\varphi)\cdot\dt\nabla^\varphi)u, \\
 \dtan\check{u}^{(\alpha)} = \check{u}^{(j,k+1,l)} - \chi_{\rm b}((\mathfrak{d}^\alpha\varphi)\cdot\dtan\nabla^\varphi)u
  + (\dtan\chi_{\rm b}) \cmfalpha u, \\
 \dnor\check{u}^{(\alpha)} = \check{u}^{(j,k,l+1)} - \chi_{\rm b}((\mathfrak{d}^\alpha\varphi)\cdot\dnor\nabla^\varphi)u
  + (\dnor\chi_{\rm b}) \cmfalpha u,
\end{cases}
\end{equation}
and in the case $k=l=0$ 
\begin{equation}\label{DGU2}
\begin{cases}
 \dt\check{u}^{(j,0,0)} = \check{u}^{(j+1,0,0)} - ((\dt^j\varphi)\cdot\dt\nabla^\varphi)u, \\
 \chi_{\rm b}\dtan\check{u}^{(j,0,0)} = \check{u}^{(j,1,0)} - \chi_{\rm b}((\dt^j\varphi)\cdot\dtan\nabla^\varphi)u, \\
 \chi_{\rm b}\dnor\check{u}^{(j,0,0)} = \check{u}^{(j,0,1)} - \chi_{\rm b}((\mathfrak{d}^\alpha\varphi)\cdot\dnor\nabla^\varphi)u.
\end{cases}
\end{equation}
Therefore, in the case $k+l\geq1$ we obtain 
\begin{equation}\label{QLSWEe3}
J\check{u}^{(j+1,k,l)}+A(u,T^\varphi,\bm{a})\check{u}^{(j,k+1,l)}+A(u,N^\varphi,\bm{a})\check{u}^{(j,k,l+1)}=f_5^{(\alpha)},
\end{equation}
where 
\begin{align*}
f_5^{(\alpha)}
&= f_4^{(\alpha)} + \chi_{\rm b} \{ J((\mathfrak{d}^\alpha\varphi)\cdot\dt\nabla^\varphi)u
 + A(u,T^\varphi,\bm{a})((\mathfrak{d}^\alpha\varphi)\cdot\dtan\nabla^\varphi)u \\
&\quad\;
  + A(u,N^\varphi,\bm{a})((\mathfrak{d}^\alpha\varphi)\cdot\dnor\nabla^\varphi)u \}
 - A(u,J\nabla^\varphi\chi_{\rm b},\bm{a}) \cmfalpha u. 
\end{align*}
In the case $k=l=0$, by replacing $J\check{u}^{(j+1,0,0)}$ with $J\chi_{\rm b}\check{u}^{(j+1,0,0)}$, \eqref{QLSWEe3} still holds with 
\begin{align*}
f_5^{(j,0,0)}
&= \chi_{\rm b}f_4^{(j,0,0)} + \chi_{\rm b} \{ J((\dt^j\varphi)\cdot\dt\nabla^\varphi)u \\
&\quad\;
 + A(u,T^\varphi,\bm{a})((\dt^j\varphi)\cdot\dtan\nabla^\varphi)u 
  + A(u,N^\varphi,\bm{a})((\dt^j\varphi)\cdot\dnor\nabla^\varphi)u \}.
\end{align*}
We will use \eqref{QLSWEe3} for multi-indices $\alpha=(j,k,l)$ satisfying $|\alpha|=m-1$ to convert the normal derivative into the tangential ones in space-time. 
In this case, if $m\geq3$, then $f_5^{(\alpha)}$ includes derivatives of $u$ and $\varphi$ up to at most order $m-1$.

%-----------------------------------------------------------
\subsection{Equations for the good unknowns in $\ul{\cI}$ near the boundary}\label{sectnearbdry}
We want to express in terms of the good unknowns the equations obtained by differentiating the interior equations $\nabla^\varphi\cdot(h_{\rm i}v_{\rm i})=0$ 
and $v_{\rm i}=\nabla^\varphi\phi_{\rm i}$. 
This is more delicate than in the previous section for the exterior region because the elliptic equation satisfied by $\phi_{\rm i}$ is of second order; 
in order to preserve an equation with a similar structure for the good unknown $\check\phi_{\rm i}^{(\alpha)}$, 
it is necessary to include subprincipal terms in the definition of $\check\phi_{\rm i}^{(\alpha)}$, as we did in Section \ref{sectdefGU}.

%-----------------------------------------------------------
\subsubsection{Differentiation of the interior equations}
We proceed to derive the equations for $\check{v}_{\rm i}^{(\alpha)}$ and $\check{\phi}_{\rm i}^{(\alpha)}$ 
obtained by applying $\mathfrak{d}^\alpha$ to the relations $\nabla^\varphi\cdot (h_{\rm i}v_{\rm i})=0$ and $v_{\rm i}=\nabla^\varphi \phi_{\rm i}$. 
These equations are stated in the following proposition.

\begin{proposition}\label{propGUinterior}
If $v_{\rm i}$, $h_{\rm i}$, and $\phi_{\rm i}$ are regular functions satisfying $\nabla^\varphi\cdot (h_{\rm i}v_{\rm i})=0$ and 
$v_{\rm i}=\nabla^\varphi \phi_{\rm i}$, then for any nonzero multi-index $\alpha\in {\mathbb N}^3$, one has 
\begin{equation}\label{QLSWEi}
 \begin{cases}
  \nabla^\varphi\cdot(h_{\rm i}\check{v}_{\rm i}^{(\alpha)}+f_{{\rm i},2}^{(\alpha)}) = f_{{\rm i},1}^{(\alpha)} &\mbox{in}\quad (0,T)\times\ul{\cI}, \\
 \check{v}_{\rm i}^{(\alpha)} = \nabla^\varphi \check{\phi}_{\rm i}^{(\alpha)}+f_{{\rm i},3}^{(\alpha)} &\mbox{in}\quad (0,T)\times\ul{\cI},
 \end{cases}
\end{equation}
where the expressions for $f_{{\rm i},1}^{(\alpha)}$, $f_{{\rm i},2}^{(\alpha)}$, and $f_{{\rm i},3}^{(\alpha)}$ depend on $\alpha$: 
\begin{itemize}
\item
If $\alpha_1+\alpha_2\geq 1$, 
then $f_{{\rm i},1}^{(\alpha)}=f_{{\rm i},01}^{(\alpha)} $ with $f_{{\rm i},01}^{(\alpha)} $ given by \eqref{deffi01} below, 
while $f_{{\rm i},2}^{(\alpha)}=0$, and $f_{{\rm i},3}^{(\alpha)}$ is given by \eqref{deffi3} and subsequent explanation below; 
\item
If $\alpha_1=\alpha_2=0$, 
then $f_{{\rm i},1}^{(\alpha)}=0 $ and $f_{{\rm i},2}^{(\alpha)}=f_{{\rm i},02}^{(\alpha)} $ with $f_{{\rm i},02}^{(\alpha)} $ given by \eqref{deffi02} below, 
while $f_{{\rm i},3}^{(\alpha)}$ is given by \eqref{deffi3bis} below. 
\end{itemize}
\end{proposition}

\begin{remark}\label{re:GU2} \ 
{\bf i.} \ 
If $|\alpha|=m\geq3$, then $f_{{\rm i},2}^{(\alpha)}$ and $f_{{\rm i},3}^{(\alpha)}$ include derivatives of $v_{\rm i}$ and $\varphi$ up to at most order $m-1$ 
thanks to our definition of the good unknowns $\check{\phi}_{\rm i}^{(\alpha)}$. 
In fact, if we adopt the standard definition of the good unknowns for 
$\check{\phi}_{\rm i}^{(\alpha)}$, that is,  $\check{\phi}_{\rm i}^{(\alpha)}=\chi_{\rm b}\cmfalpha\phi_{\rm i}$ as for $\check{v}_{\rm i}^{(\alpha)}$, 
then $f_{{\rm i},3}^{(\alpha)}$ would contain $m$-th order derivatives of $\varphi$, which would cause a difficulty for obtaining a priori estimates. 

{\bf ii.} \ 
In the case $\alpha=(m,0,0)$, we have $f_{{\rm i},1}^{(\alpha)}=0$, which is crucial in the application of the energy estimate in Proposition \ref{prop:BEE}. 
In fact, if it were not zero, then a difficulty would arise from the term $\dt\check{\phi}_{\rm i}^{(m,0,0)}$, 
which appears in the last term of the right-hand side of the energy estimate in Proposition \ref{prop:BEE}, 
because it cannot be evaluated by our energy function $E_m(t)$; 
this issue does not appear when $|\alpha|=m$ and $\alpha_1+\alpha_2\geq1$, 
because $\dt\check{\phi}_{\rm i}^{(\alpha)}$ can then be written in terms of derivatives of $v_{\rm i}$ \and $\varphi$ up to at most order $m$; 
see Section \ref{sectexprdtphi} below. 
\end{remark}

\begin{proof}
By assumption, we have $\nabla^{\varphi}\cdot(h_{\rm i}v_{\rm i})=0$ and from the definition $v_{\rm i}=\nabla^\varphi \phi_{\rm i}$, 
we have $(\nabla^\varphi)^\perp\cdot v_{\rm i}=0$. 
Let $\alpha=(\alpha_0,\alpha_1,\alpha_2)$ be a multi-index. 
As in the derivation of \eqref{QLSWEe}, we obtain 
\begin{equation}\label{preQLSWEi0}
 \begin{cases}
  \nabla^\varphi\cdot(h_{\rm i}\check{v}_{\rm i}^{(\alpha)}) = f_{{\rm i},01}^{(\alpha)} &\mbox{in}\quad (0,T)\times\ul{\cI}, \\
  (\nabla^\varphi)^\perp\cdot\check{v}_{\rm i}^{(\alpha)} = f_{{\rm i},03}^{(\alpha)} &\mbox{in}\quad (0,T)\times\ul{\cI},
 \end{cases}
\end{equation}
where in the case $\alpha_1+\alpha_2\geq1$ 
\begin{align}
 \label{deffi01}
f_{{\rm i},01}^{(\alpha)}
&= -\chi_{\rm b}\{ \nabla^\varphi\cdot\bigl( \cmfalpha\zeta_{\rm i} v_{\rm i}
 + [\mathfrak{d}^\alpha;h_{\rm i},v_{\rm i}] \bigr) 
 + {\mathcal C}^1(\mathfrak{d}^\alpha,\partial\varphi)\cdot(h_{\rm i}v_{\rm i}) \}
 + \nabla^\varphi\chi_{\rm b}\cdot h_{\rm i}  \cmfalpha v_{\rm i} , \\
  \label{deffi03}
f_{{\rm i},03}^{(\alpha)}
&= -\chi_{\rm b}   {\mathcal C}^1(\mathfrak{d}^\alpha,\partial\varphi)^\perp\cdot v_{\rm i}
 + (\nabla^\varphi\chi_{\rm b})^\perp\cdot   \cmfalpha v_{\rm i}.
\end{align}
The proposition is therefore proved for the first equation in \eqref{QLSWEi} in the case $\alpha_1+\alpha_2\geq 1$.

We now consider the case $\alpha_1=\alpha_2=0$; note that in this case, it might be sufficient to put $\chi_{\rm b}=1$ in the above equations. 
However, \eqref{preQLSWEi0} is not satisfactory in this case because of the nonzero term $f_{{\rm i},01}^{(\alpha)}$, as explained in Remark \ref{re:GU2} {\bf ii}. 
In such a case, we want to keep the divergence form of the first equation in \eqref{QLSWEi}. 
We first give a general lemma, which is valid for all types of multi-index $\alpha$, 
and in the statement of which we gather several commutator terms under the notation ${\mathfrak c}^1({\mathfrak d}^\alpha,\varphi)q_{\rm i}$, 
where ${\mathfrak c}^1({\mathfrak d}^\alpha,\varphi)$ is the scalar differential operator defined as 
\begin{equation}
\mathfrak{c}^1({\mathfrak d}^\alpha,\varphi) q
= \sum_{\beta\leq\alpha, |\beta|=|\alpha|-1}\binom{\alpha}{\beta} (\partial^\varphi\mathfrak{d}^{\alpha-\beta} q)^\star \mathfrak{d}^{\beta}\varphi 
 + \sum_{\beta\leq\alpha, 1 \leq |\beta| \leq |\alpha|-2} \binom{\alpha}{\beta}
  (\partial^\varphi\mathfrak{d}^{\beta}\varphi)^\star \mathfrak{d}^{\alpha-\beta} q,
\end{equation}
where for  $A=(a_{ij})$ a $2\times2$-matrix we write simply $A^\star=\mbox{adj}(A)=\det(A)A^{-1}$ the adjugate matrix of $A$. 
When $\vert \alpha\vert=m\geq 3$, 
the term $\mathfrak{c}^1({\mathfrak d}^\alpha,\varphi) q$ only contains derivatives of $\varphi$ and $q$ of order at most $m-1$; 
we used a lower case letter in the notation to distinguish it from the commutator terms ${\mathcal C}^1(\mathfrak{d}^\alpha,\partial \varphi)f$ and 
${\mathcal C}^2(\mathfrak{d}^\alpha,\partial_t \varphi)f$ defined in \eqref{defcomC} and that may contain derivatives of $\varphi$ and $f$ of order $m$; 
see Remark \ref{remC1C2}.

\begin{lemma}
Let $q_{\rm i}$ be a regular function defined in $\ul{\cI}$ and satisfy $\nabla^\varphi\cdot q_{\rm i}=0$. 
Then, one has 
\[
\nabla^\varphi\cdot\big( \cmfalpha q_{\rm i} + \mathfrak{c}^1({\mathfrak d}^\alpha,\varphi) q_{\rm i} \big) = 0
\]
in the tubular neighborhood $U_{\ul{\itGamma}}\cap \ul{\cI}$ of $\underline{\itGamma}$ 
in which the normal-tangential coordinates are well defined if $\alpha_1+\alpha_2>0$, and in $\ul{\cI}$ if $\alpha_1=\alpha_2=0$. 
\end{lemma}

\begin{proof}
We first transform the equation  in the normal-tangential coordinates introduced in Section \ref{subsect:NTCS}. 
Observing that 
\[
A^\star = \det(A)A^{-1} = 
\begin{pmatrix*}[r]
 a_{22} & -a_{12} \\
 -a_{21} & a_{11}
\end{pmatrix*},
\]
we have the identity $J\nabla^\varphi\cdot q = \nabla\cdot(J(\partial\varphi)^{-1}q ) = \nabla\cdot((\partial\varphi)^\star q)$, as well as 
\begin{equation}\label{F4}
\begin{cases}
 ((\partial\varphi)^\star)^{-1}(\partial\psi)^\star = \frac{1}{J}(\partial\varphi)(\partial\psi)^\star = (\partial^\varphi\psi)^\star, \\
 \nabla\cdot((\partial\psi)^\star q) = J\nabla^\varphi\cdot((\partial^\varphi\psi)^\star q), \\
 \nabla^\varphi\cdot((\partial^\varphi\psi)^\star q)) = \nabla^\varphi\cdot((\partial^\varphi q)^\star \psi))
  = \nabla^\varphi\cdot( (\nabla^\varphi\cdot q)\psi - (\psi\cdot\nabla^\varphi)q ).
\end{cases}
\end{equation}
We recall the definitions \eqref{diffeo-theta} and \eqref{theta-modify} of the maps $\theta$ and $\tilde{\theta}$. 
Due to our choice of the diffeomorphism $\varphi(t,\cdot)$, we have $\tilde{\theta}=\varphi\circ\theta$, that is, 
$\tilde{\theta}(t,r,s)=\varphi(t,\theta(r,s))$ for $(r,s)\in(-r_0,r_0)\times\mathbb{T}_L$; see \eqref{const phi}. 
Then, the equation $\nabla^\varphi\cdot q_{\rm i}=0$ can be transformed in the normal-tangential coordinates $(r,s)$ into 
$\nabla^{\tilde{\theta}}\cdot\tilde{q}_{\rm i}=0$. 
Particularly, we have 
\[
\nabla\cdot((\partial\tilde{\theta})^\star\tilde{q}_\mathrm{i}) = 0,
\]
where $\nabla$ is the nabla with respect to $(r,s)$ and so is $\partial$. 
Let $\alpha=(\alpha_0,\alpha_1,\alpha_2)$ be a multi-index. 
Applying $\dt^{\alpha_0}\partial_s^{\alpha_1}\partial_r^{\alpha_2}$ to the above equation, we see that 
\begin{align*}
0 &= \nabla\cdot\bigl( \dt^{\alpha_0}\partial_s^{\alpha_1}\partial_r^{\alpha_2}( (\partial\tilde{\theta})^\star\tilde{q}_\mathrm{i} ) \bigr) \\
&= \nabla\cdot\Biggl( \sum_{\beta\leq\alpha} \binom{\alpha}{\beta}
 (\partial\dt^{\beta_0}\partial_s^{\beta_1}\partial_r^{\beta_2}\tilde{\theta})^\star
 \dt^{\alpha_0-\beta_0}\partial_s^{\alpha_1-\beta_1}\partial_r^{\alpha_2-\beta_2}\tilde{q}_\mathrm{i} \Biggr) \\
&= \det(\partial\tilde{\theta})\nabla^{\tilde{\theta}}\cdot\Biggl( \sum_{\beta\leq\alpha} \binom{\alpha}{\beta}
 (\partial^{\tilde{\theta}}\dt^{\beta_0}\partial_s^{\beta_1}\partial_r^{\beta_2}\tilde{\theta})^\star
 \dt^{\alpha_0-\beta_0}\partial_s^{\alpha_1-\beta_1}\partial_r^{\alpha_2-\beta_2}\tilde{q}_\mathrm{i} \Biggr),
\end{align*}
where we used the first and the second identities in \eqref{F4}. 
Pulling back to the coordinate $y=\theta(r,s)$, we obtain 
\[
0 = \nabla^\varphi\cdot\Biggl( \sum_{\beta\leq\alpha} \binom{\alpha}{\beta}
 (\partial^\varphi\mathfrak{d}^{\beta}\varphi)^\star \mathfrak{d}^{\alpha-\beta} q_{\rm i} \Biggr).
\]
For the term in the case $\beta=(0,0,0)$, we have obviously $\partial^\varphi\varphi={\rm Id}_{2\times2}$. 
For the term in the case $\beta=\alpha$, by the last identity in \eqref{F4} we see that 
\begin{align*}
\nabla^\varphi\cdot( (\partial^\varphi\mathfrak{d}^{\alpha}\varphi)^\star q_{\rm i} )
&= \nabla^\varphi\cdot ( (\partial^\varphi q_{\rm i})^\star\mathfrak{d}^{\alpha}\varphi ) \\
&= \nabla^\varphi\cdot ( (\nabla^\varphi\cdot q_{\rm i})\mathfrak{d}^{\alpha}\varphi - ((\mathfrak{d}^{\alpha}\varphi)\cdot\nabla^\varphi)q_{\rm i} ) \\
&= -\nabla^\varphi\cdot \bigl( ((\mathfrak{d}^{\alpha}\varphi)\cdot\nabla^\varphi)q_{\rm i} \bigr),
\end{align*}
where we used $\nabla^\varphi\cdot q_{\rm i}=0$. 
When $|\alpha|=m$, the $m$-th order derivative $(\partial^\varphi\mathfrak{d}^{\beta}\varphi)^\star$ in the case $|\beta|=|\alpha|-1$ is still troublesome. 
However, by using the last identity in \eqref{F4} again, these terms can be replaced by terms containing only $(m-1)$-th order derivatives,  
and we have 
\begin{align*}
0 &= \nabla^\varphi\cdot\Biggl( \mathfrak{d}^{\alpha} q_{\rm i} - ((\mathfrak{d}^{\alpha}\varphi)\cdot\nabla^\varphi)q_{\rm i}
 + \sum_{\beta\leq\alpha, |\beta|=|\alpha|-1}\binom{\alpha}{\beta} (\partial^\varphi\mathfrak{d}^{\alpha-\beta} q_{\rm i})^\star \mathfrak{d}^{\beta}\varphi \\
&\makebox[4em]{}
 + \sum_{\beta\leq\alpha, 1 \leq |\beta| \leq |\alpha|-2} \binom{\alpha}{\beta}
  (\partial^\varphi\mathfrak{d}^{\beta}\varphi)^\star \mathfrak{d}^{\alpha-\beta} q_{\rm i} \Biggr),
\end{align*}
which is the identity stated in the lemma. 
\end{proof}

Using the lemma with $q_{\rm i}=h_{\rm i}v_{\rm i}$, we obtain 
\begin{equation}\label{preQLSWEi}
\nabla^\varphi\cdot\bigl( h_{\rm i} \cmfalpha v_{\rm i} + f_{{\rm i},02}^{(\alpha)} \bigr)=0,
\end{equation}
where 
\begin{equation}\label{deffi02}
f_{{\rm i},02}^{(\alpha)}
= [\mathfrak{d}^{\alpha}; h_{\rm i},v_{\rm i}] + (\cmfalpha \zeta_{\rm i})v_{\rm i} + \mathfrak{c}^1({\mathfrak d}^\alpha,\varphi)(h_{\rm i} v_{\rm i}). 
\end{equation}
Here, we note that $f_{{\rm i},02}^{(\alpha)}$ is defined only in the tubular neighborhood $U_{\ul{\itGamma}}\cap\ul{\cI}$ in the case $\alpha_1+\alpha_2\geq1$ 
due to the derivatives $\dtan$ and $\dnor$, while it is defined in all the interior domain $\ul{\cI}$ in the case $\alpha_1=\alpha_2=0$. 
Therefore, \eqref{preQLSWEi} proves the proposition for the first equation in \eqref{QLSWEi} in the remaining case $\alpha_1=\alpha_2=0$.

We now proceed to prove the proposition for the second equation of \eqref{QLSWEi}. 
We want to apply ${\mathfrak d}^\alpha$ to the relation $v_{\rm i}=\nabla^\varphi \phi_{\rm i}$ and keep the right-hand side in gradient form, 
up to terms of order $m-1$ when $|\alpha|=m\geq 3$; we will use the following notation 
\[
\mathfrak{c}^2({\mathfrak d}^\alpha,\varphi)q=
\sum_{\beta\leq\alpha, |\beta|=|\alpha|-1} \binom{\alpha}{\beta} (\partial^\varphi\mathfrak{d}^{\alpha-\beta}q)^\mathrm{T}\mathfrak{d}^{\beta}\varphi
 - \sum_{\beta\leq\alpha, 1\leq|\beta|\leq|\alpha|-2} \binom{\alpha}{\beta}
  (\partial^\varphi\mathfrak{d}^{\beta}\varphi)^\mathrm{T}\mathfrak{d}^{\alpha-\beta}q;
\]
as for $\mathfrak{c}^1({\mathfrak d}^\alpha,\varphi)q$ above, 
$\mathfrak{c}^2({\mathfrak d}^\alpha,\varphi)q$ only contains derivative of $q$ and $\varphi$ of order at most $m-1$ when $\vert \alpha\vert=m\geq 3$.

\begin{lemma} 
The following relation holds in the neighborhood $U_{\ul{\itGamma}}\cap \ul{\cI}$ of $\underline{\itGamma}$ 
in which the normal-tangential coordinates are well defined if $\alpha_1+\alpha_2>0$, and in $\ul{\cI}$ if $\alpha_1=\alpha_2=0$; 
\begin{equation}\label{preQLSWEi2}
\cmfalpha v_{\rm i}= \nabla^\varphi\bigl\{ \mathfrak{d}^{\alpha}\phi_{\rm i}
 - \sum_{\beta\leq\alpha, |\beta|\geq|\alpha|-1} \binom{\alpha}{\beta} (\mathfrak{d}^{\beta}\varphi)\cdot\mathfrak{d}^{\alpha-\beta}v_{\rm i} \bigr\}
 +\mathfrak{c}^2({\mathfrak d}^\alpha,\varphi)v_{\rm i}.
\end{equation}
\end{lemma}

\begin{proof}
With the same notations as in the proof of the previous lemma, we put $\tilde{\phi}_{\rm i}=\phi_{\rm i}\circ\theta$ and $\tilde{v}_{\rm i}=v_{\rm i}\circ\theta$. 
Then, the relation $v_{\rm i}=\nabla\phi_{\rm i}$ is transformed into $\tilde{v}_{\rm i}=\nabla^{\tilde{\theta}}\tilde{\phi}_{\rm i}$, 
so that we have $\nabla\tilde{\phi}_{\rm i}=(\partial\tilde{\theta})^\mathrm{T}\tilde{v}_{\rm i}$. 
As before, we let $\alpha=(\alpha_0,\alpha_1,\alpha_2)$ be a multi-index. 
Applying $\dt^{\alpha_0}\partial_s^{\alpha_1}\partial_r^{\alpha_2}$ and then $((\partial\tilde{\theta})^{-1})^\mathrm{T}$ to this equation, we have 
\[
\nabla^{\tilde{\theta}}\dt^{\alpha_0}\partial_s^{\alpha_1}\partial_r^{\alpha_2}\tilde{\phi}_{\rm i}
= \sum_{\beta\leq\alpha} \binom{\alpha}{\beta} (\partial^{\tilde{\theta}}\dt^{\beta_0}\partial_s^{\beta_1}\partial_r^{\beta_2}\tilde{\theta})^\mathrm{T}
 \dt^{\alpha_0-\beta_0}\partial_s^{\alpha_1-\beta_1}\partial_r^{\alpha_2-\beta_2}\tilde{v}_{\rm i}.
\]
Pulling back to the coordinate $y=\theta(r,s)$, we obtain 
\[
\nabla^\varphi\mathfrak{d}^{\alpha}\phi_{\rm i}
= \sum_{\beta\leq\alpha} \binom{\alpha}{\beta} (\partial^\varphi\mathfrak{d}^{\beta}\varphi)^\mathrm{T}\mathfrak{d}^{\alpha-\beta}v_{\rm i}.
\]
For the term in the case $\beta=\alpha$, we see that 
\begin{align*}
(\partial^\varphi\mathfrak{d}^{\alpha}\varphi)^\mathrm{T}v_{\rm i}
&= \nabla^\varphi((\mathfrak{d}^{\alpha}\varphi)\cdot v_{\rm i}) - (\partial^\varphi v_{\rm i})^\mathrm{T}\mathfrak{d}^{\alpha}\varphi \\
&= \nabla^\varphi((\mathfrak{d}^{\alpha}\varphi)\cdot v_{\rm i}) - ((\mathfrak{d}^{\alpha}\varphi)\cdot\nabla^\varphi)v_{\rm i}
 + ((\nabla^\varphi)^\perp\cdot v_{\rm i})(\mathfrak{d}^{\alpha}\varphi)^\perp \\
&= \nabla^\varphi((\mathfrak{d}^{\alpha}\varphi)\cdot v_{\rm i}) - ((\mathfrak{d}^{\alpha}\varphi)\cdot\nabla^\varphi)v_{\rm i},
\end{align*}
where we used $(\nabla^\varphi)^\perp\cdot v_{\rm i}=0$. 
When $|\alpha|=m$, the $m$-th order derivative $(\partial^\varphi\mathfrak{d}^{\beta}\varphi)^\mathrm{T}$ in the case $|\beta|=|\alpha|-1$ 
is again troublesome and treated as $(\partial^\varphi\mathfrak{d}^{\beta}\varphi)^\mathrm{T}\mathfrak{d}^{\alpha-\beta}v_{\rm i} 
= \nabla^\varphi((\mathfrak{d}^{\beta}\varphi)\cdot\mathfrak{d}^{\alpha-\beta}v_{\rm i})
 - (\partial^\varphi\mathfrak{d}^{\alpha-\beta}v_{\rm i})^\mathrm{T}\mathfrak{d}^{\beta}\varphi$. 
Therefore, we obtain 
\[
\nabla^\varphi\bigl\{ \mathfrak{d}^{\alpha}\phi_{\rm i}
 - \sum_{\beta\leq\alpha, |\beta|\geq|\alpha|-1} \binom{\alpha}{\beta} (\mathfrak{d}^{\beta}\varphi)\cdot\mathfrak{d}^{\alpha-\beta}v_{\rm i} \bigr\}
= \mathfrak{d}^{\alpha}v_{\rm i} - ((\mathfrak{d}^{\alpha}\varphi)\cdot\nabla^\varphi)v_{\rm i} -\mathfrak{c}^2({\mathfrak d}^\alpha,\varphi)v_{\rm i},
\]
which is the identity stated in the lemma. 
\end{proof}

Now, we need to multiply \eqref{preQLSWEi2} by the cut-off function $\chi_{\rm b}$ to obtain equations for good unknowns in the case $\alpha_1+\alpha_2\geq1$; 
in view of the definition of the good unknowns, this yields 
\begin{equation}\label{preQLSWEi2bis}
\check{v}_{\rm i}^{(\alpha)}=\nabla^\varphi \check\phi_{\rm i}^{(\alpha)}-(\nabla^\varphi \chi_{\rm b}) 
\big\{  \mathfrak{d}^{\alpha}\phi_{\rm i}
 - \sum_{\beta\leq\alpha, |\beta|\geq|\alpha|-1} \binom{\alpha}{\beta}
 (\mathfrak{d}^{\beta}\varphi)\cdot\mathfrak{d}^{\alpha-\beta}v_{\rm i} \big\} + \chi_{\rm b}{\mathfrak c}^2(\mathfrak{d}^\alpha,\varphi)v_{\rm i}.
\end{equation}
The term $(\nabla^\varphi\chi_{\rm b})\mathfrak{d}^{\alpha}\phi_{\rm i}$ cannot be evaluated directly by our energy 
function $E_m(t)$ and should be rewritten in term of good unknowns related to $v_{\rm i}$. 
In the case $\alpha_1\geq1$, in view of \eqref{NorTanD2} we have $\dtan=\ul{T}\cdot\nabla$ so that 
\begin{align*}
\mathfrak{d}^{\alpha}\phi_{\rm i} - (\mathfrak{d}^{\alpha}\varphi)\cdot v_{\rm i}
&= \dt^{\alpha_0}\dtan^{\alpha_1-1}\dnor^{\alpha_2}(\ul{T}\cdot\nabla\phi_{\rm i})
  - (\dt^{\alpha_0}\dtan^{\alpha_1-1}\dnor^{\alpha_2}((\ul{T}\cdot\nabla)\varphi))\cdot v_{\rm i} \\
&= \dt^{\alpha_0}\dtan^{\alpha_1-1}\dnor^{\alpha_2}(\ul{T}\cdot(\partial\varphi)^\mathrm{T}v_{\rm i})
  - (\dt^{\alpha_0}\dtan^{\alpha_1-1}\dnor^{\alpha_2}((\ul{T}\cdot\nabla)\varphi))\cdot v_{\rm i} \nonumber \\
&= [\dt^{\alpha_0}\dtan^{\alpha_1-1}\dnor^{\alpha_2},v_{\rm i}]\cdot(\ul{T}\cdot\nabla)\varphi, \nonumber
\end{align*}
where we used $(\partial\varphi)\ul{T}=(\ul{T}\cdot\nabla)\varphi$. 
Similarly, in the case $\alpha_2\geq1$ we have 
\[
\mathfrak{d}^{\alpha}\phi_{\rm i} - (\mathfrak{d}^{\alpha}\varphi)\cdot v_{\rm i}
= [\dt^{\alpha_0}\dtan^{\alpha_1}\dnor^{\alpha_2-1},v_{\rm i}]\cdot(\ul{N}\cdot\nabla)\varphi.
\]
By these equations and \eqref{preQLSWEi2bis}, we get 
\[
\check{v}_{\rm i}^{(\alpha)} = \nabla^\varphi\check{\phi}_{\rm i}^{(\alpha)} + f_{{\rm i},3}^{(\alpha)} \quad\mbox{in}\quad (0,T)\times\ul{\cI},
\]
where in the case $\alpha_1\geq1$ 
\begin{align}\label{deffi3}
f_{{\rm i},3}^{(\alpha)}
&= \chi_{\rm b} \mathfrak{c}^2(\mathfrak{d}^\alpha,\varphi)v_{\rm i} \\
&\quad\;
 - (\nabla^\varphi\chi_{\rm b})\bigl\{ [\dt^{\alpha_0}\dtan^{\alpha_1-1}\dnor^{\alpha_2},v_{\rm i}]\cdot(\ul{T}\cdot\nabla)\varphi
 - \sum_{\beta\leq\alpha, |\beta|=|\alpha|-1} \binom{\alpha}{\beta}(\mathfrak{d}^{\beta}\varphi)\cdot\mathfrak{d}^{\alpha-\beta}v_{\rm i} \bigr\}, \nonumber
\end{align}
and in the case $\alpha_1=0$ and $\alpha_2\geq1$, the term $[\dt^{\alpha_0}\dtan^{\alpha_1-1}\dnor^{\alpha_2},v_{\rm i}]\cdot(\ul{T}\cdot\nabla)\varphi$ 
in $f_{{\rm i},3}^{(\alpha)}$ should be replaced with $[\dt^{\alpha_0}\dtan^{\alpha_1}\dnor^{\alpha_2-1},v_{\rm i}]\cdot(\ul{N}\cdot\nabla)\varphi$, 
while in the case $\alpha_1=\alpha_2=0$, 
\begin{equation}\label{deffi3bis}
f_{{\rm i},3}^{(\alpha)} = \mathfrak{c}^2(\mathfrak{d}^\alpha,\varphi)v_{\rm i}.
\end{equation}
This proves the proposition for the second equation in \eqref{QLSWEi} and therefore concludes the proof of the proposition. 
\end{proof}

%-----------------------------------------------------------
\subsubsection{Expressing $\dt \check{\phi}_{\rm i}^{(\alpha)}$ in terms of $v_{\rm i}$ and $\varphi$}\label{sectexprdtphi}
In connection with Remark \ref{re:GU2} {\bf ii}, we derive here an equation for $\dt\check{\phi}_{\rm i}^{(\alpha)}$ in the case $\alpha_1+\alpha_2\geq1$. 
Let $\alpha=(\alpha_0,\alpha_1,\alpha_2)$ be a multi-index with $\alpha_1\geq1$ and put $\alpha'=(\alpha_0+1,\alpha_1-1,\alpha_2)$. 
Then, we have 
\begin{align*}
& \dt\bigl\{ \mathfrak{d}^{\alpha}\phi_{\rm i}
  - \sum_{\beta\leq\alpha, |\beta|\geq|\alpha|-1} \binom{\alpha}{\beta} (\mathfrak{d}^{\beta}\varphi)\cdot\mathfrak{d}^{\alpha-\beta}v_{\rm i} \bigr\} \\
&= \dtan\bigl\{ \mathfrak{d}^{\alpha'}\phi_{\rm i}
  - \sum_{\beta'\leq\alpha', |\beta'|\geq|\alpha'|-1} \binom{\alpha'}{\beta'} (\mathfrak{d}^{\beta'}\varphi)\cdot\mathfrak{d}^{\alpha'-\beta'}v_{\rm i} \bigr\}
 + {\mathfrak c}^3(\mathfrak{d}^\alpha,\varphi)v_{\rm i},
\end{align*}
where 
\begin{align*}
{\mathfrak c}^3(\mathfrak{d}^\alpha,\varphi)v_{\rm i}
&= - \sum_{\beta\leq\alpha, |\beta|=|\alpha|-1} \binom{\alpha}{\beta}(\mathfrak{d}^\beta\varphi)\cdot\dt\mathfrak{d}^{\alpha-\beta}v_{\rm i} 
 + \sum_{\beta'\leq\alpha', |\beta'|=|\alpha'|-1} \binom{\alpha'}{\beta'} (\mathfrak{d}^{\beta'}\varphi)\cdot\dtan\mathfrak{d}^{\alpha'-\beta'}v_{\rm i}.
\end{align*}
This equation together with $\dtan=\ul{T}\cdot\nabla=(\partial\varphi)\ul{T}\cdot\nabla^\varphi=(\dtan\varphi)\cdot\nabla^\varphi$ and \eqref{preQLSWEi2} implies 
$\dt\check{\phi}_{\rm i}^{(\alpha)} = (\dtan\varphi)\cdot\check{v}_{\rm i}^{(\alpha')} + f_{{\rm i},6}^{(\alpha)}$ with 
\[
f_{{\rm i},6}^{(\alpha)}
= \chi_{\rm b}( {\mathfrak c}^3(\mathfrak{d}^\alpha,\varphi)v_{\rm i} - (\dtan\varphi)\cdot {\mathfrak c}^2(\mathfrak{d}^\alpha,\varphi)v_{\rm i} ).
\]
Similar equation holds in the case $\alpha_2\geq1$ with $\alpha''=(\alpha_0+1,\alpha_1,\alpha_2-1)$ and we obtain 
\begin{equation}\label{Dtphii}
\dt\check{\phi}_{\rm i}^{(\alpha)} = 
\begin{cases}
 (\dtan\varphi)\cdot\check{v}_{\rm i}^{(\alpha')} + f_{{\rm i},6}^{(\alpha)} &\mbox{if}\quad \alpha_1\geq1, \\
 (\dnor\varphi)\cdot\check{v}_{\rm i}^{(\alpha'')} + f_{{\rm i},7}^{(\alpha)} &\mbox{if}\quad \alpha_2\geq1,
\end{cases}
\end{equation}
where $f_{{\rm i},7}^{(\alpha)}$ has a similar form to $f_{{\rm i},6}^{(\alpha)}$. 
We note that if $|\alpha|=m\geq3$, then $f_{{\rm i},6}^{(\alpha)}$ and $f_{{\rm i},7}^{(\alpha)}$ include derivatives of $v_{\rm i}$ and $\varphi$ 
up to at most order $m-1$.

%-----------------------------------------------------------
\subsubsection{Expressing normal derivatives in terms of tangential ones}
Similar to the case of the exterior domain in Section \ref{subsect:EGUe}, by the energy estimate given in Proposition \ref{prop:BEE}, 
we can evaluate the good unknowns $\check{v}_{\rm i}^{(\alpha_I)}$ for $\alpha_I=(\alpha_0,\alpha_1)$, 
which corresponds to the tangential derivatives in space-time. 
In order to evaluate $\check{v}_{\rm i}^{(\alpha)}$ in the presence of a normal derivative, that is when $\alpha_2\geq 1$, we use \eqref{preQLSWEi0}, 
to convert normal derivatives into tangential derivatives in space-time. 
Here, we derive such a formula. 
Plugging the decomposition \eqref{diff deco2} into \eqref{preQLSWEi0}, we obtain 
\[
\begin{cases}
 h_{\rm i}( N^\varphi\cdot\dnor\check{v}_{\rm i}^{(\alpha)} + T^\varphi\cdot\dtan\check{v}_{\rm i}^{(\alpha)} )
  = J(f_{{\rm i}, 01}^{(\alpha)} - \check{v}_{\rm i}^{(\alpha)}\cdot\nabla^\varphi\zeta_{\rm i} ), \\
 (N^\varphi)^\perp\cdot\dnor\check{v}_{\rm i}^{(\alpha)} + (T^\varphi)^\perp\cdot\dtan\check{v}_{\rm i}^{(\alpha)} = Jf_{{\rm i},03}^{(\alpha)},
\end{cases}
\]
so that 
\begin{align*}
& h_{\rm i}|N^\varphi|^2\dnor\check{v}_{\rm i}^{(\alpha)}
 + h_{\rm i}( N^\varphi\otimes T^\varphi + (N^\varphi)^\perp\otimes(T^\varphi)^\perp )\dtan\check{v}_{\rm i}^{(\alpha)} \\
&= J\{ (f_{{\rm i}, 01}^{(\alpha)} - \check{v}_{\rm i}^{(\alpha)}\cdot\nabla^\varphi\zeta_{\rm i} )N^\varphi
 + h_{\rm i }f_{{\rm i},03}^{(\alpha)}(N^\varphi)^\perp \}.
\end{align*}
Here, we let $\alpha=(j,k,l)$. 
For the terms $\dnor\check{v}_{\rm i}^{(\alpha)}$ and $\dtan\check{v}_{\rm i}^{(\alpha)}$ in the above equation, 
we use similar formulae to \eqref{DGU1} and \eqref{DGU2} to obtain 
\begin{equation}\label{QLSWEi2}
h_{\rm i}|N^\varphi|^2\check{v}_{\rm i}^{(j,k,l+1)}
 + h_{\rm i}( N^\varphi\otimes T^\varphi + (N^\varphi)^\perp\otimes(T^\varphi)^\perp )\check{v}_{\rm i}^{(j,k+1,l)} = f_{{\rm i},8}^{(\alpha)},
\end{equation}
where in the case $k+l\geq1$ 
\begin{align*}
f_{{\rm i},8}^{(\alpha)}
&= J \{ (f_{{\rm i}, 01}^{(\alpha)} - \check{v}_{\rm i}^{(\alpha)}\cdot\nabla^\varphi\zeta_{\rm i} )N^\varphi
 + h_{\rm i}f_{{\rm i},03}^{(\alpha)}(N^\varphi)^\perp \} \\
&\quad\;
 + h_{\rm i}\chi_{\rm b} \{ |N^\varphi|^2((\mathfrak{d}^\alpha\varphi)\cdot\dnor\nabla^\varphi)v_{\rm i}
  + (  N^\varphi\otimes T^\varphi + (N^\varphi)^\perp\otimes(T^\varphi)^\perp )((\mathfrak{d}^\alpha\varphi)\cdot\dtan\nabla^\varphi)v_{\rm i} \} \\
&\quad\:
 - h_{\rm i}( N^\varphi\otimes(J\nabla^\varphi\chi_{\rm b}) + (N^\varphi)^\perp\otimes(J\nabla^\varphi\chi_{\rm b})^\perp ) \cmfalpha v_{\rm i},
\end{align*}
and in the case $k=l=0$ we have a similar formula; multiply $\chi_{\rm b}$ to the first line and drop the last line in the right-hand side. 
We will use \eqref{QLSWEi2} for multi-indices $\alpha=(j,k,l)$ satisfying $|\alpha|=m-1$ to convert the normal derivative into the tangential ones. 
In this case, if $m\geq3$, then $f_{{\rm i},8}^{(\alpha)}$ includes derivatives of $v_{\rm i}$ and $\varphi$ up to at most order $m-1$.

%-------------------------------------------------------------
\subsection{Tangential differentiation of the boundary conditions  on $\ul{\itGamma}$}\label{secttangdiff}
We proceed to derive equations obtained by differentiating tangentially the boundary conditions \eqref{BC5}--\eqref{ODE5}, 
that is, by applying $\dpar^{\alpha_I}$ to these boundary conditions, where $\alpha_I=(\alpha_0,\alpha_1)$ is a multi-index.

%-------------------------------------------------------------
\subsubsection{Tangential differentiation of the boundary condition $\zeta=\zeta_{\rm i}$}
Applying $\dpar^{\alpha_I}$ to the boundary condition $\zeta=\zeta_{\rm i}$, we have $\dpar^{\alpha_I}\zeta=\dpar^{\alpha_I}\zeta_{\rm i}$, so that 
\[
\check{\zeta}^{(\alpha_I)}-\check{\zeta}_{\rm i}^{(\alpha_I)} = ((\dpar^{\alpha_I}\varphi)\cdot\nabla^\varphi)(\zeta_{\rm i}-\zeta)
 \quad\mbox{ on }\quad \ul{\itGamma},
\]
where $\check{\zeta}_{\rm i}^{(\alpha_I)} = \cmfalpha_\parallel  \zeta_{\rm i}
 = \dpar^{\alpha_I}\zeta_{\rm i}-(\dpar^{\alpha_I}\varphi)\cdot\nabla^\varphi\zeta_{\rm i}$. 
Therefore, by the same calculations as in Section \ref{subsect:DLP} we get 
$N^\varphi\cdot(\dpar^{\alpha_I}\varphi)
= -\frac{|N^\varphi|^2}{N^\varphi\cdot\nabla^\varphi(\zeta-\zeta_{\rm i})}(\check{\zeta}^{(\alpha_I)}-\check{\zeta}_{\rm i}^{(\alpha_I)})$ on $\ul{\itGamma}$. 
Here, by the decomposition \eqref{diff deco2} and the boundary condition $\zeta=\zeta_{\rm i}$ on $\ul{\itGamma}$, we have 
$JN^\varphi\cdot\nabla^\varphi(\zeta-\zeta_{\rm i})=|N^\varphi|^2(\ul{N}\cdot\nabla)(\zeta-\zeta_{\rm i})$ on $\ul{\itGamma}$. 
Therefore, we get 
\begin{equation}\label{EqPhi2}
N^\varphi\cdot(\dpar^{\alpha_I}\varphi)
= - \frac{J}{(\ul{N}\cdot\nabla)(\zeta-\zeta_{\rm i})}(\check{\zeta}^{(\alpha_I)}-\check{\zeta}_{\rm i}^{(\alpha_I)})
 \quad\mbox{on}\quad (0,T)\times\ul{\itGamma}.
\end{equation}

%-------------------------------------------------------------
\subsubsection{Tangential differentiation of the boundary condition $v=v_{\rm i}$}
Proceeding as for the boundary condition on $\zeta$, we obtain 
\[
\check{v}^{(\alpha_I)}-\check{v}_{\rm i}^{(\alpha_I)} = ((\dpar^{\alpha_I}\varphi)\cdot\nabla^\varphi)(v_{\rm i}-v)
 \quad\mbox{ on }\quad \ul{\itGamma},
\]
Therefore, by the same calculations as in Section \ref{subsect:DLP} we get 
\begin{equation}\label{QLBC0}
\begin{cases}
 N^\varphi\cdot(h\check{v}^{(\alpha_I)}+w\check{\zeta}^{(\alpha_I)})
  = N^\varphi\cdot(h_{\rm i}\check{v}_{\rm i}^{(\alpha_I)}+f_{{\rm i},4}^{(\alpha_I)}) &\mbox{on}\quad (0,T)\times\ul{\itGamma}, \\
 (N^\varphi)^\perp\cdot(h\check{v}^{(\alpha_I)})=(N^\varphi)^\perp\cdot(h_{\rm i}\check{v}_{\rm i}^{(\alpha_I)}) &\mbox{on}\quad (0,T)\times\ul{\itGamma}, 
\end{cases}
\end{equation}
where $f_{{\rm i},4}^{(\alpha_I)}= w_{\rm i}  \check{\zeta}_{\rm i}^{(\alpha_I)}  = w_{\rm i} \cmfalpha_\parallel  \zeta_{\rm i}$.

%-------------------------------------------------------------
\subsubsection{Tangential differentiation of the evolution equation for $\psi_{\rm i}$}
As for \eqref{ODE5}, we write it as $\dt\phi_{\rm i}-(\dt\varphi)\cdot v+\frac12|v|^2+\gr\zeta=0$ on $\ul{\itGamma}$. 
Applying $\dpar^{\alpha_I}$ to this equation and using $v=v_{\rm i}$ on $\ul{\itGamma}$, we see that 
\begin{align*}
0 &= \dt(\dpar^{\alpha_I}\phi_{\rm i}) - \dpar^{\alpha_I}((\dt\varphi)\cdot v) + v\cdot\dpar^{\alpha_I}v + \gr\dpar^{\alpha_I}\zeta + [\dpar^{\alpha_I};v,\cdot v] \\
&= \dt\bigl\{ \check{\phi}_{\rm i}^{(\alpha_I)}
 + \sum_{\beta_I\leq\alpha_I, |\beta_I|\geq|\alpha_I|-1} \binom{\alpha_I}{\beta_I}(\dpar^{\beta_I}\varphi)\cdot\dpar^{\alpha_I-\beta_I}v_{\rm i} \bigr\} 
 - \sum_{\beta_I\leq\alpha_I} \binom{\alpha_I}{\beta_I}(\dpar^{\beta_I}\dt\varphi)\cdot\dpar^{\alpha_I-\beta_I}v \\
&\quad\;
 + v\cdot( \check{v}^{(\alpha_I)}+((\dpar^{\alpha_I}\varphi)\cdot\nabla^\varphi)v )
 + \gr( \check{\zeta}^{(\alpha_I)}+(\dpar^{\alpha_I}\varphi)\cdot\nabla^\varphi\zeta )+ [\dpar^{\alpha_I};v,\cdot v] \\
&= \dt\check{\phi}_{\rm i}^{(\alpha_I)} + w\cdot\check{v}^{(\alpha_I)} + \gr\check{\zeta}^{(\alpha_I)}
 + (\dpar^{\alpha_I}\varphi)\cdot(\dt v +\gr\nabla^\varphi\zeta) + w\cdot((\dpar^{\alpha_I}\varphi)\cdot\nabla^\varphi)v + [\dpar^{\alpha_I};v,\cdot v] \\
&\quad\;
 + \sum_{\beta_I\leq\alpha_I, |\beta_I|=|\alpha_I|-1} \binom{\alpha_I}{\beta_I}(\dpar^{\beta_I}\varphi)\cdot\dpar^{\alpha_I-\beta_I}\dt v_{\rm i} 
 - \sum_{\beta_I\leq\alpha_I, 1\leq|\beta_I|\leq|\alpha_I|-2} \binom{\alpha_I}{\beta_I}(\dpar^{\beta_I}\dt\varphi)\cdot\dpar^{\alpha_I-\beta_I}v_{\rm i}.
\end{align*}
Here, by the same calculation as in Section \ref{subsect:DLP} we have 
\begin{align*}
&(\dpar^{\alpha_I}\varphi)\cdot(\dt v +\gr\nabla^\varphi\zeta) + w\cdot((\dpar^{\alpha_I}\varphi)\cdot\nabla^\varphi)v \\
&= (\dpar^{\alpha_I}\varphi)\cdot\bigl( \dt^\varphi v + \nabla^\varphi(\tfrac12|v|^2+\mbox{\tt g}\zeta) + ((\nabla^\varphi)^\perp\cdot v)(\dt\varphi)^\perp \bigr) \\
&= 0. 
\end{align*}
Therefore, recalling that we defined $\check\psi_{\rm i}^{(\alpha_I)}$ in Section \ref{sectdefGU} as the trace of 
$\check{\phi}_{\rm i}^{(\alpha_I)}$ on $\ul{\itGamma}$, we get
\begin{equation}\label{QLBC00}
\check{\psi}_{\rm i}^{(\alpha_I)} + w\cdot\check{v}^{(\alpha_I)} + \gr\check{\zeta}^{(\alpha_I)} =f_{{\rm i},5}^{(\alpha_I)}
 \quad\mbox{on}\quad (0,T)\times\ul{\itGamma},
\end{equation}
where 
\begin{align*}
f_{{\rm i},5}^{(\alpha_I)}
&= \sum_{\beta_I\leq\alpha_I, 1\leq|\beta_I|\leq|\alpha_I|-2} \binom{\alpha_I}{\beta_I} (\dpar^{\beta_I}\dt\varphi)\cdot\dpar^{\alpha_I-\beta_I}v_{\rm i} \\
&\quad\;
 - \sum_{\beta_I\leq\alpha_I, |\beta_I|=|\alpha_I|-1} \binom{\alpha_I}{\beta_I}(\dpar^{\beta_I}\varphi)\cdot\dpar^{\alpha_I-\beta_I}\dt v_{\rm i} 
 - [\dpar^{\alpha_I};v_{\rm i},\cdot v_{\rm i}].
\end{align*}

%-----------------------------------------------------------
\subsection{Equations for good unknowns away from the boundary}\label{subsect:EqGUb}
We proceed to derive equations for the good unknowns $\check{u}_{\rm r}^{(\alpha)},\check{v}_{\rm i,r}^{(\alpha)}$, and $\check{\phi}_{\rm i,r}^{(\alpha)}$, 
whose supports are away from the boundary $\ul{\itGamma}$. 
In this case, we do not need to use the normal-tangential coordinates. 
As in the derivation of \eqref{QLSWEe}, applying $\chi_{\rm e}\bm{\partial}^\alpha$ to \eqref{NLSWEe5}, we get 
\begin{equation}\label{QLSWEr}
\begin{cases}
 \dt\check{\zeta}_{\rm r}^{(\alpha)}+\nabla^\varphi\cdot(h\check{v}_{\rm r}^{(\alpha)}+w\check{\zeta}_{\rm r}^{(\alpha)}) = f_{{\rm r},1}^{(\alpha)}
  &\mbox{in}\quad (0,T)\times\ul{\cE}, \\
 \dt\check{v}_{\rm r}^{(\alpha)}+\nabla^\varphi(w\cdot\check{v}_{\rm r}^{(\alpha)}+\gr\check{\zeta}_{\rm r}^{(\alpha)}) = f_{{\rm r},2}^{(\alpha)}
  &\mbox{in}\quad (0,T)\times\ul{\cE}, 
\end{cases}
\end{equation}
where $f_{{\rm r},1}^{(\alpha)}$ and $f_{{\rm r},2}^{(\alpha)}$ are obtained exactly as $f_{1}^{(\alpha)}$ and $f_{2}^{(\alpha)}$ that appear in \eqref{QLSWEe}; 
for instance, 
\begin{align*}
f_{{\rm r},1}^{(\alpha)}
&= -\chi_{\rm e}\{ {\mathcal C}^2(\bm{\partial}^\alpha,\dt\varphi)\zeta
  - (\dt\varphi)\cdot  {\mathcal C}^1(\bm{\partial}^\alpha,\partial\varphi)\zeta  
   +{\mathcal C}^1(\bm{\partial}^\alpha,\partial\varphi)\cdot (hv)
 + \nabla^\varphi\cdot( [\bm{\partial}^\alpha;h,v] ) \} \\
&\quad\;
 + (\dt^\varphi\chi_{\rm e}) \check{\bm{\partial}}^\alpha\zeta 
 + \nabla^\varphi\chi_{\rm e}\cdot\bigl( (\check{\bm{\partial}}^\alpha \zeta)v + h (\check{\bm{\partial}}^\alpha v)  \bigr)
 - (\nabla^\varphi\cdot\dt\varphi)\check{\zeta}^{(\alpha)}.
\end{align*}
Similarly, as in the derivation of \eqref{QLSWEi}, we get for all $\alpha$ such that $\alpha_1+\alpha_2\geq 1$ that 
\begin{equation}\label{QLSWEir}
\begin{cases}
 \nabla^\varphi\cdot(h_{\rm i}\check{v}_{\rm i,r}^{(\alpha)}) = f_{{\rm i,r},1}^{(\alpha)} &\mbox{in}\quad (0,T)\times\ul{\cI}, \\
 \check{v}_{\rm i,r}^{(\alpha)} = \nabla^\varphi\check{\phi}_{\rm i,r}^{(\alpha)} + f_{{\rm i,r},3}^{(\alpha)} &\mbox{in}\quad (0,T)\times\ul{\cI},
\end{cases}
\end{equation}
where $f_{{\rm i},{\rm r},1}^{(\alpha)}$ is easily deduced from the expression \eqref{deffi01} for $f_{{\rm i},01}^{(\alpha)}$,
\begin{align*}
f_{{\rm i},{\rm r},1}^{(\alpha)}
&= -\chi_{\rm i}\{ \nabla^\varphi\cdot\bigl( \check{\bm{\partial}}^\alpha\zeta_{\rm i} v_{\rm i}
 + [\mathfrak{d}^\alpha;h_{\rm i},v_{\rm i}] \bigr) 
 + {\mathcal C}^1(\bm{\partial}^\alpha,\partial\varphi)\cdot(h_{\rm i}v_{\rm i}) \}
 + \nabla^\varphi\chi_{\rm b}\cdot h_{\rm i} \check{\bm{\partial}}^\alpha v_{\rm i} , 
\end{align*}
while $f_{{\rm i},{\rm r},3}^{(\alpha)}$ is similarly deduced from the expression \eqref{deffi3} for $f_{{\rm i},3}^{(\alpha)}$. 
Moreover, as in the derivation of \eqref{Dtphii}, we get also 
\begin{equation}\label{Dtphiir}
\dt\check{\phi}_{\rm i,r}^{(\alpha)} = 
\begin{cases}
 (\partial_1\varphi)\cdot\check{v}_{\rm i,r}^{(\alpha')} + f_{{\rm i,r},6}^{(\alpha)} &\mbox{if}\quad \alpha_1\geq1, \\
 (\partial_2\varphi)\cdot\check{v}_{\rm i,r}^{(\alpha'')} + f_{{\rm i,r},7}^{(\alpha)} &\mbox{if}\quad \alpha_2\geq1,
\end{cases}
\end{equation}
where $f_{{\rm i,r},6}^{(\alpha)}$ and $f_{{\rm i,r},6}^{(\alpha)}$ are deduced from the quantities $f_{{\rm i},6}^{(\alpha)}$ and 
$f_{{\rm i},6}^{(\alpha)}$ that appear in \eqref{Dtphii} with straightforward adaptations. 
Finally, we remark that the case $\alpha_1=\alpha_2=0$ is covered by Proposition \ref{propGUinterior} 
because in that case the good unknowns are defined without using any localization function.

%-----------------------------------------------------------
\subsection{Equations for the derivatives of $\gamma$}\label{subsect:Eqg}
To conclude this section, we derive equations for derivatives of $\gamma$ under Assumption \ref{ass:CL} on the unknown curve $\itGamma(t)$. 
Such equations are essentially given by \eqref{EqPhi2}. 
Here, we relate the derivative $\dt^{\alpha_0}\partial_s^{\alpha_1}\gamma$ to $(N^\varphi\cdot\dpar^{\alpha_I}\varphi)_{\vert_{\ul{\itGamma}}}$, 
which is the quantity that appears in the left-hand side of \eqref{EqPhi2}. 
To this end, we recall that the diffeomorphism $\varphi(t,\cdot)$ was constructed from $\gamma(t,\cdot)$ by \eqref{const phi}. 
By a straightforward calculation, we see that 
\begin{align*}
(\partial\theta)_{\vert_{r=0}} &= 
\begin{pmatrix}
  \ul{x}_2'(s) & \ul{x}_1'(s) \\
 -\ul{x}_1'(s) & \ul{x}_2'(s)
\end{pmatrix}, \\
(\partial\tilde{\theta})_{\vert_{r=0}} &=
\begin{pmatrix}
  \ul{x}_2'(s) & (1+\kappa(s)\gamma(t,s))\ul{x}_1'(s) + (\partial_s\gamma(t,s))\ul{x}_2'(s) \\
 -\ul{x}_1'(s) & (1+\kappa(s)\gamma(t,s))\ul{x}_2'(s) - (\partial_s\gamma(t,s))\ul{x}_1'(s) 
\end{pmatrix},
\end{align*}
where we used $(\partial_r\gamma^{\rm ext})_{\vert_{r=0}} = 0$ and $\ul{x}''(s) = \kappa(s)\ul{x}'(s)^\perp$. 
Therefore, in view of ${N^\varphi}_{\vert_{x=\ul{x}(s)}}
= \bigl(J((\partial\varphi)^{-1})^\mathrm{T}\ul{N}\bigr)_{\vert_{x=\ul{x}(s)}}
= \bigl( \det((\partial\widetilde{\theta})(\partial\theta)^{-1}) ((\partial\theta)(\partial\widetilde{\theta})^{-1})^\mathrm{T}\bigr)_{\vert_{x=\ul{x}(s)}}
 (-\ul{x}'(s))^\perp$, 
we obtain 
\begin{equation}\label{Nphi2}
{N^\varphi}_{\vert_{x=\ul{x}(s)}}
 = (1+\kappa(s)\gamma(t,s))\ul{N} - (\partial_s\gamma)(t,s)\ul{N}^\perp.
\end{equation}
On the other hand, for a multi-index $\alpha_I=(\alpha_0,\alpha_1)$, we see that 
\begin{align*}
(\dpar^{\alpha_I}\varphi)_{\vert_{x=\ul{x}(s)}}
&= (\dt^{\alpha_0}\partial_s^{\alpha_1}(\varphi\circ\theta))_{\vert_{r=0}} \\
&= \dt^{\alpha_0}\partial_s^{\alpha_1}\bigl( \ul{x}(s)+\gamma(t,s)(-\ul{x}'(s))^\perp \bigr) \\
&= (\dt^{\alpha_0}\partial_s^{\alpha_1}\gamma)\ul{N}
 + [\partial_s^{\alpha_1},(-\ul{x}')^\perp]\dt^{\alpha_0}\gamma + \dt^{\alpha_0}\partial_s^{\alpha_1}\ul{x}.
\end{align*}
Therefore, we get 
\begin{equation}\label{gamma eq2}
(1+\kappa\gamma)\dt^{\alpha_0}\partial_s^{\alpha_1}\gamma
=(N^\varphi\cdot\dpar^{\alpha_I}\varphi)_{\vert_{x=\ul{x}(s)}} + b^{(\alpha_I)},
\end{equation}
where 
\[
b^{(\alpha_I)} = -( (1+\kappa\gamma)\ul{N} - (\partial_s\gamma)\ul{N}^\perp )\cdot
 \bigl( [\partial_s^{\alpha_1},(-\ul{x}')^\perp]\dt^{\alpha_0}\gamma + \dt^{\alpha_0}\partial_s^{\alpha_1}\ul{x} \bigr).
\]
We note that if $|\alpha_I|=m\geq2$, then $b^{(\alpha_I)}$ includes derivatives of $\gamma$ up to at most order $m-1$.

\begin{remark}\label{re:EEg}
Let us use \eqref{EqPhi2} and \eqref{gamma eq2} for $\alpha_I=(1,0)$. 
Then, noting that $b^{(1,0)}=0$, $J_{\vert_{x=\ul{x}(s)}}=1+\kappa\gamma$, $\check{\zeta}^{(1,0)}=\dt^\varphi\zeta=-\nabla^\varphi\cdot(hv)$, and 
$\zeta_{\rm i}^{(1,0)}=\dt^\varphi\zeta_{\rm i}=0$, we obtain 
\[
\dt\gamma =  \left( \frac{\nabla^\varphi\cdot(hv)}{\ul{N}\cdot\nabla(\zeta-\zeta_{\rm i})} \right)_{\vert_{x=\ul{x}(s)}}.
\]
This is the evolution equation for the unknown function $\gamma$, that is, for the unknown curve $\itGamma(t)$. 
\end{remark}

%----------------------------------------------------------------------------------------------------------------------
\section{A priori estimates}\label{sect:APE}
In this last section, we prove the a priori estimate given in Theorem \ref{th:main}. 
Let $m\geq 3$ and $u=(\zeta,v^{\rm T})^{\rm T}, v_{\rm i}, \ul{p}_{\rm i}$, and $\gamma$ be a regular solution 
to the shallow water model \eqref{nlswe}--\eqref{mc} satisfying initially \eqref{EstIni}. 
We recall that the diffeomorphism $\varphi(t,\cdot)$ has been constructed from $\gamma$ by \eqref{const phi}. 
Then, by the analysis in Sections \ref{subsect:EF2} and \ref{sectfixing}, 
there exists a scalar function $\psi_{\rm i}$ such that \eqref{NLSWEe5}--\eqref{ODE5} hold. 
With $ \ul{\cI}_0$, $c_0$ and $M_0$ as in Assumption \ref{ass:BSB}, 
and recalling that $r_0$ is the width of the tubular neighborhood in which the normal-tangential coordinates are defined, 
we suppose also that the solution satisfies 
\begin{equation}\label{Hyp1}
 \begin{cases}
  \{ \varphi(t,x) \,|\, x\in\ul{\cI} \} \subset \ul{\cI}_0 \mbox{ for } 0\leq t\leq T, \\
  \inf_{(t,x)\in(0,T)\times\ul{\cE}}(\gr h(t,x)-|w(t,x)|^2) \geq c_0, \\
  \inf_{(t,x)\in(0,T)\times\ul{\itGamma}}| \ul{N}\cdot(\nabla\zeta-\nabla\zeta_{\rm i})(t,x)| \geq c_0, \\
  \sup_{0<t<T}|\gamma(t)|_{L^\infty(\mathbb{T}_L)} \leq \eta_0r_0, \\
  \sup_{0<t<T}(\|u(t)\|_{m-1,{\rm e}} + \|v_{\rm i}(t)\|_{m-1,{\rm i}} + |\gamma(t)|_{m-1}) \leq 2M_0,
 \end{cases}
\end{equation}
where as in the statement of Theorem \ref{th:main}, one has $0<\eta_0^{\rm in}<\eta_0<1$, and 
\begin{equation}\label{Hyp2}
 \begin{cases}
  \|\bm{\partial}(h,w)\|_{L^1(0,T;L^\infty(\ul{\cE}))} + \|\bm{\partial} h_{\rm i}\|_{L^1(0,T;L^\infty(\ul{\cI}))}
   + \|\bm{\partial}\partial\varphi\|_{L^1(0,T;L^\infty(\R^2))} \leq M_1, \\
  \sup_{0<t<T}E_m(t) \leq M_2, \quad \int_0^T|\gamma(t)|_m^2{\rm d}t \leq M_3,
 \end{cases}
\end{equation}
where the time $T$ and the constants $M_1, M_2$, and $M_3$ will be defined later. 
Then, Assumption \ref{ass:CL} is satisfied with a positive constant $\delta_0$ depending only on $\eta_0$ and $m_0=1$. 
We note that the first condition in \eqref{Hyp1} guarantees that the function $\zeta_{\rm i}(t,\cdot)=Z_{\rm w}\circ\varphi(t,\cdot)$ is defined as a 
function of $(t,x)\in(0,T)\times\ul{\cI}$. 
In the following, we simply denote positive constants depending on $c_0,M_0,\eta_0$, and $m$ by the same symbol $C_0$, which may change from line to line, 
whereas we track carefully the dependence of the constants $M_1$, $M_2$, and $M_3$. 
Then, by the analysis of Section \ref{subsect:ChoiceD}, especially, Lemma \ref{lem:EstDM} together with the Sobolev embedding theorem we have 
\[
\begin{cases}
 \inf_{x\in\R^2}J(t,x) = \inf_{x\in\R^2}\det(\partial\varphi(t,x)) \geq C_0^{-1}, \\
 \sum_{|\alpha| \leq m-2} \|\bm{\partial}^\alpha\widetilde{\varphi}(t)\|_{L^\infty(\R^2)} \leq C_0, \\
 \sum_{|\alpha| \leq m-1} \|\bm{\partial}^\alpha\widetilde{\varphi}(t)\|_{L^2\cap L^4(\R^2)} \leq C_0, \\
 \sum_{|\alpha| = m} \|\bm{\partial}^\alpha\widetilde{\varphi}(t)\|_{L^2\cap L^4(\R^2)} \leq C_0|\gamma(t)|_m, \\
 \sum_{|\alpha| \leq m-3}( \|\bm{\partial}^\alpha u(t)\|_{L^\infty(\ul{\cE})} + \|\bm{\partial}^\alpha v_{\rm i}(t)\|_{L^\infty(\ul{\cI})} ) \leq C_0
\end{cases}
\]
for $0\leq t\leq T$. 
We will use these estimates freely in the following without any comments. 
Moreover, to simplify the notation, for non-negative integer $k$ we write $|\bm{\partial}^k f|=\sum_{|\alpha|=k}|\bm{\partial}^\alpha f|$, and 
similar notation will be used in the following. 
Such a simplification causes no confusion.

We first show in Section \ref{sectCEN} how to control various derivatives of the solution in terms of the energy $E_m(t)$
introduced in Section \ref{subsect:result}, and we also define equivalent energy norms, one of them involving only tangential derivatives. 
We then control this latter energy in Section \ref{sectAppl}, using the linear estimate of Proposition \ref{prop:BEE}, 
while the evolution of the surface parametrization is controled by using Proposition \ref{prop:ABR}. 
The lower order terms involved in these energy estimates are controled in Section \ref{sectestLOT}. 
We then conclude the proof of Theorem \ref{th:main} by proving that the conditions \eqref{Hyp1} and \eqref{Hyp2} remain satisfied 
on a time interval $[0,T]$ with $T>0$ as in the statement of the theorem. 
This is done in Section \ref{secttransv} for the transversality condition which requires special care in the critical case $m=3$ 
and in Section \ref{sectAPE} for the other ones.

%-----------------------------------------------------------
\subsection{Controls in terms of the energy norms}\label{sectCEN}
We show in the following lemma how to control various norms in terms of the energy $E_m(t)$ introduced in Section \ref{subsect:result}.

\begin{lemma}\label{lem:EstDuv}
For $0\leq t\leq T$, we have 
\begin{enumerate}
\item[{\rm (i)}]
$\|\bm{\partial}^m u(t)\|_{L^2(\ul{\cE})} + \|\bm{\partial}^m v_{\rm i}(t)\|_{L^2(\ul{\cI})} \leq C_0(1+|\gamma(t)|_m)E_m(t)$;
\item[{\rm (ii)}]
$\|\bm{\partial}^{m-1} u(t)\|_{L^2\cap L^4(\ul{\cE})} + \|\bm{\partial}^{m-1} v_{\rm i}(t)\|_{L^2\cap L^4(\ul{\cI})}
 \leq C_0 E_m(t)^{1/2} ( \|u(t)\|_{m-1,{\rm e}} + \|v_{\rm i}(t)\|_{m-1,{\rm i}})^{1/2}$;
\item[{\rm (iii)}]
$\|\bm{\partial}^{m-2} u(t)\|_{L^\infty(\ul{\cE})} + \|\bm{\partial}^{m-2} v_{\rm i}(t)\|_{L^\infty(\ul{\cI})}
 \leq C_0 E_m(t)^{1/4} ( \|u(t)\|_{m-1,{\rm e}} + \|v_{\rm i}(t)\|_{m-1,{\rm i}})^{3/4}$;
\item[{\rm (iv)}]
$\|\partial\bm{\partial}^{m-1} u(t)\|_{L^2(\ul{\cE})} + \|\partial\bm{\partial}^{m-1} v_{\rm i}(t)\|_{L^2(\ul{\cI})} \leq C_0(1+|\gamma(t)|_m^{1/2})E_m(t)$.
\end{enumerate}
\end{lemma}

\begin{proof}
We give a proof of the estimates only for $u$, because the estimates for $v_{\rm i}$ can be obtained in the same way. 
Moreover, we write $\|\cdot\|_{L^p}$ instead of $\|\cdot\|_{L^p(\ul{\cE})}$ for simplicity. 
We see that 
\begin{align*}
\|\bm{\partial}^m u\|_{L^2}
&\leq \|\bm{\partial}^m u - ((\bm{\partial}^m\varphi)\cdot\nabla^\varphi)u\|_{L^2}
 + \|((\bm{\partial}^m\varphi)\cdot\nabla^\varphi)u\|_{L^2} \\
&\leq E_m(t) + \|\bm{\partial}^m\varphi\|_{L^4} \|(\partial\varphi)^{-1}\|_{L^\infty} \|\partial u\|_{L^4} \\
&\leq E_m(t) + C_0|\gamma|_m \|u\|_{H^2} \\
&\leq C_0(1+|\gamma|_m)E_m,
\end{align*}
where we used the Sobolev embedding theorem $H^1(\ul{\cE}) \hookrightarrow L^4(\ul{\cE})$. 
This shows (i).

We proceed to evaluate $\|\bm{\partial}^{m-1}u(t)\|_{L^4}$ and $\|\bm{\partial}^{m-2}u(t)\|_{L^\infty}$. 
By the Gagliardo--Nirenberg interpolation inequality $\|f\|_{L^4} \lesssim \|\nabla f\|_{L^2}^{1/2}\|f\|_{L^2}^{1/2} + \|f\|_{L^2}$, we see that 
\begin{align*}
&\|\bm{\partial}^{m-1} u\|_{L^4} \\
&\leq \|\bm{\partial}^{m-1} u - ((\bm{\partial}^{m-1}\varphi)\cdot\nabla^\varphi)u\|_{L^4}
 + \|((\bm{\partial}^{m-1}\varphi)\cdot\nabla^\varphi)u\|_{L^4} \\
&\leq C_0 \bigl\{ \|\partial(\bm{\partial}^{m-1} u - ((\bm{\partial}^{m-1}\varphi)\cdot\nabla^\varphi)u)\|_{L^2}^{1/2}
  \|\bm{\partial}^{m-1} u - ((\bm{\partial}^{m-1}\varphi)\cdot\nabla^\varphi)u\|_{L^2}^{1/2} \\
&\quad\;
 + \|\bm{\partial}^{m-1} u - ((\bm{\partial}^{m-1}\varphi)\cdot\nabla^\varphi)u\|_{L^2}
 + \|\bm{\partial}^{m-1}\varphi\|_{L^4} \|\partial u\|_{L^\infty} \bigr\} \\
&\leq C_0 \bigl\{ \bigl( \|\bm{\partial}^{m} u - ((\bm{\partial}^{m}\varphi)\cdot\nabla^\varphi)u\|_{L^2}
 + \|\bm{\partial}^{m-1}\varphi\|_{L^4} \|\partial^2u\|_{L^4} 
 + \|\bm{\partial}^{m-1}\varphi\|_{L^4} \|\partial^2\varphi\|_{L^4} \|\partial u\|_{L^\infty} \bigr)^{1/2}  \\
&\qquad
  \times\bigl( \|\bm{\partial}^{m-1}u\|_{L^2} + \|\bm{\partial}^{m-1}\varphi\|_{L^4} \|\partial u\|_{L^4} \bigr)^{1/2} \\
&\quad\;
 + \|u\|_{m-1,{\rm e}} + \|\bm{\partial}^{m-1}\varphi\|_{L^4} \|\partial u\|_{L^4} + \|\partial u\|_{L^\infty} \bigr\} \\
&\leq C_0 \bigl\{ (E_m+\|\bm{\partial}^{m-1} u\|_{L^4}+\|\partial u\|_{L^\infty} )^{1/2} \|u\|_{m-1,{\rm e}}^{1/2} + \|\partial u\|_{L^\infty} \bigr\},
\end{align*}
which yields $\|\bm{\partial}^{m-1} u\|_{L^4} \leq C_0 ( E_m^{1/2}\|u\|_{m-1,{\rm e}}^{1/2} + \|\partial u\|_{L^\infty} )$. 
Particularly, we get 
\begin{equation}\label{CalIneq1}
\|\bm{\partial}^{m-1} u\|_{L^4} \leq C_0 ( E_m^{1/2} \|u\|_{m-1,{\rm e}}^{1/2} + \|\bm{\partial}^{m-2} u\|_{L^\infty} ).
\end{equation}
Using this, the Gagliardo--Nirenberg interpolation inequality, and the Sobolev embedding theorem, we see that 
\begin{align*}
\|\bm{\partial}^{m-2} u\|_{L^\infty}
&\leq C_0 ( \|\bm{\partial}^{m-1} u\|_{L^4}^{1/2} \|\bm{\partial}^{m-2} u\|_{L^4}^{1/2} + \|\bm{\partial}^{m-2} u\|_{L^4} ) \\
&\leq C_0 ( \|\bm{\partial}^{m-1} u\|_{L^4}^{1/2} \|u\|_{m-1,{\rm e}}^{1/2} + \|u\|_{m-1,{\rm e}} ) \\
&\leq C_0 ( E_m^{1/4}\|u\|_{m-1,{\rm e}}^{3/4} + \|\bm{\partial}^{m-2}u\|_{L^\infty}^{1/2}\|u\|_{m-1,{\rm e}}^{1/2} ),
\end{align*}
which yields $\|\bm{\partial}^{m-2} u\|_{L^\infty} \leq C_0E_m^{1/4}\|u\|_{m-1,{\rm e}}^{3/4}$. 
This shows (iii). 
Plugging this into \eqref{CalIneq1}, we obtain (ii).

Finally, by Lemma \ref{lem:EstDM} and (iii) we see that 
\begin{align*}
\|((\partial\bm{\partial}^{m-1}\varphi)\cdot\nabla^\varphi)u\|_{L^2}
&\leq \|\partial\bm{\partial}^{m-1}\varphi\|_{L^2} \|(\partial\varphi)^{-1}\|_{L^\infty} \|\partial u\|_{L^\infty} \\
&\leq C_0|\gamma|_m^{1/2}E_m.
\end{align*}
Therefore, in the same way as the proof of (i) we obtain (iv). 
\end{proof}

For a multi-index $\alpha=(\alpha_0,\alpha_1,\alpha_2)$ satisfying $|\alpha|=m$, let $\check{u}^{(\alpha)}, \check{v}_{\rm i}^{(\alpha)}, 
\check{u}_{\rm r}^{(\alpha)}$, and $\check{v}_{\rm i,r}^{(\alpha)}$ be the good unknowns introduced in Section \ref{subsect:defGU}. 
The energy $E_{m}(t)$ is defined in terms of $\check{\bm{\partial}}^\alpha u$ and $\check{\bm{\partial}}^\alpha v_{\rm i}$; 
in view of the results of the previous section, it is convenient to use instead quantities that involve the good unknowns. 
We therefore define other energy functions $\check{E}_m(t)$ and $\check{E}_{m,\parallel}(t)$ of order $m$ by 
\begin{align*}
\check{E}_m(t)
&= \sum_{|\alpha|=m}( \|\check{u}^{(\alpha)}(t)\|_{L^2(\ul{\cE})} + \|\check{v}_{\rm i}^{(\alpha)}(t)\|_{L^2(\ul{\cI})} ) \\
&\quad\;
 + \sum_{|\alpha|=m, \alpha_0\leq m-1}( \|\check{u}_{\rm r}^{(\alpha)}(t)\|_{L^2(\ul{\cE})} + \|\check{v}_{\rm i,r}^{(\alpha)}(t)\|_{L^2(\ul{\cI})} )
 + \|u(t)\|_{m-1,{\rm e}} + \|v_{\rm i}(t)\|_{m-1,{\rm i}}
\end{align*}
and 
\begin{align*}
\check{E}_{m,\parallel}(t)
&= \sum_{|\alpha_I|=m}( \|\check{u}^{(\alpha_I)}(t)\|_{L^2(\ul{\cE})} + \|\check{v}_{\rm i}^{(\alpha_I)}(t)\|_{L^2(\ul{\cI})} ) \\
&\quad\;
 + \sum_{|\alpha|=m, \alpha_0\leq m-1}( \|\check{u}_{\rm r}^{(\alpha)}(t)\|_{L^2(\ul{\cE})} + \|\check{v}_{\rm i,r}^{(\alpha)}(t)\|_{L^2(\ul{\cI})} )
 + \|u(t)\|_{m-1,{\rm e}} + \|v_{\rm i}(t)\|_{m-1,{\rm i}},
\end{align*}
where $\alpha_I=(\alpha_0,\alpha_1)$. 
Here, as in the proof of Lemma \ref{lem:EstDuv} (i), for any multi-index $\alpha$ satisfying $1\leq|\alpha|\leq m-1$ we have 
\[
\begin{cases}
 \|((\bm{\partial}^\alpha\varphi)\cdot\nabla^\varphi)u(t)\|_{L^2(\ul{\cE})} \leq C_0\|u(t)\|_{m-1,{\rm e}}, \\
 \|((\bm{\partial}^\alpha\varphi)\cdot\nabla^\varphi)v_{\rm i}(t)\|_{L^2(\ul{\cI})} \leq C_0\|v_{\rm i}(t)\|_{m-1,{\rm i}}. 
\end{cases}
\]
In view of this together with \eqref{NorTanD2} and \eqref{diff deco}, we see that $E_m(t)$ and $\check{E}_m(t)$ are equivalent, that is, we have 
\begin{equation}\label{eqEcEm}
C_0^{-1}E_m(t) \leq \check{E}_m(t) \leq C_0E_m(t). 
\end{equation}
Since the good unknowns $\check{u}^{(\alpha)}$ and $\check{v}_{\rm i}^{(\alpha)}$ satisfy \eqref{QLSWEe3} and \eqref{QLSWEi2} and we have 
$\det( A(u,N^\varphi,\bm{a}) ) \geq C_0^{-1}$ and $h_{\rm i}|N^\varphi|^2 \geq C_0^{-1}$, 
we can convert the normal derivatives into the tangential ones in space-time modulo lower order terms. 
Therefore, we have 
\begin{equation}\label{EquivE}
\check{E}_{m,\parallel}(t) \leq \check{E}_m(t) \leq C_0( \check{E}_{m,\parallel}(t) + R_1(t) ),
\end{equation}
where $R_1(t)=R_{{\rm e},1}(t)+R_{{\rm i},1}(t)=\|r_{{\rm e},1}(t)\|_{L^2(\ul{\cE})}+\|r_{{\rm i},1}(t)\|_{L^2(\ul{\cI})}$ with 
\[
r_{{\rm e},1} = \sum_{|\alpha|=m-1} |f_5^{(\alpha)}|, \quad
r_{{\rm i},1} = \sum_{|\alpha|=m-1} |f_{{\rm i},8}^{(\alpha)}|;
\]
these terms are the lower order terms appearing in \eqref{QLSWEe3} and \eqref{QLSWEi2} 
and include derivatives of $u$, $v_{\rm i}$, and $\varphi$ of order at most $m-1$. 
It follows from \eqref{eqEcEm} and \eqref{EquivE} that it is sufficient to derive an energy estimate for $\check{E}_{m,\parallel}(t)$ 
in order to control the energy $E_m(t)$; the two norms are actually equivalent; see Remark \ref{re:EFs} below.

%-----------------------------------------------------------
\subsection{Application of energy estimates}\label{sectAppl}
We recall that the good unknowns satisfy the equations derived in Sections \ref{subsect:EGUe}--\ref{subsect:EqGUb}. 
Introduce $R_2(t)=R_{{\rm i},2}(t)=\|r_{{\rm i},2}(t)\|_{L^2(\ul{\cI})}$ with 
\[
r_{{\rm i},2} = \sum_{|\alpha_I|=m} |(f_{{\rm i},2}^{(\alpha_I)},\ldots,f_{{\rm i},7}^{(\alpha)})| 
 + \sum_{|\alpha|=m,\alpha_0\leq m-1} |(f_{{\rm i,r},3}^{(\alpha)},f_{{\rm i,r},6}^{(\alpha)},f_{{\rm i,r},7}^{(\alpha)})|;
\]
these are the lower order terms in \eqref{QLSWEi}, \eqref{Dtphii}, \eqref{QLBC0}, \eqref{QLBC00}, \eqref{QLSWEir}, and \eqref{Dtphiir} 
and also include derivatives of $u,v_{\rm i}$, and $\varphi$ up to at most order $m-1$. 
Next, put $R_3(t)=R_{{\rm e},3}(t)+R_{{\rm i},3}(t)=\|r_{{\rm e},3}(t)\|_{L^2(\ul{\cE})}+\|r_{{\rm i},3}(t)\|_{L^2(\ul{\cI})}$ with 
\begin{align*}
r_{{\rm e},3}
&= \sum_{|\alpha_I|=m} |(f_1^{(\alpha_I)},f_2^{(\alpha_I)},f_3^{(\alpha_I)})| 
 + \sum_{|\alpha|=m,\alpha_0\leq m-1} |(f_{{\rm r},1}^{(\alpha)},f_{{\rm r},2}^{(\alpha)})| + \sum_{|\alpha|\leq m-1}|\dt\bm{\partial}^\alpha u|, \\
r_{{\rm i},3}
&= \sum_{|\alpha_I|=m} |(f_{{\rm i},1}^{(\alpha_I)},\bm{\partial}(f_{{\rm i},2}^{(\alpha_I)},\ldots,f_{{\rm i},5}^{(\alpha_I)}))| 
 + \sum_{|\alpha|=m,\alpha_0\leq m-1} |(f_{{\rm i,r},1}^{(\alpha)},\bm{\partial}f_{{\rm i,r},3}^{(\alpha)})|
 + \sum_{|\alpha|\leq m-1}|\dt\bm{\partial}^\alpha v_{\rm i}|;
\end{align*}
these quantities contain the lower order terms in \eqref{QLSWEe}, \eqref{QLSWEi}, \eqref{QLSWEr}, and \eqref{QLSWEir}, 
that contain derivatives of $u$, $v_{\rm i}$ and $\varphi$ of order at most $m$, and derivatives of lower order terms in  
\eqref{QLSWEi}, \eqref{QLBC0}, \eqref{QLBC00}, and \eqref{QLSWEir} that contain only derivatives of order at most $m-1$. 
In order to control the last term in the energy estimate in Proposition \ref{prop:BEE}, we put 
\begin{align*}
R_4(t)
&= \sum_{|\alpha_I|=m, \alpha_0\leq m-1} \|(\dt\check{\phi}_{\rm i}^{(\alpha_I)},f_{{\rm i},5}^{(\alpha_I)})(t)\|_{L^2(\ul{\cI})}
 \|f_{{\rm i},1}(t)\|_{L^2(\ul{\cI})} \\
&\quad\;
 + \sum_{|\alpha|=m, \alpha_0\leq m-1} \|\dt\check{\phi}_{\rm i,r}^{(\alpha)}(t)\|_{L^2(\ul{\cI})} \|f_{{\rm i,r},1}(t)\|_{L^2(\ul{\cI})},
\end{align*}
where we recall that $\dt\check{\phi}_{\rm i}^{(\alpha_I)},f_{{\rm i},5}^{(\alpha_I)},\dt\check{\phi}_{\rm i,r}^{(\alpha)}$, 
and $f_{{\rm i,r},1}$ are as in \eqref{Dtphii}, \eqref{QLBC00}, \eqref{Dtphiir}, and \eqref{QLSWEir}, respectively. 
These terms include derivatives of $v_{\rm i}$ and $\varphi$ of order $m$. 
Finally, we put 
\[
R_5(t) = \sum_{|\alpha_I|=m} |(\check{\zeta}_{\rm i}^{(\alpha_I)},b^{(\alpha_I)})(t)|_{L^2(\ul{\itGamma})},
\]
where $b^{(\alpha_I)}$ is as in \eqref{gamma eq2}.

To evaluate the energy function $\check{E}_{m,\parallel}(t)$, we use the energy estimate in Proposition \ref{prop:BEE}. 
We apply it with $(\check{u},\check{v}_{\rm i},\check{\psi}_{\rm i})=(\check{u}^{(\alpha_I)},\check{v}_{\rm i}^{(\alpha_I)},\check{\psi}_{\rm i}^{(\alpha_I)})$, 
which satisfy \eqref{QLSWEe}, \eqref{QLSWEi}, and \eqref{QLBC00}, 
and with $(\check{u},\check{v}_{\rm i},\check{\phi}_{\rm i})=(\check{u}_{\rm r}^{(\alpha)},\check{v}_{\rm i,r}^{(\alpha)},\check{\phi}_{\rm i,r}^{(\alpha)})$ 
and $(f_{{\rm i},4},f_{{\rm i},5})=(0,0)$, which satisfy \eqref{QLSWEr} and \eqref{QLSWEir}. 
We also use obvious estimates $\|u(t)\|_{m-1,{\rm e}} \leq \|u(0)\|_{m-1,{\rm e}}+\int_0^t\|\dt u(t')\|_{m-1,{\rm e}}{\rm d}t'$ and 
$\|v_{\rm i}(t)\|_{m-1,{\rm i}} \leq \|v_{\rm i}(0)\|_{m-1,{\rm i}}+\int_0^t\|\dt v_{\rm i}(t')\|_{m-1,{\rm i}}{\rm d}t'$. 
Then, we obtain the desired energy estimate on $\check{E}_{m,\parallel}(t)$:
\begin{equation}\label{NLEE1}
\sup_{0\leq t\leq T}\check{E}_{m,\parallel}(t)
\leq C_0 e^{C_0M_1}\bigl\{ \check{E}_{m,\parallel}(0) + R_2(0) + \int_0^TR_3(t){\rm d}t + \left( \int_0^TR_4(t){\rm d}t \right)^{1/2} \bigr\}.
\end{equation}

In order to prove the a priori estimate of the theorem, we also need an estimate on the parametrization $\gamma$ of the contact line. 
By Proposition \ref{prop:ABR}, we obtain the following additional boundary estimate 
\begin{align*}
\sum_{|\alpha_I|=m} \int_0^T |\check{\zeta}^{(\alpha_I)}(t)|_{L^2(\ul{\itGamma})}^2 {\rm d}t
&\leq C_0\bigl\{ (1+T+M_1)\sup_{0\leq t\leq T}( \check{E}_{m,\parallel}(t) + R_2(t) )^2 
 + \left( \int_0^TR_3(t){\rm d}t \right)^2 \bigr\}.
\end{align*}
On the other hand, $\gamma$ and $\varphi$ satisfy \eqref{EqPhi2} and \eqref{gamma eq2}, so that we have 
\[
|\gamma(t)|_m \leq C_0\left( \sum_{|\alpha_I|=m}|\check{\zeta}^{(\alpha_I)}(t)|_{L^2(\ul{\itGamma})} + R_5(t) + 1\right). 
\]
Therefore, we get 
\begin{align}\label{NLEE2}
\int_0^T|\gamma(t)|_m^2{\rm d}t
&\leq C_0 \bigl\{ (1+T+M_1)\sup_{0\leq t\leq T}( \check{E}_{m,\parallel}(t) + R_2(t))^2 \\
&\quad\;
 + \left( \int_0^TR_3(t){\rm d}t \right)^2 + \int_0^T (R_5(t) + 1)^2{\rm d}t \bigr\}. \nonumber
\end{align}

%-----------------------------------------------------------
\subsection{Estimates for lower order terms}\label{sectestLOT}
We proceed to evaluate the lower order terms $R_j(t)$ for $j=1,\ldots,5$.

\begin{lemma}\label{lem:EstR1}
It holds that $R_1(t) \leq C_0E_m(t)^{1/2}(\|u(t)\|_{m-1,{\rm e}}+\|v_{\rm i}(t)\|_{m-1,{\rm i}})^{1/2}$ for $\leq t\leq T$. 
\end{lemma}

\begin{proof}
We first evaluate $R_{{\rm e},1}(t)=\|r_{{\rm e},1}(t)\|_{L^2(\ul{\cE})}$. 
We write $\|\cdot\|_{L^p}$ instead of $\|\cdot\|_{L^p(\ul{\cE})}$ for simplicity. 
In view of \eqref{linearization 1}, for any multi-index $\alpha$ satisfying $|\alpha|=m-1$ we have 
\begin{align*}
\chi_{\rm b}|{\mathcal C}^1(\mathfrak{d}^\alpha,\partial \varphi)f|
&= \chi_{\rm b}|([\mathfrak{d}^\alpha,\nabla^\varphi]+(\partial^\varphi\mathfrak{d}^\alpha\varphi)^{\rm T}\nabla^\varphi)f| \\
&\leq C_0 \bigl\{ \sum_{|\beta|\leq m-2}|\bm{\partial}^\beta\nabla f| + |\bm{\partial}^{m-2}\partial\varphi||\bm{\partial}\nabla f|
 +  |\bm{\partial}^{m-2}\partial\varphi||\bm{\partial}\partial\varphi||\nabla f| \bigr\}.
\end{align*}
Therefore, by a straightforward calculation, we obtain 
\begin{align*}
r_{{\rm e},1}
&\leq C_0\bigl\{ \sum_{3\leq|\alpha|\leq m-1}|\bm{\partial}^\alpha u| + (|\bm{\partial}^{m-1}\varphi|+1)|\bm{\partial}^2u| \\
&\quad\;
 + (|\bm{\partial}^{m-1}\varphi|+1)(|\bm{\partial}^2\varphi|+1)|\bm{\partial}u| + |\bm{\partial}u||\bm{\partial}^{m-1}u| \\
&\quad\;
 + (|\bm{\partial}^2u|+|\bm{\partial}u|)|\bm{\partial}^{m-2}u| + (|\bm{\partial}^{m-1}\varphi|+1)|\bm{\partial}u|^2 \bigr\}.
\end{align*}
Particularly, in the case $m\geq 4$ we have 
\[
r_{{\rm e},1}
\leq C_0 \bigl\{ \sum_{|\alpha|\leq m-1}|\bm{\partial}^\alpha u| + |\bm{\partial}^{m-1}\varphi||(\bm{\partial}^2u,\bm{\partial}u)| 
 + |\bm{\partial}^2u||\bm{\partial}^{m-2}u| \bigr\},
\]
so that, by the Sobolev embedding theorem $H^1(\ul{\cE}) \hookrightarrow L^4(\ul{\cE})$, 
\begin{align*}
\|r_{{\rm e},1}\|_{L^2}
&\leq C_0 \bigl( \|u\|_{m-1,{\rm e}} + \|\bm{\partial}^{m-1}\varphi\|_{L^4}\|(\bm{\partial}^2u,\bm{\partial}u)\|_{L^4}
 + \|\bm{\partial}^2u\|_{L^4}\|\bm{\partial}^{m-2}u\|_{L^4} \bigr) \\
&\leq C_0(1+\|u\|_{3,{\rm e}})\|u\|_{m-1,{\rm e}} \\
&\leq C_0\|u\|_{m-1,{\rm e}}.
\end{align*}
In the critical case $m=3$, we have 
\[
r_{{\rm e},1}
\leq C_0 \{ (|\bm{\partial}^2\varphi|+1)|\bm{\partial}^2u| + (|\bm{\partial}^2\varphi|^2+1)|\bm{\partial}u|
 + |\bm{\partial}u||\bm{\partial}^2u| + (|\bm{\partial}^2\varphi|+1)|\bm{\partial}u|^2 \},
\]
so that 
\begin{align*}
\|r_{{\rm e},1}\|_{L^2}
&\leq C_0 \bigl( \|u\|_{2,{\rm e}} + \|\bm{\partial}^2\varphi\|_{L^4}\|\bm{\partial}^2u\|_{L^4} + \|\bm{\partial}^2\varphi\|_{L^4}^2\|\bm{\partial}u\|_{L^\infty} \\
&\quad\;
 + \|\bm{\partial}u\|_{L^4}\|\bm{\partial}^2u\|_{L^4} + \|\bm{\partial}^2\varphi\|_{L^4}\|\bm{\partial}u\|_{L^4}\|\bm{\partial}u\|_{L^\infty}
 + \|\bm{\partial}u\|_{L^4}^2 \bigr) \\
&\leq C_0 (1+\|u\|_{2,{\rm e}})(\|u\|_{2,{\rm e}}+\|\bm{\partial}^2u\|_{L^4}+\|\bm{\partial}u\|_{L^\infty}) \\
&\leq C_0 E_3^{1/2} \|u\|_{2,{\rm e}}^{1/2},
\end{align*}
where we used Lemma \ref{lem:EstDuv}. 
In any cases, we obtain $R_{{\rm e},1}(t) \leq C_0 E_3(t)^{1/2} \|u(t)\|_{2,{\rm e}}^{1/2}$.

Since the structure of $r_{{\rm i},1}$ is the same as that of $r_{{\rm e},1}$ by replacing $u$ with $v_{\rm i}$, 
we obtain $R_{{\rm i},1}(t) \leq C_0 E_3(t)^{1/2} \|v_{\rm i}(t)\|_{2,{\rm i}}^{1/2}$ in the same way as above. 
Therefore, we obtain the desired estimate. 
\end{proof}

\begin{remark}\label{re:EFs}
As we noted in \eqref{eqEcEm}, $E_m(t)$ and $\check{E}_m(t)$ are equivalent. 
Therefore, by Lemma \ref{lem:EstR1} we have $R_1(t) \leq C_0\check{E}_m(t)^{1/2}\check{E}_{m,\parallel}(t)^{1/2}$, 
which together with \eqref{EquivE} implies $\check{E}_{m,\parallel}(t) \leq \check{E}_m(t) \leq C_0\check{E}_{m,\parallel}(t)$. 
Therefore, the three energy functions $E_m(t)$, $\check{E}_m(t)$, and $\check{E}_{m,\parallel}(t)$ are all equivalent. 
\end{remark}

\begin{lemma}\label{lem:EstR2}
It holds that $R_2(t) \leq C_0(E_m(t)+1)$ for $\leq t\leq T$. 
\end{lemma}

\begin{proof}
The structure of $r_{{\rm i},2}$ is almost the same as that of $r_{{\rm i},1}$. 
The only difference rises from $f_{{\rm i},4}^{(\alpha_I)}$, which is not necessarily zero even if $E_m(t)=0$. 
However, we easily get $\|f_{{\rm i},4}^{(\alpha_I)}\|_{L^2(\ul{\cI})} \leq C_0$, so that we obtain the desired estimate. 
\end{proof}

\begin{lemma}\label{lem:EstR3}
It holds that $R_3(t) \leq C_0(|\gamma(t)|_m^{3/2}+1)(E_m(t)^2+1)$ for $\leq t\leq T$. 
\end{lemma}

\begin{proof}
We first evaluate $R_{{\rm e},3}(t)=\|r_{{\rm e},3}(t)\|_{L^2(\ul{\cE})}$. 
We write $\|\cdot\|_{L^p}$ instead of $\|\cdot\|_{L^p(\ul{\cE})}$ for simplicity. 
In view of \eqref{linearization 1}, for any multi-index $\alpha$ satisfying $|\alpha|=m$ we have 
\begin{align*}
\chi_{\rm b}|{\mathcal C}^1(\mathfrak{d}^\alpha,\partial \varphi)f|
&= \chi_{\rm b}|([\mathfrak{d}^\alpha,\nabla^\varphi]+(\partial^\varphi\mathfrak{d}^\alpha\varphi)^{\rm T}\nabla^\varphi)f| \\
&\leq C_0 \bigl\{ \sum_{|\beta|\leq m-1}|\bm{\partial}^\beta\nabla f| + |\bm{\partial}^{m-2}\partial\varphi||\bm{\partial}^2\nabla f| \\
&\quad\;
 + \bigl( |\bm{\partial}^{m-1}\partial\varphi| + (|\bm{\partial}^{m-2}\partial\varphi|+1)(|\bm{\partial}\partial\varphi|+1) \bigr)|\bm{\partial}\nabla f| \\
&\quad\;
 + \bigl( |\bm{\partial}^{m-1}\partial\varphi|(|\bm{\partial}\partial\varphi|+1)
 + (|\bm{\partial}^{m-2}\partial\varphi|+1)(|\bm{\partial}\partial\varphi|+1)^2 \bigr) |\nabla f| \bigr\}.
\end{align*}
Therefore, by a straightforward calculation, we obtain 
\begin{align*}
r_{{\rm e},3}
&\leq C_0 \bigl\{ \sum_{m_1+m_2\leq m-1}(|\bm{\partial}^{m_1+1}u|+1)|\bm{\partial}^{m_2+1}u|
 + |\bm{\partial}^{m-1}\varphi|( |\bm{\partial}^3u| + |\bm{\partial}^2u||\bm{\partial}u| ) \\
&\quad\;
 + \bigl( |\bm{\partial}^m\varphi| + (|\bm{\partial}^{m-1}\varphi|+1)(|\bm{\partial}^2\varphi|+1) \bigr)( |\bm{\partial}^2u|+|\bm{\partial}u|^2 ) \\
&\quad\;
 + \bigl( |\bm{\partial}^m\varphi|(|\bm{\partial}^2\varphi|+1) + (|\bm{\partial}^{m-1}\varphi|+1)(|\bm{\partial}^2\varphi|+1)^2 \bigr) |\bm{\partial}u| \bigr\}.
\end{align*}
Particularly, in the case $m\geq 4$ we have 
\begin{align*}
r_{{\rm e},3}
&\leq C_0 \bigl\{ \sum_{|\alpha|\leq m}|\bm{\partial}^\alpha u|
  + |\bm{\partial}^2u||\bm{\partial}^{m-1}u| + (|\bm{\partial}^3u|+|\bm{\partial}^2u|)|\bm{\partial}^{m-2}u| \\
&\quad\;
 + |\bm{\partial}^{m-1}\varphi||\bm{\partial}^3u| + (|\bm{\partial}^m\varphi|+|\bm{\partial}^{m-1}\varphi|)(|\bm{\partial}^2u|+|\bm{\partial}u|) \bigr\},
\end{align*}
so that, by the Sobolev embedding theorem $H^1(\ul{\cE}) \hookrightarrow L^4(\ul{\cE})$, 
\begin{align*}
\|r_{{\rm e},3}\|_{L^2}
&\leq C_0 \bigl( \|u\|_{m,{\rm e}} + \|\bm{\partial}^2u\|_{L^4} \|\bm{\partial}^{m-1}u\|_{L^4}
 + \|(\bm{\partial}^3u,\bm{\partial}^2u)\|_{L^4} \|\bm{\partial}^{m-2}u\|_{L^4} \\
&\quad\;
 + \|\bm{\partial}^{m-1}\varphi\|_{L^4} \|\bm{\partial}^3u\|_{L^4} 
 + \|(\bm{\partial}^m\varphi,\bm{\partial}^{m-1}\varphi)\|_{L^4} \|(\bm{\partial}^2u,\bm{\partial}u)\|_{L^4} \bigr) \\
&\leq C_0 \{ (\|u\|_{m-1,{\rm e}}+1)\|u\|_{m,{\rm e}} + |\gamma|_m \|u\|_{m-1,{\rm e}} \} \\
&\leq C_0(|\gamma|_m+1)E_m,
\end{align*}
where we used Lemma \ref{lem:EstDuv}. 
In the critical case $m=3$, we have 
\begin{align*}
r_{{\rm e},3}
&\leq C_0 \bigl\{ (|\bm{\partial}u|+1)(|\bm{\partial}^3u|+|\bm{\partial}^2u|+|\bm{\partial}u|) + |\bm{\partial}^2u|^2
 + |\bm{\partial}^2\varphi|( |\bm{\partial}^3u| + |\bm{\partial}^2u||\bm{\partial}u| ) \\
&\quad\;
 + (|\bm{\partial}^3\varphi|+|\bm{\partial}^2\varphi|^2+|\bm{\partial}^2\varphi|)(|\bm{\partial}^2u|+|\bm{\partial}u|^2)
 + (|\bm{\partial}^3\varphi||\bm{\partial}^2\varphi|+|\bm{\partial}^3\varphi|+|\bm{\partial}^2\varphi|^3)|\bm{\partial}u| \bigr\},
\end{align*}
so that 
\begin{align*}
\|r_{{\rm e},3}\|_{L^2}
&\leq C_0 \bigl( (\|\bm{\partial}u\|_{L^\infty}+1)\|u\|_{3,{\rm e}} + \|\bm{\partial}^2u\|_{L^4}^2
 + \|\bm{\partial}^2\varphi\|_{L^\infty}( \|\bm{\partial}^3u\|_{L^2} + \|\bm{\partial}^2u\|_{L^2}\|\bm{\partial}u\|_{L^\infty} ) \\
&\quad\;
 + ( \|(\bm{\partial}^3\varphi,\bm{\partial}^2\varphi)\|_{L^4} + \|\bm{\partial}^2\varphi\|_{L^8}^2 )
  ( \|\bm{\partial}^2u\|_{L^4} + \|\bm{\partial}u\|_{L^\infty}^{3/2}\|\bm{\partial}u\|_{L^2}^{1/2} ) \\
&\quad\;
 + ( \|\bm{\partial}^3\varphi\|_{L^4}\|\bm{\partial}^2\varphi\|_{L^4} + \|\bm{\partial}^3\varphi\|_{L^2} + \|\bm{\partial}^2\varphi\|_{L^6}^3 )
  \|\bm{\partial}u\|_{L^\infty} \\
&\leq C_0 \bigl( (E_3+1)(|\gamma|_3+1)E_3 + E_3^2 \\
&\quad\;
 + |\gamma|_3^{1/2}( (|\gamma|_3+1)E_3 + E_3^2) + (|\gamma|_3+1)(E_3 + E_3^2) \bigr) \\
&\leq C_0(|\gamma|_3^{3/2}+1)(E_3+1)E_3,
\end{align*}
where we used Lemma \ref{lem:EstDuv} and $\|\bm{\partial}^2\varphi\|_{L^p} \leq C_0|\gamma|_3^{1/2-2/p}$ for $p\in[4,\infty]$, 
which comes from Lemma \ref{lem:EstDM}. 
In any cases, we obtain $R_{{\rm e},3}(t) \leq C_0(|\gamma(t)|_m^{3/2}+1)(E_m(t)+1)E_m(t)$.

Since the structure of $r_{{\rm i},3}$ is the same as that of $r_{{\rm e},3}$ by replacing $u$ with $v_{\rm i}$, 
we obtain $R_{{\rm i},3}(t) \leq C_0(|\gamma(t)|_m^{3/2}+1)(E_m(t)^2+1)$ in the same way as above. 
We note that this is the place where we need not only to $Z_{\rm w}\in C^m(\overline{\ul{\cI}_0})$ 
but also to $Z_{\rm w}\in C^{m+1}(\overline{\ul{\cI}_0})$ in Assumption \ref{ass:BSB}. 
Therefore, we obtain the desired estimate. 
\end{proof}

\begin{lemma}\label{lem:EstR4}
It holds that $R_4(t) \leq C_0(|\gamma(t)|_m^{3/2}+1)(E_m(t)^3+1)$ for $\leq t\leq T$. 
\end{lemma}

\begin{proof}
By \eqref{Dtphii} and \eqref{Dtphiir}, for any multi-index $\alpha$ satisfying $|\alpha|=m$ and $\alpha_0\leq m-1$ we have 
$|\dt\check{\phi}_{\rm i}^{(\alpha)}| + |\dt\check{\phi}_{\rm i,r}^{(\alpha)}|
\leq C_0\sum_{|\beta|=m} ( |\check{v}_{\rm i}^{(\beta)}| + |\check{v}_{\rm i,r}^{(\beta)}| ) + r_{{\rm i},2}$,
which implies 
$\|\dt\check{\phi}_{\rm i}^{(\alpha)}\|_{L^2(\ul{\cI})} + \|\dt\check{\phi}_{\rm i,r}^{(\alpha)}\|_{L^2(\ul{\cI})}\leq C_0( \check{E}_m + R_2 )$. 
Therefore, we obtain $R_4 \leq C_0 (\check{E}_m + R_2 )R_3$, which together with Lemmas \ref{lem:EstR2} and \ref{lem:EstR3} yields the desired estimate. 
\end{proof}

\begin{lemma}\label{lem:EstR5}
It holds that $R_5(t) \leq C_0$ for $0\leq t\leq T$. 
\end{lemma}

\begin{proof}
For any multi-index $\alpha_I$ satisfying $|\alpha_I|=m$, we see easily that 
\[
\begin{cases}
 |\zeta_{\rm i}^{(\alpha_I)}| \leq C_0( |\bm{\partial}^{m-1}\varphi|+1 ), \\
 |b^{(\alpha_I)}| \leq C_0 (\sum_{j+k=m-1}|\dt^j\partial_s^k\gamma| +1 ),
\end{cases}
\]
which together with Lemma \ref{lem:EstDM} implies $R_5 \leq C_0(|\gamma|_{m-1}+1)\leq C_0$. 
\end{proof}

%-----------------------------------------------------------
\subsection{The transversality condition}\label{secttransv}
When we show the estimates in \eqref{Hyp1} for some time $T>0$, we usually evaluate the time derivative of each quantities. 
Such a standard procedure can work in the case $m\geq4$. 
However, in the critical case $m=3$, this procedure does not work, especially, for the transversality condition 
$|\ul{N}\cdot(\nabla\zeta-\nabla\zeta_{\rm i})| \geq c_0$ on $\ul{\itGamma}$, 
because we do not have the boundedness of the second order derivatives $\dt\nabla\zeta$ and $\dt\nabla\zeta_{\rm i}$. 
To bypass this difficulty, we prepare the following lemmas.

\begin{lemma}\label{lem:Embedding1}
Let $\Omega=\ul{\cE}$ or $\ul{\cI}$. 
There exists a positive constant $C$ such that for any $f\in H^2((0,T)\times\Omega)$ 
and any $t\in(0,T)$ we have 
\[
\|f(t,\cdot)-f(0,\cdot)\|_{L^\infty(\Omega)}
 \leq C\sqrt{t}\bigl( \|f\|_{H^2((0,T)\times\Omega)} + \|f(0,\cdot)\|_{H^{3/2}(\Omega)} + \|\dt f(0,\cdot)\|_{H^{1/2}(\Omega)} \bigr),
\]
where the constant $C$ does not depend on $T$. 
\end{lemma}

\begin{proof}
By using an appropriate extension operator from $H^s(\Omega)$ to $H^s(\R^2)$, it is sufficient to show the estimate in the case $\Omega=\R^2$. 
We first consider the case $f(0,x)\equiv \dt f(0,x)\equiv0$. 
Then, by extending $f(t,x)$ for $t<0$ by zero and denoting it by $f_0$, 
we have $f_0\in H^2((-\infty,T)\times\R^2)$ and $\|f_0\|_{H^2((-\infty,T)\times\R^2)}=\|f\|_{H^2((0,T)\times\R^2)}$. 
We further extend $f_0(t,x)$ for $t>T$ smoothly by a standard procedure. 
Then, we have $F_0\in H^2(\R\times\R^2)$ and 
$\|F_0\|_{H^2(\R\times\R^2)} \leq C\|f_0\|_{H^2((-\infty,T)\times\R^2)}$. 
Now, by the Sobolev embedding theorem $H^2(\R^3) \hookrightarrow C^{1/2}(\R^3)$ we have 
\begin{align*}
|f(t,x)-f(0,x)|
&= |F_0(t,x)-F_0(0,x)| \\
&\leq C\sqrt{t}\|F_0\|_{H^2(\R\times\R^2)} \\
&\leq C\sqrt{t}\|f\|_{H^2((0,T)\times\R^2)}.
\end{align*}

We then consider the general case. 
By the trace theorem, we have $f(0,\cdot) \in H^{3/2}(\R^2)$ and $\dt f(0,\cdot) \in H^{1/2}(\R^2)$. 
Therefore, there exists $F\in H^2(\R\times\R^2)$ such that 
\[
\begin{cases}
 F(0,x)=f(0,x), \quad \dt F(0,x)=\dt f(0,x), \\
 \|F\|_{H^2(\R\times\R^2)}
  \leq C( \|f(0,\cdot)\|_{H^{3/2}(\R^2)} + \|\dt f(0,\cdot)\|_{H^{1/2}(\R^2)} \bigr).
\end{cases}
\]
Putting $f_1:=f-F$, we have $f_1(0,x)\equiv \dt f_1(0,x)\equiv0$. 
Therefore, we see that 
\begin{align*}
|f(t,x)-f(0,x)|
&\leq |f_1(t,x)-f_1(t,x)| + |F(t,x)-F(0,x)| \\
&\leq C\sqrt{t} \bigl( \|f_1\|_{H^2((0,T)\times\R^2)} + \|F\|_{H^2(\R\times\R^2)} \bigr) \\
&\leq C\sqrt{t} \bigl( \|f\|_{H^2((0,T)\times\R^2)} + \|F\|_{H^2(\R\times\R^2)} \bigr),
\end{align*}
which gives the desired estimate.
\end{proof}

\begin{lemma}\label{lem:Embedding2}
Let $\Omega=\ul{\cE}$ or $\ul{\cI}$. 
There exists a positive constant $C$ such that for any $f\in H^2((0,T)\times\Omega)$ 
and any $t\in(0,T)$ we have 
\[
\|f(t,\cdot)-f(0,\cdot)\|_{L^\infty(\Omega)} \leq C \|f\|_{H^2((0,T)\times\Omega)},
\]
where the constant $C$ does not depend on $T$. 
\end{lemma}

\begin{proof}
By a standard trace theorem, we have $\|f(0,\cdot)\|_{H^{1/2}(\Omega)} \lesssim \|f\|_{H^1((0,1)\times\Omega)}$. 
Therefore, by a simple scaling $t\to Tt$, we obtain $\|f(0,\cdot)\|_{H^{1/2}(\Omega)} \lesssim (\frac{1}{\sqrt{T}}+\sqrt{T})\|f\|_{H^1((0,T)\times\Omega)}$, 
so that $\|f(0,\cdot)\|_{H^{3/2}(\Omega)} + \|\dt f(0,\cdot)\|_{H^{1/2}(\Omega)} \lesssim (\frac{1}{\sqrt{t}}+\sqrt{t})\|f\|_{H^2((0,T)\times\Omega)}$ 
for $0<t\leq T$. 
This and Lemma \ref{lem:Embedding1} imply $|f(t,x)-f(0,x)|\lesssim (1+t) \|f\|_{H^2((0,T)\times\Omega)}$. 
Therefore, the desired estimate holds in the case $0\leq t\leq1$. 
On the other hand, in the case $t>1$, by the Sobolev embedding theorem we see that 
$\|f(0,\cdot)\|_{L^\infty(\Omega)} \lesssim \|f\|_{L^\infty((0,1)\times\Omega)} \lesssim \|f\|_{H^2((0,1)\times\Omega)} \lesssim \|f\|_{H^2((0,t)\times\Omega)}$. 
Similarly, we have also $\|f(t,\cdot)\|_{L^\infty(\Omega)}\lesssim \|f\|_{H^2((0,t)\times\Omega)}$. 
Therefore, we obtain the desired estimate. 
\end{proof}

Now, we evaluate $|\ul{N}\cdot(\nabla\zeta-\nabla\zeta_{\rm i})(t,x)|$ as follows. 
For any $x\in\ul{\itGamma}$, by Lemma \ref{lem:Embedding2} we see that 
\begin{align*}
&|\ul{N}\cdot(\nabla\zeta-\nabla\zeta_{\rm i})(t,x)-\ul{N}\cdot(\nabla\zeta-\nabla\zeta_{\rm i})(0,x)| \\
&\leq \|\nabla\zeta(t,\cdot)-\nabla\zeta(0,\cdot)\|_{L^\infty(\ul{\cE})} + \|\nabla\zeta(t,\cdot)-\nabla\zeta(0,\cdot)\|_{L^\infty(\ul{\cI})} \\
&\lesssim \|\nabla\zeta\|_{H^2((0,T)\times\ul{\cE})} + \|\nabla\zeta_{\rm i}\|_{H^2((0,T)\times\ul{\cI})}.
\end{align*}
Here, by Lemma \ref{lem:EstDuv} (iv) we have 
\begin{align*}
\|\nabla\zeta\|_{H^2((0,T)\times\ul{\cE})}^2
&\leq \sum_{|\alpha|\leq m-1}\int_0^T\|\nabla\bm{\partial}^\alpha\zeta(t)\|_{L^2(\ul{\cE})}^2 {\rm d}t \\
&\leq C_0\int_0^T(|\gamma(t)|_m+1)E_m(t)^2 {\rm d}t \\
&\leq C_0M_2^2((M_3T)^{1/2}+T).
\end{align*}
In view of $|\nabla\bm{\partial}^\alpha\zeta_{\rm i}| \leq C_0( |\partial\bm{\partial}^{m-1}\varphi|+|\bm{\partial}^{m-1}\varphi|+1 )$ and Lemma \ref{lem:EstDM}, 
we have also $\|\nabla\zeta_{\rm i}\|_{H^2((0,T)\times\ul{\cI})}^2 \leq C_0((M_3T)^{1/2}+T)$. 
Therefore, we obtain 
\begin{equation}\label{TransCond2}
|\ul{N}\cdot(\nabla\zeta-\nabla\zeta_{\rm i})(t,x)-\ul{N}\cdot(\nabla\zeta-\nabla\zeta_{\rm i})(0,x)|
\leq C_0(M_2+1)((M_3T)^{1/2}+T)^{1/2},
\end{equation}
for any $(t,x)\in[0,T]\times\ul{\itGamma}$.

%-----------------------------------------------------------
\subsection{Completion of the a priori estimates}\label{sectAPE}
By Lemma \ref{lem:EstR3}, our assumption \eqref{Hyp2}, and H\"older's inequality, we see that 
\begin{align*}
\int_0^TR_3(t){\rm d}t
&\leq C_0(M_2^2+1)\int_0^T(|\gamma(t)|_m^{3/2}+1){\rm d}t \\
&\leq C_0(M_2^2+1)\{ \left(\int_0^T|\gamma(t)|_m^2{\rm d}t\right)^{3/4}T^{1/4} + T \} \\
&\leq C_0(M_2^2+1)( (M_3^3T)^{1/4}+T ).
\end{align*}
Similarly, by Lemma \ref{lem:EstR4} we have $\int_0^TR_4(t){\rm d}t \leq C_0(M_2^3+1)( (M_3^3T)^{1/4}+T )$. 
Therefore, by \eqref{Hyp2}, \eqref{NLEE1}, Lemma \ref{lem:EstR2}, and Remark \ref{re:EFs}, we obtain 
\begin{align}\label{APE1}
\sup_{0\leq t\leq T}E_m(t)
&\leq C_0 e^{C_0M_1}\bigl\{ E_m(0) + 1 + \int_0^TR_3(t){\rm d}t + \left( \int_0^TR_4(t){\rm d}t \right)^{1/2} \bigr\} \\
&\leq C_0 e^{C_0M_1}\bigl\{ 1 + (M_2^3+1)((M_3^3T)^{1/4}+T ) \}. \nonumber
\end{align}
Similarly, we get 
\begin{align}\label{APE2}
\int_0^T|\gamma(t)|_m^2{\rm d}t
&\leq C_0 \bigl\{ (1+T+M_1)\sup_{0\leq t\leq T}(E_m(t)+1)^2 + \left( \int_0^TR_3(t){\rm d}t \right)^2 + T \bigr\} \\
&\leq C_0 \bigl\{ (M_1+1)(M_2+1)^2 + (M_2^2+1)^2 ( (M_3^3T)^{1/4}+T )^2 \bigr\}. \nonumber
\end{align}
As for the first condition in \eqref{Hyp2}, by Lemmas \ref{lem:EstDM} and \ref{lem:EstDuv} we see that 
\begin{align*}
\|\bm{\partial}(h,w)\|_{L^\infty(\ul{\cE})} + \|\bm{\partial} h_{\rm i}\|_{L^\infty(\ul{\cI})} + \|\bm{\partial}\partial\varphi\|_{L^\infty(\R^2)}
&\leq C_0 \bigl( \|\bm{\partial}u\|_{L^\infty(\ul{\cE})} + \|(\bm{\partial}^2\varphi,\bm{\partial}^2\varphi)\|_{L^\infty(\R^2)} \bigr) \\
&\leq C_0 ( E_m + |\gamma|_m^{1/2} + 1 ) \\
&\leq C_0(M_2+1)(|\gamma|_m^{1/2}+1),
\end{align*}
so that 
\begin{align}\label{APE3}
&\|\bm{\partial}(h,w)\|_{L^1(0,T;L^\infty(\ul{\cE}))} + \|\bm{\partial} h_{\rm i}\|_{L^1(0,T;L^\infty(\ul{\cI}))}
   + \|\bm{\partial}\partial\varphi\|_{L^1(0,T;L^\infty(\R^2))} \\
&\leq C_0(M_2+1)((M_3T^3)^{1/4}+T). \nonumber
\end{align}
In view of \eqref{APE1}--\eqref{APE3}, we first choose $M_1$ as an arbitrary positive constant, for example, $M_1=1$. 
Then, we choose $M_2=2C_0e^{C_0M_1}$ and $M_3=2C_0(M_1+1)(M_2+1)^2$. 
If we take the time $T>0$ appropriately small depending only on these constants $M_1,M_2,M_3$, and $C_0$, 
then we see that the solution satisfies in fact \eqref{Hyp2}.

It remains to show that the solution satisfies also \eqref{Hyp1}. 
Since $\overline{\cI(0)} \subset \ul{\cI}_0$ and $\ul{\cI}_0$ is open, there exists a constant $\delta_1>0$ such that for any $X\in\R^2$, 
$\operatorname{dist}(X,\cI(0))\leq\delta_1$ implies $X\in\ul{\cI}_0$. 
We note also that $\cI(0)=\{ \varphi(0,x) \,|\, x\in\ul{\cI} \}$. 
Therefore, in order to prove the first condition in \eqref{Hyp1}, it is sufficient to show that for any $x\in\ul{\cI}$, 
we have $|\varphi(t,x)-\varphi(0,x)|\leq\delta_1$. 
By Lemmas \ref{lem:EstDM} and \ref{lem:EstDuv} and the Sobolev embedding theorem $H^1(\mathbb{T}_L) \hookrightarrow L^\infty(\mathbb{T}_L)$, we have 
\[
\begin{cases}
 \|\dt(\gr h-|w|^2)(t)\|_{L^\infty(\ul{\cE})} \leq C_0(M_2+1)(|\gamma|_m^{1/2}+1), \\
 |\dt\gamma(t)|_{L^\infty(\mathbb{T}_L)} \leq C_0, \\
 \|\dt u(t)\|_{m-1,{\rm e}}+\|\dt v_{\rm i}(t)\|_{m-1,{\rm i}}+|\dt\gamma(t)|_{m-1} \leq C_0(M_2+1)(|\gamma(t)|_m+1),
\end{cases}
\]
which together with \eqref{TransCond2} and the assumptions in \eqref{EstIni} on the initial data implies 
\[
\begin{cases}
 \sup_{x\in\ul{\cI}}|\varphi(t,x)-\varphi(0,x)| \leq C_0T, \\
 \inf_{(t,x)\in(0,T)\times\ul{\cE}}(\gr h(t,x)-|w(t,x)|^2) \geq 2c_0 - C_0(M_2+1)((M_3T^3)^{1/4}+T), \\
  \inf_{(t,x)\in(0,T)\times\ul{\itGamma}}| \ul{N}\cdot(\nabla\zeta-\nabla\zeta_{\rm i})(t,x)| \geq 2c_0 - C_0(M_2+1)((M_3T)^{1/2}+T)^{1/2}, \\
  \sup_{0<t<T}|\gamma(t)|_{L^\infty(\mathbb{T}_L)} \leq \eta_0^{\rm in}r_0 + C_0T, \\
  \sup_{0<t<T}(\|u(t)\|_{m-1,{\rm e}} + \|v_{\rm i}(t)\|_{m-1,{\rm i}} + |\gamma(t)|_{m-1}) \leq M_0 + C_0(M_2+1)((M_3T)^{1/2}+T).
\end{cases}
\]
Therefore, if we take $T>0$ further small depending on the constants $\delta_1, c_0,\eta_0^{\rm in},\eta_0,C_0,M_0,M_2$, and $M_3$, 
then we see that the solution satisfies in fact \eqref{Hyp1}, too. 
The proof of Theorem \ref{th:main} is complete. 
\hfill$\Box$

%----------------------------------------------------------------------------------------------------------------------
%----------------------------------------------------------------------------------------------------------------------

\end{document}